\documentclass{article}

\usepackage{amsmath,amsfonts,amsthm,amssymb}
\usepackage[utf8]{inputenc}
\usepackage[english]{babel}
\usepackage{csquotes}
\usepackage{subcaption}
\usepackage{acronym}
\usepackage{enumerate}
\usepackage{comment}
\usepackage{tikz}
\usetikzlibrary{decorations.markings}

\usepackage{thmtools}
\usepackage[colorlinks=true,citecolor=blue]{hyperref}
\usepackage[capitalise,sort]{cleveref}
\newcommand{\crefnopar}[1]{\namecref{#1}~\ref{#1}} 

\makeatletter 
\AtBeginDocument
 {
   \def\ltx@label#1{\cref@label{#1}}
   \def\label@in@display@noarg#1{\cref@old@label@in@display{#1}}
\def\label@in@mmeasure@noarg#1{%
    \begingroup%
      \measuring@false%
      \cref@old@label@in@display{#1}
    \endgroup}%
 } %
\makeatother

\usepackage{macros}

\usepackage{biblatex}
\usepackage[left=2.5cm,right=2.5cm,top=2.5cm,bottom=2.5cm]{geometry}
\theoremstyle{definition}
\newtheorem{theorem}{Theorem}
\newtheorem{proposition}[theorem]{Proposition}
\newtheorem{lemma}[theorem]{Lemma}
\newtheorem{corollary}[theorem]{Corollary}
\newtheorem{definition}[theorem]{Definition}
\newtheorem{assumption}{Assumption}

\theoremstyle{remark}

\newtheorem{example}{Example}[section]

\usepackage{fancyhdr}
\usepackage{authblk} 

\begin{document}

\date{September 18, 2026}
\title{\textsc{Fermat Distance-to-Measure:\\
a Robust Measure-Based Metric}}

\author[1]{Jérôme Taupin\thanks{jerome.taupin@inria.fr}}
\author[1]{Frédéric Chazal\thanks{frederic.chazal@inria.fr}}
\affil[1]{Inria Saclay \& Institut de Mathématiques d'Orsay\thanks{307 Rue Michel Magat, 91400 Orsay, France}}

\maketitle

\begin{abstract}
Given a probability measure with density, Fermat distances and density-driven metrics are conformal transformations of the Euclidean metric that shrink distances in high density areas and enlarge distances in low density areas.
Although they have been widely studied and have shown to be useful in various data science tasks, they are limited to measures with density (with respect to Lebesgue measure, or volume form on manifold).
In this paper, by replacing the density with the \emph{distance-to-measure}, we introduce a new metric, the \emph{Fermat distance-to-measure}, defined for any probability measure in $\RR^d$.
We derive strong stability properties for the Fermat distance-to-measure with respect to the measure and propose an estimator from random sampling of the same measure, featuring an explicit bound on its convergence rate.
\end{abstract}

\section{Introduction}

Fermat distances belong to the class of density-driven metrics that have been established as a useful tool for statistics and machine learning tasks, such as semi-supervised learning or clustering, when the Euclidean metric does not sufficiently exploit the geometry of the distribution underlying the data. 
Given a measure~$\mu$ with density $f$ with respect to the Lebesgue measure in a Euclidean space or the volume form in a more general Riemannian manifold, the Fermat distance between two points $x$ and $y$ is defined as
\begin{align}
    \label{eq:fermat}
    \inf_\gamma \int_\gamma f^{-\beta}~,
\end{align}
where the infimum is taken over all rectifiable paths $\gamma$ from $x$ to $y$, and $\beta>0$ is a parameter.
This quantity appears naturally as the limit of a weighted graph metric built on top of random samples of~$\mu$, where the weight of an edge is a power of its Euclidean length, see~\cite{hwangShortestPathRandom2016,groismanNonhomogeneousEuclideanFirstpassage2022}.
This discrete metric, which will be referred to as the sample Fermat distance, has long been studied, with applications ranging from clustering tasks \cite{sapienzaWeightedGeodesicDistance2018,mckenziePowerWeightedShortest2019}, with applications in biology \cite{carpioCancerFingerprintsTopological2022,manousidakiClusteringVisualizationSinglecell2024} and image segmentation \cite{garciatrillosFermatDistancesMetric2024}, to manifold learning \cite{costaManifoldLearningGeodesic2003} and persistent homology in topological data analysis \cite{fernandezIntrinsicPersistentHomology2023}.
Although convergence guarantees for this estimator have been established in~\cite{hwangShortestPathRandom2016,groismanNonhomogeneousEuclideanFirstpassage2022} as the sample size increases, quantitative bounds on the other hand come at the price of strict smoothness conditions on the density $f$, see~\cite{fernandezIntrinsicPersistentHomology2023, garciatrillosFermatDistancesMetric2024}. Moreover, the existence of the continuous Fermat distance requires the measure $\mu$ to have a well-defined density~$f$, restricting the convergence guarantees of the sample Fermat distance to such measures.

The aim of this paper is to introduce a new family of measure-based metrics---called \ac{fdtm}---that overcomes the above limitations and holds strong stability properties with respect to perturbations of the measure. 
We rely on the notion of \ac{dtm} functions introduced in~\cite{chazalGeometricInferenceMeasures2011}. Given any probability measure on $\RR^d$, the \ac{dtm} is a distance-like function that behaves similarly to a negative power of the density, although the existence of a density is not required to define it.
Intuitively, the \ac{dtm} function associated to a measure $\mu$ at a point $x \in \RR^d$ represents how far $x$ is from a given fraction $m$ of the total mass of $\mu$, and our approach consists in replacing $f^{-\beta}$ in \cref{eq:fermat} with the \ac{dtm} to define the \ac{fdtm}.
When the density $f$ is Lipschitz, an adequate normalization of a negative power of the \ac{dtm} converges to $f$ as the mass parameter $m$ goes to $0$, which suggests that this substitution yields a similar metric.
Moreover, thanks to regularity and stability properties of the \ac{dtm}, the \ac{fdtm} presents several benefits compared to the classical Fermat distance.
First, it is well-defined for any measure $\mu$, regardless of the existence of a density.
Second, we provide quantitative stability results for the \ac{fdtm} \ac{wrt} perturbations of the measure in a Wasserstein metric, and to perturbations of its support in Hausdorff distance.
Last, we propose a natural estimator of the \ac{fdtm} from random samples of the measure and, building on the stability properties of the \ac{fdtm}, provide explicit convergence rates for this estimator that do not depend on the ambient dimension but on some mild regularity properties of the sampled measure.

The formal definition of \ac{fdtm} and statements of the above results are provided in \cref{sec:background_results}. 
\Cref{sec:geodesic&stability} is devoted to the study of geodesics and stability properties of the \ac{fdtm}. 
The estimation of the \ac{fdtm} from random samples is studied in \cref{sec:estimation} where an effective estimator is provided together with an upper bound on its convergence rate.
A similar lower bound is also provided on the convergence rate of any estimator in the worst case \ac{wrt} the measure.
Finally, we give some numerical illustrations of our results in \cref{sec:simulations}.
Although most of the proofs of the paper rely on simple ideas, they often require technical arguments and intermediate results that are detailed in the appendices.

\section{Definitions and Overview of the Results}
\label{sec:background_results}

The setting of this paper is the space $\RR^d$ endowed with the Euclidean norm $\|\cdot\|$.
A list of notations used throughout this paper is given in \cref{tab:notations}.
In this section, we provide the definition of the \ac{dtm} and \ac{fdtm} then give an overview of the results we obtain.

\subsection{The Fermat Distance-to-Measure Metric}

Given a probability measure $\mu$ over $\RR^d$ and two real parameters $m\in(0,1]$ and $p\ge1$, the \ac{dtm} is defined in~\cite{chazalGeometricInferenceMeasures2011} for all $x\in\RR^d$ as
\begin{align*}
    d_{\mu,m,p}(x)
    \eqdef \Lp \frac{1}{m} \int_0^m \delta_{\mu,u}(x)^p\, \d u \Rp^{1/p}~,
\end{align*}
denoted $d_\mu(x)$ for short, where
\begin{align*}
    \delta_{\mu,u}(x)
    \eqdef \inf \bLa r>0 \,:\, \mu\bLp\cB(x,r)\bRp > u \bRa
\end{align*}
is the pseudo-\ac{dtm}, that is the minimal radius of a ball centered at $x$ that contains a mass $u$ \ac{wrt}~$\mu$.
In this paper, the parameters $m$ and $p$ are fixed constants hence we generally omit the notational dependency of the \ac{dtm} \ac{wrt} these parameters to avoid heavy notations.
Notice that $\delta_{\mu,u}$ is a finite function everywhere for all $u<1$. Therefore, when $m<1$, $d_\mu$ is well-defined and finite over the whole space $\RR^d$. 
Moreover, it is shown in~\cite{chazalGeometricInferenceMeasures2011} that $d_\mu$ is $1$-Lipschitz, \ie, 
\begin{align}
    \label{eq:dtm_lip}
    |d_\mu(x) - d_\mu(y)|
    \le \|x - y\|
    \qquad\textrm{for any}\qquad
    x,y \in \RR^d~,
\end{align}
and stable with respect to $\mu$, \ie,
\begin{align}
    \label{eq:dtm_stab_wasserstein}
    \|d_\mu - d_\nu\|_\infty
    \le \frac{W_p(\mu,\nu)}{m^{1/p}}~,
\end{align}
where $W_p$ denotes the $p$-Wasserstein distance defined by
\begin{align*}
    W_p(\mu,\nu)
    = \inf_{\pi\in\Pi(\mu,\nu)} \Lp \int_{\RR^d \times \RR^d} \|x-y\|^p \pi(\d x,\d y) \Rp^{\frac1p}~,
\end{align*}
and with $\Pi(\mu,\nu)$ the set of probability measures on $\RR^d \times \RR^d$ that have marginals $\mu$ and $\nu$.

The function $\delta_{\mu,u}$ generalizes the distance function to the support $\supp(\mu)$ of $\mu$---the latter being exactly $\delta_{\mu,0}$---and reflects how far $x$ is from a mass $u$ of the measure $\mu$. The \ac{dtm}, obtained by taking the $p$-Hölder mean of $\delta_{\mu,u}$ over $u\in[0,m]$, intuitively behaves like the inverse of a density: $d_\mu(x)$ is large when $x$ belongs to a low-density area or when $x$ is far away from the support of $\mu$, and $d_\mu(x)$ is small when $x$ is in a high-density area.
Moreover, the parameter $m$ can be thought of as the resolution of $d_\mu$. Setting $m$ to be smaller results in the \ac{dtm} being more sensitive to local variations of the measure, while larger values of $m$ result in a smoother function. In fact, when $\mu$ has a density, we shall see that by letting $m$ go to $0$ the density may be retrieved from the \ac{dtm}.
These properties motivate the introduction of the \ac{fdtm} between two points as the integral of a power of the \ac{dtm}, minimized over paths between $x$ and $y$.

\begin{definition}[Fermat distance-to-measure (\ac{fdtm})]
Let $\mu$ be a probability measure over $\RR^d$, $\F$ be a closed subset of $\RR^d$, and $\exponent>0$ be a power parameter.
Given a rectifiable path $\gamma:[0,1]\to\RR^d$, we call \ac{fdtm} length of the path $\gamma$ the quantity
\begin{align*}
    D_{\mu,m,p,\exponent}(\gamma)
    \eqdef \int_\gamma d_{\mu,m,p}^\exponent
    = \int_0^1 d_{\mu,m,p}(\gamma(t))^\exponent \|\dot{\gamma}(t)\| \d t~,
\end{align*}
denoted $D_\mu(\gamma)$ for short.
Given $x,y \in \RR^d$, we call \ac{fdtm} between $x$ and $y$ with measure $\mu$ and domain $\F$ the quantity
\begin{align*}
    D_{\mu,\F,m,p,\exponent}(x,y)
    \eqdef \inf_{\gamma \in \Gamma_\F(x,y)} D_{\mu,m,p,\exponent}(\gamma)~,
\end{align*}
denoted $D_{\mu,\F}(x,y)$ for short.
In the above definition $\Gamma_\F(x,y)$ is the set of rectifiable paths $\gamma : [0,1] \to \RR^d$ such that $\gamma(0)=x$, $\gamma(1)=y$, and with the additional constraint that $\gamma$ may leave the domain $\F$ only by drawing straight lines.
Formally, for all open interval $I\subset[0,1]$ such that $\gamma(I) \subset \RR^d\setminus \F$, there exists a unit vector $u\in\RR^d$ such that for all $t\in I$,
\begin{align*}
    \dot\gamma(t)
    = \|\dot\gamma(t)\| \cdot u~.
\end{align*}
\end{definition}

As for $m$ and $p$, $\exponent$ is a fixed parameter, and we generally omit the notational dependency \ac{wrt}~these three parameters.
The above definition involves important notions to discuss such as the domain $\F$. Let us first introduce a few notations.
The rectifiable paths $\gamma$ belonging to $\Gamma_\F(x,y)$ are called \emph{admissible paths}.
Recall that a path $\gamma$ is rectifiable if $\|\dot\gamma\|$ is defined almost everywhere on $[0,1]$ and integrable. Its Euclidean length is then denoted by
\begin{align*}
    |\gamma|
    \eqdef \int_0^1 \|\dot\gamma(t)\| \d t
    < \pinfty~.
\end{align*}
The notion of integration along a rectifiable path is well-defined. Moreover, such paths may always be re-parameterized as a Lipschitz path.
In the sequel, paths $\gamma$ will often be implicitly assumed to be parameterized with constant Euclidean velocity, \ie, such that $\|\dot\gamma(t)\| = |\gamma|$ almost everywhere.
Some proofs shall instead consider an arc length parameterization, where $\gamma$ is defined over $[0,|\gamma|]$ and such that $\|\dot\gamma(t)\| = 1$ almost everywhere.
The most basic path in $\Gamma_\F(x,y)$ is the segment $t\mapsto (1-t)x + ty$, which will be written as $[x,y]$ for short. Its \ac{fdtm} length is given by
\begin{align*}
    D_\mu([x,y])
    = \|x-y\| \int_0^1 d_\mu\bLp(1-t)x + ty\bRp^\exponent \d t~.
\end{align*}
Note that the straight path is always admissible regardless of $\F$ since its derivative is constant.
More generally, we shall refer as $[x_1, x_2, \dots, x_k]$ to the polygonal path going through $x_1$, $x_2$, \dots, $x_k$ in this order, which is admissible if all intermediate points belong to the domain $F$.

We now give a few important remarks before stating our contributions.

\medskip\noindent
\textit{About the role of $\F$.}
The domain $\F \subset \RR^d$ is introduced to achieve both general stability results and interesting convergence rates for the estimation of the \ac{fdtm} in a way that allows to evaluate all admissible paths with reasonable computational complexity, that is independent of the ambient dimension~$d$.
For estimation purpose $\F$ will be chosen to be the support of~$\mu$, so that sampling points according to $\mu$ creates a point cloud that fills the domain, and polygonal paths on this point cloud can approximate all admissible paths.
Drawing straight edges outside the domain is allowed in order to improve the stability of the \ac{fdtm} \ac{wrt} the domain without hindering its estimation, as the straight edges outside the support can still be approximated by polygonal paths.

Consider briefly the more constrained definition where the admissible paths must remain entirely within the domain. In this context, quantitative stability and estimation properties would require some assumption of regularity on the domain.
For instance,~\cite{taupinEstimationConformalMetrics2026} studies \emph{conformal metrics} on sets of positive reach, that is metrics defined by minimizing the integral of some function over paths that remain inside the domain. The Fermat distance is an example of such metric where the conformal factor is a negative power of the density. The \ac{fdtm} with the more restrictive definition of admissible paths would also be one such example.
Assuming that the domain has positive reach and that the conformal factor is Lipschitz,~\cite{taupinEstimationConformalMetrics2026} derives an estimator of the conformal metric similar to what is presented in the following sections, with explicit convergence rates.
Note that the Lipschitz property of the integrand is crucial, which further motivates the use of the \ac{dtm}. Indeed, recall that the $d_\mu$ is Lipschitz, and so is $d_\mu^\exponent$ as long as $d_\mu$ is bounded above and below.
As the case of sets with positive reach is discussed in~\cite{taupinEstimationConformalMetrics2026}, we focus in this article on the more general context where no regularity is assumed on the domain, hence the slightly more convoluted definition of admissible paths to deal with irregularities.

In general, if $\F\subset \G\subset\RR^d$ then $\Gamma_\F(x,y) \subset \Gamma_\G(x,y)$ for all endpoints $x$ and $y$, and it follows immediately that $D_{\mu,\G}(x,y) \le D_{\mu,\F}(x,y)$.
Moreover, $\bLp \F, D_{\mu,\F} \bRp$ is a metric space (see \cref{app:equivalent_topology}).
In particular, the triangular inequality applies when considering endpoints within $\F$.
However, this is not true when considering general endpoints:
If $x,y\in \F$ and $z\notin \F$, it may happen that
    \[D_\mu(x,z) + D_\mu(z,y)
    \le D_\mu([x,z,y])
    < D_\mu(x,y)\]
if the polygonal path $[x,z,y]$ does not belong to $\Gamma_\F(x,y)$.

\medskip\noindent
\textit{Choice of the parameters.}
As mentioned previously, the \ac{fdtm} intuitively behaves similarly to the Fermat distance. Indeed, we shall see that when the density does exist, it is the limit of $m (d_{\mu,m})^{-d}$ as $m$ tends to $0$, up to some constant. The choice of the parameter $m$ is thus crucial as there is a tradeoff between robustness for large $m$ and similarity to the Fermat distance for small $m$.
Moreover, if one wants to compare the \ac{fdtm} to the Fermat distance with parameter $\beta$, the appropriate choice for $\exponent$ is $d \beta$, as $m^{-\beta} d_{\mu,m}^\exponent$ then converges to $f^{-\beta}$.
As for $p$, it does not hold any particular interest apart from dictating which Wasserstein distance appears in \cref{eq:dtm_stab_wasserstein}. Other works featuring the \ac{dtm} such as~\cite{biauWeightedKnearestNeighbor2011} typically choose $p=2$.

\medskip\noindent
\textit{Restriction to a compact subset $\K$.}
Throughout the paper, $\K \subset \RR^d$ denotes a compact convex set with diameter $\diam(\K)$.
Without loss of generality and to achieve clearer results, we choose to restrict all objects to be contained within $\K$, that is, support of measures, domains, and endpoints.
The main benefit of this assumption is to uniformly upper bound the \ac{dtm} over $\K$ by $\diam(\K)$, which in particular also holds over all admissible paths. Indeed, if the support is in $\K$ then all points in $\K$ are at distance at most $\diam(\K)$ from any point of the support, thus the \ac{dtm} is upper bounded by $\diam(\K)$.
Denote $\cM_\K$ the set of measures with support included in $\K$. Then for any such measure $\mu\in\cM_\K$, domain $\F\subset\K$, endpoints $x,y\in\K$, and path $\gamma\in\Gamma_\F(x,y)$,
\begin{align}
    \label{eq:dtm_uniform_bound}
    \|d_\mu\|_{\infty,K} \le \diam(\K)
    \qquad\textrm{and}\qquad
    \|d_\mu\|_{\infty,\gamma} \le \diam(\K)
    ~.
\end{align}
The fact that this assumption is without loss of generality is discussed in \cref{app:restriction_compact}.
In particular, the implication that $\mu$ has compact support in $\RR^d$ is not very restrictive, as we show in \cref{app:restriction_compact} that given any measure $\mu$ and any compact area of study $A \subset \RR^d$, the restriction of the \ac{fdtm} of $\mu$ over $A$ coincides with the \ac{fdtm} of a ``cropped'' modification of $\mu$ with support included in a larger compact subset $\K$. 

\begin{table}[tb]
    \centering
    \begin{tabular}{r|l}
        \textsc{Notation} & \textsc{Definition} \\
        \hline \\
        $\|\cdot\|$ & Euclidean norm over $\RR^d$ \\
        $\cB(x, r)$ & open ball centered at $x\in\RR^d$ with radius $r$ \\
        $\ball(x, r)$ & closed ball centered at $x\in\RR^d$ with radius $r$ \\
        $\mu, \nu$ & measures over $\RR^d$ \\
        $\supp(\mu)$ & support of the measure $\mu$ \\
        $\F, \G$ & domains for paths \\
        $m$ & fraction of mass considered for the \ac{dtm} \\
        $p$ & exponent of the Hölder mean considered for the \ac{dtm} \\
        $\exponent$ & exponent of the integrated \ac{dtm} considered for the \ac{fdtm} \\
        $\K$ & compact convex subset of $\RR^d$ (often with diameter $1$)\\
        $a, b$ & parameters used of the standard assumption on measures \\
        $\sigma$ & lower bound on the \ac{dtm} \\
        $\cM_\K$ & set of measures with support included in $\K$ \\
        $\cM_{\K,a,b,\sigma}$ & set of measures satisfying \cref{assum} \\
        $d_\mu$ & \ac{dtm} associated with $\mu$\\
        $\gamma$ & rectifiable path in $\RR^d$ \\
        $|\gamma|$ & Euclidean length of path $\gamma$ \\
        $\|d_\mu\|_{\infty,\gamma}$ & maximal value of the \ac{dtm} reached on the path $\gamma$ \\
        $L_{\mu,\delta}$ & sublevel set of $d_\mu$  below the value $\delta$ \\
        $D_\mu(\gamma)$ & \ac{fdtm} length of the path $\gamma$ \\
        $D_{\mu,\F}(x,y)$ & \ac{fdtm} distance from $x$ to $y$ \\
        $D_\mu(x,y)$ & special case when $\F = \supp(\mu)$ \\
        $\Gamma_\F(x,y)$ & set of rectifiable paths with domain $\F$ from $x$ to $y$ \\
        $\Gamma_{\mu,\F}^\star(x,y)$ & set of geodesics paths for $\mu$ with domain $\F$ from $x$ to $y$ \\
        $[x_1, x_2, \dots, x_k]$ & polygonal path going through $x_1$, $x_2$, \dots, $x_k$ \\
        $W_p$  & Wasserstein distance of order $p$ \\
        $\dtv$ & Total-Variation distance \\
        $\hausdorff$ & Hausdorff distance between closed subsets of $\RR^d$ \\
        $n$ & number of sample points \\
        $\XX_n$ & family of $n$ random samples from a measure $\mu$ \\
        $\hat\mu_n$ & empirical measure associated with $\XX_n$
    \end{tabular}
    \caption{Notations used throughout the paper}
    \label{tab:notations}
\end{table}

\medskip\noindent
\textit{Lower bounding the \ac{dtm}.}
For $d_\mu$ to vanish at a point $x$, $\mu$ needs to have an atom of mass at least $m$ at~$x$.
Assuming that this does not happen anywhere, $d_\mu$ is positive and, as a continuous and proper function, achieves a positive minimum over $\RR^d$:
\begin{lemma}
\label{lemma:dtm_lower_bound}
Let $\mu$ be a probability measure over $\RR^d$ such that $\mu(\{x\}) < m$ for all $x\in\RR^d$.
Then $\min d_\mu > 0$.
\end{lemma}

\subsection{Overview of the Results}
\label{sec:contributions}

We now summarize our contributions.

\subsubsection{Existence of geodesics and upper bound on their Euclidean length.}
Without domain constraints, the \ac{fdtm} metric can be seen as a continuous conformal transformation of the Euclidean metric via the \ac{dtm} function, and the existence of minimizing geodesics follows from classical arguments (see for instance Proposition 2.5.19 in~\cite{buragoCourseMetricGeometry2001}).
The proof is adapted to account for domain constraints in \cref{sec:geodesic_existence}.

\begin{theorem}[Existence of geodesics]
\label{thm:intro:geodesic}
Let $\mu$ be a probability measure over $\RR^d$ and $\F\subset\RR^d$ a closed domain. Then for all $x,y\in\RR^d$, there exists a path $\gamma\in\Gamma_\F(x,y)$ such that
\begin{align*}
    D_{\mu,F}(x,y)
    = D_\mu(\gamma)~.
\end{align*}
We denote $\Gamma_{\mu,\F}^\star(x,y)$ the set of such paths, called (minimizing) geodesics.
\end{theorem}

Notice that \cref{thm:intro:geodesic} holds regardless of the measure $\mu$ or the domain $\F$.
In particular, the existence of geodesics holds for all measures in $\RR^d$, including those without compact support. This is tied with the fact that geodesics may not escape a large ball around the endpoints due to the regularity properties of the \ac{dtm}---see \cref{corol:euclidean_bound_local}---regardless of the measure.
This observation contrasts with the Fermat distances, for which one may construct a measure whose support is made of a family of disjoint strips from $x$ to $y$ that make increasingly longer detours while being thinner but with higher density. More precisely, consider a family of strips $(s_k)_k$ with respective lengths $(L_k)_k$, thickness $(h_k)_k$, and density $(a_k)_k$ such that:
\begin{itemize}
    \item $L_k \to \pinfty$, $h_k\to0$, $a_k\to\pinfty$.
    \item The series with general term $L_k h_k a_k$ converges to $1$, ensuring that the underlying measure is a probability. 
    \item $L_k a_k^{-\exponent}$ is decreasing, ensuring that each strip is shorter in the sense of \cref{eq:fermat} despite being longer Euclidean-wise.
\end{itemize}
In this example each strip $s_k$ yields a lower Fermat distance between $x$ and $y$ than the previous one, so that there are no geodesics for the Fermat metric.
On the other hand, strips with total mass lower than $m$ are negligible from the point of view of the \ac{dtm}. Therefore, the \ac{fdtm} favors the first strips, that have both smaller Euclidean length and smaller \ac{dtm} values.

In addition to the existence of geodesics, the $1$-Lipschitz property of the \ac{dtm} (\crefnopar{eq:dtm_lip}) and the fact that it cannot take small values on a subset that is too large---see \cref{lemma:dtm_covering}---allow to upper bound their Euclidean length, regardless of the underlying measure and domain.
This upper bound then plays a fundamental role in establishing stability properties of the \ac{fdtm}.

\begin{theorem}[Upper bound on the Euclidean length of geodesics]
\label{thm:intro:euclidean_bound}
Let $\mu$ be a probability measure over $\RR^d$ and $\F\subset\RR^d$ a closed domain.
Then for all geodesic $\gamma\in\Gamma_{\mu,\F}^\star(x,y)$,
\begin{align}
    \label{eq:thm:euclidean_bound}
    |\gamma|
    \lesssim \|d_\mu\|_{\infty,\gamma}~,
\end{align}
where $\lesssim$ hides a multiplicative constant---made explicit in the proofs---only depending on $m$ and $\exponent$.
\end{theorem}

\begin{corollary}
\label{corol:euclidean_bound_uniform}
Assume that $\supp(\mu) \subset \K$, $\F \subset \K$, $x,y\in\K$, and let $\gamma\in\Gamma_{\mu,\F}^\star(x,y)$.
Then
\begin{align*}
    |\gamma|
    \lesssim \diam(\K)~,
\end{align*}
where $\lesssim$ hides the same multiplicative constant as in \cref{eq:thm:euclidean_bound}.
\end{corollary}

\begin{corollary}
\label{corol:euclidean_bound_local}
Let $\gamma\in\Gamma_{\mu,\F}^\star(x,y)$.
Then
\begin{align*}
    |\gamma|
    \lesssim \|x-y\|~,
\end{align*}
where $\lesssim$ hides a multiplicative constant depending on $\exponent$ and the constant from \cref{eq:thm:euclidean_bound}.
\end{corollary}

The reasoning to derive \cref{thm:intro:euclidean_bound} is described in \cref{sec:geodesic_length}.
\Cref{corol:euclidean_bound_uniform} is a direct consequence of \cref{eq:thm:euclidean_bound,eq:dtm_uniform_bound}.
As for \cref{corol:euclidean_bound_local}, it is also a direct consequence of \cref{eq:thm:euclidean_bound} if $\|d_\mu\|_{\infty,\gamma}$ is smaller than $\|x-y\|$. Else, it is still possible to obtain the domination by the Euclidean distance by simpler means, which is detailed in \cref{app:euclidean_bound_local}.

Recall that under mild assumptions on the measure $\mu$, the \ac{dtm} is lower bounded by a positive quantity (see \cref{lemma:dtm_lower_bound}).
In this special case the Euclidean length of geodesics can be upper bounded in a straightforward manner, albeit with a dependency in the \ac{dtm} lower bound.
\Cref{prop:euclidean_bound_min_dtm} is also proven in \cref{sec:geodesic_length}.
\begin{proposition}
\label{prop:euclidean_bound_min_dtm}
Assume that $\min d_\mu>0$ and let $\gamma\in\Gamma_{\mu,\F}^\star(x,y)$.
Then
\begin{align*}
    |\gamma|
    \le \Lp\frac{\max_{[x,y]}d_\mu}{\min d_\mu}\Rp^\exponent \|x-y\|~.
\end{align*}
Under the same assumptions as in \cref{corol:euclidean_bound_uniform}, this further implies that
\begin{align*}
    |\gamma|
    \le \Lp\frac{\diam(\K)}{\min d_\mu}\Rp^\exponent \diam(\K)~.
\end{align*}
\end{proposition}

\subsubsection{Stability of the \ac{fdtm}.}
Beyond the fact that the \ac{fdtm} metric is defined for any probability measure on $\RR^d$, one of the main motivations to introduce it as a variant of the Fermat distance is that it turns out to be more stable \ac{wrt} the measure.
This property follows from \cref{thm:intro:euclidean_bound}, the Wasserstein stability of the \ac{dtm} (\crefnopar{eq:dtm_stab_wasserstein}), and is proven in \cref{sec:stability_dtm}.
The following statements hold regardless of the choice of parameters $m$, $p$, and $\exponent$. However, if $\exponent < 1$, the \ac{dtm} needs to be positive over $\K$ for them to hold. This is not a very restrictive assumption, see \cref{lemma:dtm_lower_bound}.

\begin{theorem}[\ac{fdtm} stability \ac{wrt} the measure]
\label{thm:intro:fdtm_stab_dtm}
Let $\mu,\nu \in \cM_\K$ and $\F \subset \K$ a closed domain.
Then
\begin{align*}
    \bvv D_{\mu,\F} - D_{\nu,\F} \bvv_{\infty,\K}
    \lesssim \diam(\K)^\exponent~W_p(\mu,\nu)~,
\end{align*}
where $\lesssim$ hides a multiplicative constant---made explicit in the proofs---only depending on $m$ and $\exponent$ when $\exponent \ge 1$. If $\exponent<1$ this constant also involves a lower bound on $d_\mu$ and $d_\nu$.
\end{theorem}

Likewise, the \ac{fdtm} is stable \ac{wrt} the domain. This property is obtained by approximating a geodesic for the first domain using a path that is admissible for the second domain, see \cref{sec:stability_domain}.

\begin{theorem}[\ac{fdtm} stability \ac{wrt} the domain]
\label{thm:intro:fdtm_stab_domain}
Let $\mu \in \cM_\K$ and $\F,\G \subset \K$ two closed domains.
Then
\begin{align*}
    \bvv D_{\mu,\F} - D_{\mu,\G} \bvv_{\infty,\K}
    \lesssim \diam(\K)^{\exponent+\frac12} \sqrt{\hausdorff(\F,\G)}~,
\end{align*}
where $\lesssim$ hides a multiplicative constant---made explicit in the proofs---only depending on $m$ and $\exponent$ when $\exponent \ge 1$, and $\hausdorff(\F,\G)$ denotes the Hausdorff distance between $\F$ and $\G$ defined as
\begin{align*}
    \hausdorff(\F,\G)
    \eqdef \max \bLp \|d(\cdot,\G)\|_{\infty,\F} , \|d(\cdot,\F)\|_{\infty,\G} \bRp~.
\end{align*}
If $\exponent<1$ the constant also involves a lower bound on $d_\mu$.
\end{theorem}

Both \cref{thm:intro:fdtm_stab_dtm,thm:intro:fdtm_stab_domain} require no assumption on the regularity of the measures and domains. This is possible due to the definition of admissible paths allowing ``shortcuts'' to override irregularities in the domain.

\subsubsection{Empirical \ac{fdtm}.}
The previous stability results provide efficient tools to study the estimation of the \ac{fdtm} from random samples. 
Given $n\ge1$ sample points $(X_1,\dots,X_n)$ drawn \ac{iid} from $\mu$, we denote $\XX_n = \{X_1, \dots, X_n\}$ the corresponding point cloud and $\hat\mu_n = \frac1n \sum_{k=1}^n \delta_{X_k}$ the empirical measure.
Estimation results are restricted to the special case where the domain $\F$ is chosen as the support $\supp(\mu)$ of the measure, and we denote $D_\mu = D_{\mu,\supp(\mu)}$ for simplicity.
Indeed, to avoid a strong computation cost depending on the ambient dimension $d$, it is easier to compute paths made up of straight lines between sample points, which are exactly the paths admissible for the domain $\XX_n$.
Since $d_{\hat\mu_n}$ is expected to converge towards $d_\mu$ and $\XX_n = \supp(\hat\mu_n)$ is expected to fill up the whole support $\supp(\mu)$, the \ac{fdtm} $D_\mu$ can be estimated by the empirical \ac{fdtm} $D_{\hat\mu_n} = D_{\hat\mu_n,\XX_n}$, which is the metric of a finite weighted graph.
The weights are the integrals of the \ac{dtm} over edges, hence cannot be reasonably computed in general. In practice, as the \ac{dtm} is $1$-Lipschitz, each integral can be approximated by a discrete sum over a regular subdivision of the underlying edge, using finer subdivisions as $n$ grows to retain convergence properties.
For simplicity we keep the original definition of the weights in our theoretical results. Practical considerations are discussed in \cref{sec:algorithm}.

In order to control the convergence rate of the empirical \ac{fdtm} over $\K$, the offset is decomposed into two terms as
\begin{align*}
    \bvv D_\mu - D_{\hat\mu_n} \bvv_{\infty,\K}
    \le \bvv D_{\mu,\supp(\mu)} - D_{\hat\mu_n,\supp(\mu)} \bvv_{\infty,\K}
        + \bvv D_{\hat\mu_n,\supp(\mu)} - D_{\hat\mu_n,\XX_n} \bvv_{\infty,\K}~.
\end{align*}
Then, \cref{thm:intro:fdtm_stab_domain,thm:intro:fdtm_stab_dtm} are used to upper bound both terms by quantities for which convergence rates can be derived.
To do so, the following assumption is used.
\begin{assumption}
\label{assum}
\hfill
\begin{enumerate}[(i)]
    \item The support of $\mu$ is included in a compact subset of $\RR^d$: $\supp(\mu) \subset \K$.
    \item $\mu$ is a $(a,b)$-standard probability measure for some $a>0$ and $b\ge1$, \ie, for all $x\in\supp(\mu)$ and $0 < r \le a^{-1/b}$,
    \begin{align*}
        \mu\bLp\cB(x,r)\bRp
        \ge ar^b.
    \end{align*}
    \item $\min d_\mu \ge \sigma$ for some $\sigma>0$.
    \item $\supp(\mu)$ is connected.
\end{enumerate}
\end{assumption}
Recall that \cref{assum}.(i) is without loss of generality (see \cref{app:restriction_compact}).
\Cref{assum}.(ii) is the most important assumption and will direct the rate of convergence of the empirical \ac{fdtm}. The term \emph{$(a,b)$-standard} is borrowed from \cite{chazalRatesConvergenceRobust2016} and $b$ may be interpreted as an upper bound on the intrinsic dimension of the support of $\mu$.
Regarding \cref{assum}.(iii), \cref{lemma:dtm_lower_bound} ensures that it is satisfied by any measure without atoms of mass larger than $m$.
\Cref{assum}.(iii) and (iv) are technical assumptions that do not have a fundamental impact on the results. 
We denote $\cM_{\K,a,b,\sigma}$ the set of measures satisfying \cref{assum} for given $\K$, $a$, $b$, $\sigma$, and can now state our main result.

\begin{theorem}[Convergence of the empirical \ac{fdtm}]
\label{thm:intro:fdtm_estimate}
Assume that $\mu \in \cM_{\K,a,b,\sigma}$ satisfies \cref{assum} and that $m \le \frac12$.
Then for all $n\ge\frac1m$,
\begin{align*}
    \EE\Lb\bvv D_\mu - D_{\hat\mu_n} \bvv_{\infty,\K}\Rb
    \lesssim d \diam(\K)^{\exponent+1} \frac{\log(n)}{n^{\frac{1}{2b}}}~,
\end{align*}
where $\lesssim$ hides a multiplicative constant depending on $m$, $p$, $\exponent$, $\diam(\K)\sigma^{-1}$, and $\diam(\K)a^{1/b}$.
\end{theorem}
A more detailed statement is provided in \cref{sec:convergence}.
A weaker version may be stated without \cref{assum}.(iii) and (iv), which does not require any lower bound on the \ac{dtm} but features an additional term of order $n^{-1/2p}$, possibly achieving a worse rate of convergence when $p$ is large.
The assumption that $m \le \frac12$ is purely technical and not very restrictive, as using a higher value $m$ for the \ac{dtm} yields a very coarse object without much interest.

The main strength of \cref{thm:intro:fdtm_estimate} is its generality. The most restrictive assumption is the $(a,b)$-standard assumption which remains light compared to regularity conditions on the density or on the support of the measure.
Compared to the convergence result provided by~\cite[Theorem 2.3]{groismanNonhomogeneousEuclideanFirstpassage2022} for the Fermat distance, \cref{thm:intro:fdtm_estimate} is not reduced to the asymptotic regime and does not involve any normalization constant.

Another idea to achieve a similar result in the case where the Fermat distance cannot be defined due to $\mu$ not having a density, would be to use the original Fermat distance but with a ``smoothed'' version of~$\mu$, for instance through a convolution with a Gaussian measure, instead of replacing the density with the \ac{dtm} in \cref{eq:fermat}.
However, this would require some regularity assumptions on the support, and by doing so the intrinsic dimension of the measure would default to the ambient dimension, which would affect negatively the convergence properties of the Fermat distance. In \cref{thm:intro:fdtm_estimate}, the convergence rate does not involve the ambient dimension.

\subsubsection{Minimax lower bound.}
Our final point of interest is to evaluate how close to optimal is our proposed \ac{fdtm} estimator from a minimax viewpoint.
We establish a lower bound on the convergence rate of any estimator of the \ac{fdtm} in its worst case scenario for the measure~$\mu$, see \cref{sec:minimax}.

\begin{theorem}[Minimax lower bound]
\label{thm:intro:minimax}
Fix $\K\subset\RR^d$, $b\ge1$, and assume that $\K \subset \RR^d$ is not a segment.
Then there exist $x,y \in \K$ and constants $a_0$, $\sigma_0$, and $n_0$ depending on $m$, $p$, $\exponent$, and $b$, such that for all $a \le a_0$, $\sigma \le \sigma_0$, and $n \ge n_0$,
\begin{align*}
    \min_{\hat D\in\cD_n} \max_{\mu\in\cM_{\K,a,b,\sigma}} \EE\Lb\bv \hat D - D_\mu(x,y) \bv\Rb
    \gtrsim \frac{\|x-y\|^{\exponent+1}}{n^{\frac1b\wedge\frac12}}~,
\end{align*}
where $\cD_n$ is the set of possible estimators of $D_\mu(x,y)$ from $n$ sample points, and $\gtrsim$ hides a multiplicative constant depending on $m$, $p$, $\exponent$, and $b$.
\end{theorem}

The above lower bound features a convergence rate depending on the inverse of the intrinsic dimension, similarly to the upper bound in \cref{thm:intro:fdtm_estimate}, although a gap subsists when $b\ge2$.
Our estimator is therefore minimax optimal when $b=1$. The minimax optimal convergence rate when $b\ge2$ remains unclear and its precise study is beyond the scope of this paper.

\subsubsection{Link between DTM and density}
Given a measure $\mu$ with density \ac{wrt} the Lebesgue measure, it has been shown in~\cite{biauWeightedKnearestNeighbor2011} that the \ac{dtm} associated with the empirical measure of a random \ac{iid} sampling from $\mu$ can be used to estimate the density.
Here we provide a similar statement in the deterministic case, that is comparing the density to the \ac{dtm} of the measure itself.
\begin{restatable}{theorem}{dtmDensity}
\label{thm:dtm_density}
Assume that $\mu$ has density $f$ \ac{wrt} to the Lebesgue measure in $\RR^d$. 
Denote $\omega_d$ the volume of the unit ball in $\RR^d$, and for all $x\in \RR^d$, define
\begin{align*}
    \hat f_m(x)
    = \frac{1}{\omega_d} \Lp\frac{d}{d+p}\Rp^{\frac{d}{p}} \frac{m}{d_{\mu,m,p}(x)^d}~.
\end{align*}
If $f$ is $L$-Lipschitz, then for all $m \in (0,1]$,
\begin{align*}
    \bvv\hat f_m - f\bvv_\infty
    \le 2\Lp\frac{d+1}{\omega_d} L^d m\Rp^{\frac{1}{d+1}}~.
\end{align*}
\end{restatable}
The estimator used in \cite{biauWeightedKnearestNeighbor2011} is the same as $\hat f_m$ except the \ac{dtm} of the empirical measure is used instead of~$d_\mu$.
\Cref{thm:dtm_density} is proven using elementary computations in \cref{sec:dtm_density}.
The above hints that by setting $\exponent = d\beta$ where $\beta$ is the parameter of the Fermat distance in \cref{eq:fermat}, the \ac{fdtm} converges to the Fermat distance when $m\to0$ when comparing relative distances.
Indeed, the normalized integrand $m^{-\beta} (d_{\mu,m,p})^\exponent$ is shown to converge to $f^{-\beta}$ up to a constant factor over the support. Moreover, it goes to infinity outside as $d_{\mu,m,p}$ converges to the distance to the support, which imply that geodesics will avoid venturing outside the support as $m$ goes to $0$.

Given $\alpha>1$, the sample Fermat distance, which is the natural estimator of the Fermat distance, is defined as the complete graph metric with vertices the sample points and edges weights $w_{x,y} = \|x-y\|^\alpha$.
It has been shown in~\cite[Theorem 1]{hwangShortestPathRandom2016} that the sample Fermat distance converges under appropriate normalization and regularity assumptions on the measure towards the (continuous) Fermat distance defined by \cref{eq:fermat}, provided that $d\beta = \alpha-1$.
Therefore, if one want to compare the \ac{fdtm} to the sample Fermat distance, the relevant choice of parameters is $\exponent = \alpha-1$.

\section{Geodesics and Stability Properties of the FDTM}
\label{sec:geodesic&stability}

This sections details how the existence of geodesics for the \ac{fdtm} is ensured, how to obtain a uniform bound on their Euclidean length, and how this ensures quantitative stability bounds on the \ac{fdtm}.

\subsection{Existence of Geodesics}
\label{sec:geodesic_existence}

Let $\mu$ be a probability measure with support contained in a compact convex set $\K$, $\F\subset\K$ a closed domain, and $x,y\in\RR^d$. Recall that these assumptions are without loss of generality (see \cref{app:restriction_compact}).
Without domain restrictions, the existence of geodesics can be derived from compactness arguments (see Proposition 2.5.19 in~\cite{buragoCourseMetricGeometry2001}).
In general, the \ac{fdtm} model does not exactly fit the framework for this result due to the fact that the concatenation of two admissible paths may not be admissible if both paths meet outside the domain. In the following we adapt the proof to account for this issue.

\begin{proof}[Proof of \cref{thm:intro:geodesic}]
We first treat the case of endpoints $x,y \in \F$.
Note that the topologies of both the Euclidean and the \ac{fdtm} metrics are equivalent (a rigorous explanation is given in \cref{app:equivalent_topology}).
In particular, since $\F$ is a closed subset of $\K$, $(\F,\|\cdot\|)$ is a compact metric space and so is $(\F,D_{\mu,\F})$.
Moreover, the set of rectifiable curve, \ie, of curves with finite Euclidean length, is the same as the set of curves with finite \ac{fdtm} length.

Let $(\gamma_k)_{k\ge1}$ a sequence of paths in $\Gamma_\F(x,y)$ such that $\lim_{k\to\pinfty} D_\mu(\gamma_k) = D_{\mu,\F}(x,y)$.
This sequence has uniformly bounded \ac{fdtm} length by definition.
Therefore, according to~\cite[Theorem 2.5.14]{buragoCourseMetricGeometry2001} and replacing the sequence with a subsequence if needed, $\gamma_k$ converges uniformly to a rectifiable curve $\gamma$.
Moreover,~\cite[Proposition 2.3.4.iv]{buragoCourseMetricGeometry2001} states that $\gamma \mapsto D_\mu(\gamma)$ is lower semi-continuous, which implies that
\begin{align*}
    D_\mu(\gamma) \le \liminf_{k\to\pinfty} D_\mu(\gamma_k) = D_{\mu,\F}(x,y)~.
\end{align*}
It remains to show that $\gamma\in\Gamma_\F(x,y)$. It is immediate from the uniform convergence that $\gamma(0)=x$ and $\gamma(1)=y$.
Now, given any open interval $I$ such that $\gamma(I) \subset \RR^d\setminus \F$, it suffices to show that $\dot\gamma$ has constant direction over all segments $J\subset I$.
First, $\gamma(J)$ is compact and at positive distance away from $\F$, so that by uniform convergence $\gamma_k(J) \subset \RR^d\setminus \F$ for $k$ large enough, hence $\dot\gamma_k(t) = \|\dot\gamma_k(t)\|u_k$ for all $t\in J$ where $u_k$ is a unit vector.
Then for all $s<t\in J$,
\begin{align*}
    \gamma_k(t)-\gamma_k(s)
    = \int_s^t \dot\gamma_k
    = \Lp \int_s^t \|\dot\gamma_k\| \Rp \cdot u_k~.
\end{align*}
Due to the uniform convergence of $\gamma_k$ to $\gamma$, $\gamma_k(t)-\gamma_k(s)$ converges to $\gamma(t)-\gamma(s)$.
As $u_k$ is a unit vector,
\begin{align*}
    \lim_{k\to\pinfty} u_k
    = \lim_{k\to\pinfty} \frac{\gamma_k(t)-\gamma_k(s)}{\|\gamma_k(t)-\gamma_k(s)\|}
    = \frac{\gamma(t)-\gamma(s)}{\|\gamma(t)-\gamma(s)\|}~.
\end{align*}
Since this is true for arbitrary $s,t\in J$, we get
\begin{align*}
    \dot\gamma(t)
    = \|\dot\gamma(t)\| u
\end{align*}
for all $t$ in the interior of $J$, where $u = \lim_{k\to\pinfty} u_k$.
This concludes the fact that $\gamma\in\Gamma_\F(x,y)$.
Therefore, $\gamma$ is a geodesic.

If $x$ does not belong to $\F$, notice that any path in $\Gamma_\F(x,y)$ is made of a straight line $[x,z]$ going to $\F$ followed by a path in $\Gamma_\F(z,y)$, where $z\in \F$ may be equal to $y$.
Thus
\begin{align*}
    D_{\mu,\F}(x,y)
    = \inf_{z\in \F}\bLp D_\mu([x,z]) + D_{\mu,\F}(z,y) \bRp
\end{align*}
is the infimum of a continuous function of $z$ over the compact set $\F$, hence a minimum.
Therefore, there exists a geodesic.
The same reasoning goes for $y\notin \F$ with the addition of the direct path $[x,y]$ which may not go through $\F$ at all.
This concludes the proof in the general case.
\end{proof}

\subsection{Bounding the Length of Geodesics}
\label{sec:geodesic_length}

In the particular case where the \ac{dtm} is lower bounded by a positive quantity, a bound can be obtained on the length of geodesics depending on the lower bound by using elementary arguments.
In general, the properties of the \ac{dtm} allow to obtain a bound that does not require any assumption on $\mu$, although this requires a more complex reasoning.

\subsubsection{Case of a Lower Bounded DTM}
When $\min d_\mu > 0$, the length of geodesics can easily be upper bounded.
Indeed, the \ac{fdtm} between two endpoints $x$ and $y$ is at most the \ac{fdtm} length of the segment $[x,y]$, and
\begin{align*}
    D_{\mu,\F}(x,y)
    \le D_\mu([x,y])
    \le \|x-y\| \max_{[x,y]}d_\mu^\exponent~.
\end{align*}
Then, if $\gamma \in \Gamma_{\mu,\F}^\star(x,y)$, by using the lower bound on the \ac{dtm},
\begin{align*}
    D_{\mu,\F}(x,y)
    = D_\mu(\gamma)
    \ge |\gamma| (\min d_\mu)^\exponent~.
\end{align*}
Therefore,
\begin{align*}
    |\gamma|
    \le \Lp\frac{\max_{[x,y]}d_\mu}{\min d_\mu}\Rp^\exponent \|x-y\|~,
\end{align*}
which proves \cref{prop:euclidean_bound_min_dtm}.

\subsubsection{General Case (sketch of proof)}
\label{sec:geodesic_length_general}

We now focus on the general case where the upper bound from \cref{thm:intro:euclidean_bound} can be derived with no dependency on the measure itself nor on the domain.
The study of the Euclidean length of geodesic paths requires to consider the sublevel sets
\[L_{\mu,\delta} \eqdef \{x\in\RR^d : d_\mu(x) < \delta\}\]
of the \ac{dtm}. A key property of the \ac{dtm} is that $d_\mu(x)$ being small implies that a substantial amount of mass from $\mu$ lies nearby $x$. This remark allows to upper bound the size of the sublevel sets $L_{\mu,\delta}$ through their covering and packing numbers, see \cref{lemma:dtm_covering}.

The upper bound on the Euclidean length of geodesics stated by \cref{thm:intro:euclidean_bound} builds upon the above property and the Lipschitz property of the \ac{dtm} (\crefnopar{eq:dtm_lip}).
Indeed, one can modify the trajectory of a geodesic path within a given sublevel to create an alternative path that avoids long squiggling in areas of low \ac{dtm}. Both the Euclidean and \ac{fdtm} lengths of this alternative path within the sublevel can be upper bounded knowing that the \ac{dtm} is Lipchitz and that the size of the sublevel is proportional to the level itself, see \cref{lemma:geodesic:gamma_modif_no_domain} in the special case without domain constraints and \cref{lemma:geodesic:gamma_modif_domain} in the general case.
The alternative path also has larger \ac{fdtm} length than the geodesic by definition.
It follows that given a geodesic $\gamma$ and some $\rho > 1$, the integral of $d_\mu^\exponent$ over the section of $\gamma$ that lies within $L_{\mu,\rho\delta} \setminus L_{\mu,\delta}$ is upper bounded by $(\rho\delta)^{\exponent+1}$.
Since this integral is also naturally lower bounded by the length of the corresponding section of $\gamma$ times $\delta^\exponent$, one can eventually show that the length of the section is at most proportional to $\delta$. Applying this reasoning to a sequence of levels $(\delta_k)_k = (\rho^{-k} \max d_\mu)_k$, the overall Euclidean length of the geodesic is upper bounded as stated by \cref{thm:intro:euclidean_bound}.

\Cref{app:geodesics} provides the upper bound on the covering numbers of the sublevels of the \ac{dtm}, more insight on the construction of the alternative path (see for instance \cref{fig:path_decomposition_no_domain,fig:path_decomposition_with_domain}), and the complete proof of \cref{thm:intro:euclidean_bound}.
Precisely, the statement is first proven in the case where no constraints are put on the domain, \ie, $\F=\RR^d$, in \cref{app:geodesic:details_no_domain}.
The proof in the general case is more technical, and is detailed in \cref{app:geodesic:details_domain}.
Finally, \cref{app:euclidean_bound_local} contains the detailed computations to establish \cref{corol:euclidean_bound_local}.

\subsection{Stability of the FDTM}
\label{sec:stability}

The upper bound on the length of geodesics allows to derive stability results in a straightforward manner.
Recall that $d_\mu$ is $1$-Lipschitz and upper bounded by $d_\mu \le \diam(\K)$ over $\K$ (see Eqs. \ref{eq:dtm_lip} and \ref{eq:dtm_uniform_bound}).
It follows that $d_\mu^\exponent$ is $\kappa$-Lipschitz over $\K$ where $\kappa$ is a Lipschitz constant of $t\in[\min d_\mu,\diam(\K)] \mapsto t^\exponent$. If $q \ge 1$ then $\kappa = \exponent\diam(\K)^{q-1}$. If $q < 1$ and $\min d_\mu \ge \sigma > 0$ over $\K$ then $\kappa = \exponent\sigma^{q-1}$.

\subsubsection{Stability with respect to the Measure}
\label{sec:stability_dtm}

We state a more precise version of \cref{thm:intro:fdtm_stab_dtm}.
While the Lipschitz stability \ac{wrt} the Wasserstein distance is interesting in itself, an intermediate bound will be used instead when dealing with convergence rate of the empirical \ac{fdtm}.
\cref{thm:fdtm_stab_dtm} is a direct consequence of the upper bound on the length of geodesics from \cref{corol:euclidean_bound_uniform} and of the Wasserstein stability of the \ac{dtm} from \cref{eq:dtm_stab_wasserstein}.

\begin{theorem}
\label{thm:fdtm_stab_dtm}
Let $\mu,\nu \in \cM_\K$ and $\F \subset \K$ a closed domain.
Then
\begin{align*}
    \bvv D_{\mu,\F} - D_{\nu,\F} \bvv_{\infty,\K}
    \le \lambda \diam(\K) \bvv d_\mu^\exponent-d_\nu^\exponent \bvv_{\infty,\K}~,
\end{align*}
where $\lambda$ is the hidden constant in \cref{eq:thm:euclidean_bound}.
In particular, if $\exponent \ge 1$, then
\begin{align*}
    \bvv D_{\mu,\F} - D_{\nu,\F} \bvv_{\infty,\K}
    \le \frac{\exponent \lambda}{m^{1/p}} \diam(\K)^\exponent~W_p(\mu,\nu)~,
\end{align*}
and if $\exponent < 1$ and both $\d_\mu$ and $\d_\nu$ are lower bounded by $\sigma > 0$, then
\begin{align*}
    \bvv D_{\mu,\F} - D_{\nu,\F} \bvv_{\infty,\K}
    \le \frac{\exponent \lambda}{m^{1/p}} \frac{\diam(\K)}{\sigma^{1-\exponent}} W_p(\mu,\nu)~.
\end{align*}
\end{theorem}

\begin{proof}
Consider two endpoints $x$ and $y$ in $\K$, and let $\gamma\in\Gamma_{\nu,\F}^\star(x,y)$ be a geodesic.
Since by definition $D_{\nu,\F}(x,y) = D_\nu(\gamma)$ and $D_{\mu,\F}(x,y) \le D_\mu(\gamma)$,
\begin{align*}
    D_{\mu,\F}(x,y) - D_{\nu,\F}(x,y)
    \le \int_\gamma \bv d_\mu^\exponent-d_\nu^\exponent\bv
    \le |\gamma| \bvv d_\mu^\exponent-d_\nu^\exponent\bvv_{\infty,\K}
    \le \lambda \diam(\K) \bvv d_\mu^\exponent-d_\nu^\exponent\bvv_{\infty,\K}~,
\end{align*}
where $\lambda>0$ is the hidden constant in \cref{corol:euclidean_bound_uniform}.
Switching the roles of $\mu$ and $\nu$ gives the same bound for the opposite quantity, which concludes the first inequality of \cref{thm:fdtm_stab_dtm}.
The upper bounds involving $W_p(\mu,\nu)$ follow from the $\kappa$-Lipschitz property of $t \mapsto t^\exponent$ previously discussed and from \cref{eq:dtm_stab_wasserstein}. Indeed,
\begin{align*}
    \bvv d_\nu^\exponent-d_\mu^\exponent\bvv_{\infty,\K}
    \le \kappa \bvv d_\nu-d_\mu\bvv_{\infty,\K}
    \le \frac{\kappa}{m^{1/p}} W_p(\mu,\nu)~.
\end{align*}
\end{proof}

\subsubsection{Stability with respect to the Domain (sketch of proof)}
\label{sec:stability_domain}

We state an intermediate version of \cref{thm:intro:fdtm_stab_domain}.
\begin{theorem}
\label{thm:fdtm_stab_domain_intermediate}
Let $\mu \in \cM_\K$ and $\F,\G \subset \K$ two closed domains.
Assume that any point in $\F$ is at distance from $\G$ at most
\begin{align*}
    \|d(\cdot,\G)\|_{\infty,\F}
    \le \frac{\diam(\K)^\exponent}{25\kappa}~,
\end{align*}
where $\kappa$ is the Lipschitz constant of $d_\mu^\exponent$ over $\K$ described previously.
For all $x,y\in\K$, if $q \ge 1$ then
\begin{align*}
    D_{\mu,\G}(x,y)
    \le D_{\mu,\F}(x,y)
        + 10\lambda \sqrt{\exponent} \diam(\K)^{q+\frac12} \sqrt{\|d(\cdot,\G)\|_{\infty,\F}}~,
\end{align*}
and if $q < 1$ then
\begin{align*}
    D_{\mu,\G}(x,y)
    \le D_{\mu,\F}(x,y)
        + 10\lambda \sqrt{\exponent} \frac{\diam(\K)^{\frac{q}{2}+1}}{(\min d_\mu)^{\frac{1-q}{2}}} \sqrt{\|d(\cdot,\G)\|_{\infty,\F}}~,
\end{align*}
where $\lambda$ is the hidden constant in \cref{corol:euclidean_bound_uniform}.
\end{theorem}

It is straightforward that \cref{thm:fdtm_stab_domain_intermediate} directly implies \cref{thm:intro:fdtm_stab_domain} when the Hausdorff distance $\|d(\cdot,\G)\|_{\infty,\F}$ is smaller than the given threshold.
In the other case, the constant $\lambda$ given in \cref{app:geodesic:details_domain} is larger than one so that being above the threshold implies that
\begin{align*}
    \diam(\K)^\exponent
    \le 25\kappa~\|d(\cdot,\G)\|_{\infty,\F}
    \le 100\kappa\lambda^2~\|d(\cdot,\G)\|_{\infty,\F}~.
\end{align*}
Then, considering the straight path $[x,y] \in \Gamma_\G(x,y)$ gives the following bound on the \ac{fdtm} by using \cref{eq:dtm_uniform_bound}:
\begin{align*}
    D_{\mu,\G}(x,y)
    \le \|x-y\| \bLp\max_\K d_\mu\bRp^\exponent
    \le \diam(\K)^{\exponent+1}
    \le 10\lambda \sqrt{\kappa} \diam(\K)^{\frac{\exponent}{2}+1} \sqrt{\|d(\cdot,\G)\|_{\infty,\F}}~.
\end{align*}
The desired upper bounds follow by replacing $\kappa$ with the appropriate quantity depending on $\exponent$.

We now describe the main reasoning behind \cref{thm:fdtm_stab_domain_intermediate}.
Given a geodesic $\gamma$ \ac{wrt} the domain $\F$, $\gamma$ is approximated by another path which is admissible for the domain $\G$.
This is done by cutting $\gamma$ into small bits, identifying the nearest neighbor in $\G$ at each step, and drawing a polygonal path between the resulting points.
The presence of long outer sections where points are not in $\F$ requires some caution. Indeed, the points outside $\F$ may not be approximated efficiently by points in $\G$.
In this case however, since the outer section must draw a straight line between two points in $\F$ by definition, the segment connecting their nearest neighbors in $\G$ still provides a good approximation of the geodesic.
This analysis is detailed in \cref{app:stability_domain} and eventually shows that if bits are chosen with a typical length $r$---barring long outer sections---then the resulting \ac{fdtm} offset is at most proportional to $r+\frac1r\max_{x\in \F} d(x,\G)$.
Therefore, $r$ is best chosen to scale as the square root of $\max_{x\in \F} d(x,\G)$, which eventually yields \cref{thm:fdtm_stab_domain_intermediate}.

\section{Estimating the FDTM}
\label{sec:estimation}

In this section, we study the special case where the domain is chosen as the support of the measure and the \ac{fdtm} $D_\mu = D_{\mu,\supp(\mu)}$ is estimated by sampling a set $\XX_n$ of $n$ points \ac{iid} according to $\mu$ and considering the empirical \ac{fdtm} $D_{\hat\mu_n} = D_{\hat\mu_n,\XX_n}$.

\subsection{Theoretical Convergence Rate}
\label{sec:convergence}

Applying the previous stability results to compare $D_{\hat\mu_n,\XX_n}$ to $D_{\mu,\supp(\mu)}$ allows to obtain explicit convergence rate of the empirical \ac{fdtm} under \cref{assum}.
In \cref{assum}.(ii), $b$ may be interpreted as an upper bound on the intrinsic dimension of the support of $\mu$. This is true in the sense of the Minkowski dimension, which coincides with standard notion of dimension if $\supp(\mu)$ is a sub-manifold of $\RR^d$.
Conversely, if $\mu$ is supported on a compact manifold of dimension $b$ with density bounded by below, then $\mu$ is $(a,b)$-standard where $a$ depends on the density lower bound and on the geometry of the manifold, see~\cite{niyogiFindingHomologySubmanifolds2008}.
Finally, note that if a measure $\mu$ is $(a,b)$-standard, then for all $b'\ge b$ it is also $(a',b')$-standard for small enough $a'$.\footnote{By letting $a' = a^{b'/b}$, one has $(a')^{-1/b'} = a^{-1/b}$ and $ar^b \ge a'r^{b'}$ for $0 < r \le (a')^{-1/b'}$.}
In order to control the convergence rate of the empirical \ac{fdtm} over $\K$, we decompose the offset into two terms as
\begin{align}
    \label{eq:estimation:decomposition}
    \bvv D_\mu - D_{\hat\mu_n} \bvv_{\infty,\K}
    \le \bvv D_{\mu,\supp(\mu)} - D_{\hat\mu_n,\supp(\mu)} \bvv_{\infty,\K}
        + \bvv D_{\hat\mu_n,\supp(\mu)} - D_{\hat\mu_n,\XX_n} \bvv_{\infty,\K}~.
\end{align}
Both terms can be further upper bounded using our stability results \cref{thm:fdtm_stab_dtm,thm:intro:fdtm_stab_domain}, involving $\|d_\mu^p-d_{\hat\mu_n}^p\|_\infty$ and $\hausdorff(\supp(\mu),\XX_n)$ respectively.
These two quantities then converge with known rate under \cref{assum}.(ii), namely
\begin{align}
    \label{eq:estimation:speeds}
    \EE\Lb\bvv d^p_\mu-d^p_{\hat\mu_n} \bvv_{\infty,\K}\Rb
    \lesssim \frac{\log(n)}{\sqrt{n}}
    \qquad\textrm{and}\qquad
    \EE\bLb\hausdorff(\supp(\mu),\XX_n)\bRb
    \lesssim \Lp\frac{\log(n)}{n}\Rp^{\frac{1}{b}}~.
\end{align}
The former is directly deduced from~\cite{chazalRatesConvergenceRobust2016} whereas the latter is a standard result.
Note that the conditions (iii) and (iv) in \cref{assum} are required only to obtain the convergence rate on the empirical \ac{dtm}.
Without these assumptions, one could instead study the convergence of the empirical measure in Wasserstein distance.
Indeed,~\cite{boissardMeanSpeedConvergence2014} showed that, assuming that $\mu$ is $(a,b)$-standard,
\begin{align*}
    \EE\bLb W_p(\mu,\hat\mu_n)\bRb
    \lesssim \max\Lp \frac{1}{n^{\frac{1}{2p}}},\frac{1}{n^{\frac{1}{b}}}\Rp~.
\end{align*}
This would however lead to a term of order $n^{-1/2p}$ in the final bound, which is slower than the statement below if $p>b$.
\Cref{eq:estimation:decomposition,eq:estimation:speeds} put together with \cref{thm:fdtm_stab_dtm,thm:intro:fdtm_stab_domain} allow for the following convergence rate which is a more detailed version of \cref{thm:intro:fdtm_estimate}.

\begin{restatable}{theorem}{fdtmEstimate}
\label{thm:fdtm_estimate}
Assume that $\mu \in \cM_{\K,a,b,\sigma}$ satisfies \cref{assum} and that $\frac1n \le m \le \frac12$.
Then, if $p \le \exponent$,
\begin{align*}
    \EE\Lb\bvv D_\mu - D_{\hat\mu_n} \bvv_{\infty,\K}\Rb
    \lesssim \diam(\K)^{\exponent+1} \Lp
        \diam(\K)^{-\frac12} \Lp\frac{\log(n)}{an}\Rp^{\frac{1}{2b}}
        + \frac{d}{\diam(\K) a^{\frac1b}} \frac{\log(n)}{\sqrt{n}}
        \Rp~.
\end{align*}
If $1 \le \exponent < p$,
\begin{align*}
    \EE\Lb\bvv D_\mu - D_{\hat\mu_n} \bvv_{\infty,\K}\Rb
    \lesssim \diam(\K)^{\exponent+1} \Lp
        \diam(\K)^{-\frac12} \Lp\frac{\log(n)}{an}\Rp^{\frac{1}{2b}}
        + d \frac{\diam(\K)^{p-\exponent-1}}{\sigma^{p-\exponent}a^{\frac1b}} \frac{\log(n)}{\sqrt{n}}
        \Rp~.
\end{align*}
If $\exponent < 1$,
\begin{align*}
    \EE\Lb\bvv D_\mu - D_{\hat\mu_n} \bvv_{\infty,\K}\Rb
    \lesssim \diam(\K)^{\exponent+1} \Lp
        \diam(\K)^{-\frac{q}{2}} \sigma^{\frac{\exponent-1}{2}} \Lp\frac{\log(n)}{an}\Rp^{\frac{1}{2b}}
        + d \frac{\diam(\K)^{p-\exponent-1}}{\sigma^{p-\exponent}a^{\frac1b}} \frac{\log(n)}{\sqrt{n}}
        \Rp~.
\end{align*}
In each of the above three cases, $\lesssim$ hides a multiplicative constant depending on $m$, $p$, and $\exponent$.
\end{restatable}
Details on the above discussion are provided in \cref{app:estimation}.

\subsection{Minimax Lower Bound}
\label{sec:minimax}

In this section we describe the reasoning to get the minimax lower bound stated in \cref{thm:intro:minimax}, which is based on Le Cam's Lemma (see for instance~\cite{yuAssouadFanoCam}).
This lemma essentially states that if two measures $\mu_1,\mu_2 \in \cM_{\K,a,b,\sigma}$ are sufficiently close in total variation distance while maintaining a significant \ac{fdtm} offset, then any estimator of the \ac{fdtm} cannot distinguish between both \acp{fdtm} up to a certain amount of samples.
This provides a worse case lower bound on the convergence rate of any estimator of the \ac{fdtm} for the class of measure $\cM_{\K,a,b,\sigma}$.
We now describe an example of such measures such that Le Cam's Lemma yields the lower bound of order $n^{-1/b}$ stated in \cref{thm:intro:minimax}.

\begin{example}
\label{ex:lecam_b}
Let $x$ and $y$ be the basis vectors of $\RR^2$.
Given $\alpha\in(0,1)$, $r>0$, and $0<\epsilon<(1-\alpha)^{1/b}$, let
\begin{align*}
    \mu
    = m\alpha\delta_{-ry} + m(1-\alpha)\delta_{ry} + (1-m)\rho
    \qquad\textrm{and}\qquad
    \nu
    = \mu - m\epsilon^b\delta_{ry} + m\lambda~,
\end{align*}
where $\rho$ is the uniform probability measure over $[3ry,4ry]$, and $\lambda$ has density $z \mapsto b\|z-(r-\epsilon)y\|^{b-1}$ \ac{wrt} the Lebesgue measure on $[(r-\epsilon)y, ry]$, amounting to a total mass $\lambda\bLp[(r-\epsilon)y, ry]\bRp = \epsilon^b$.
Then, there exists a choice of $\alpha\le\frac12$ and $r\le\frac14$ depending on $p$ and $\exponent$ along with positive constants $a$, $c$, and $C$ depending on $m$, $p$, and $\exponent$ such that for all $0 < \epsilon \le c$, $\mu$ and $\nu$ are both $(a,b)$-standard, satisfy $\dtv(\mu,\nu) = m\epsilon^b$, and
\begin{align*}
    D_\mu(-x,x) - D_\nu(-x,x)
    \ge C\epsilon~.
\end{align*}
\end{example}

\begin{figure}
    \centering
    \begin{tikzpicture}[scale=2]
        \def\r{0.8}
        \def\t{0.25}
        
        \coordinate (X) at (-1.5,0);
        \coordinate (Y) at (1.5,0);
        \coordinate (Z) at (0,\r);
        \coordinate (W) at (0,-\r);
        \coordinate (Z') at (0,\r-\t);

        \draw[thick] (X) ++(-0.05,-0.05) -- ++(0.1,0.1);
        \draw[thick] (X) ++(-0.05,0.05) -- ++(0.1,-0.1);
        \draw[thick] (Y) ++(-0.05,-0.05) -- ++(0.1,0.1);
        \draw[thick] (Y) ++(-0.05,0.05) -- ++(0.1,-0.1);
        \draw[fill] (Z) circle (0.8pt);
        \draw[fill] (W) circle (0.8pt);
        
        \draw[dotted, thick] (Z) -- (Z') node[midway, right, xshift=-3pt, yshift=-1pt] {$\lambda$};
        
        \tikzset{midarrow/.style={decoration={markings, mark=at position 0.5 with {\arrow{>}}}, postaction={decorate}}}
        \draw[midarrow] (X) -- (Y);
        \draw[midarrow] (X) -- (Z);
        \draw[midarrow] (Z) -- (Y);
        \draw[midarrow, thick] (X) -- (Z');
        \draw[midarrow, thick] (Z') -- (Y);\draw[midarrow] (X) -- (W);\draw[midarrow] (W) -- (Y);
        
        \node[above left] at (X) {$-x$};
        \node[above right] at (Y) {$x$};
        \node[above] at (Z) {$ry$ (with mass $1-\alpha$)};
        \node[below] at (W) {$-ry$ (with mass $\alpha$)};
    \end{tikzpicture}
    \caption{\cref{ex:lecam_b} in the case $m=1$ with the main admissible paths. The dotted line represents the addition of a small density in $\nu$ which allows for a new shorter path (drawn thicker) that is shorter which is the reason for the significant offset in \ac{fdtm}.}
    \label{fig:lecam_b}
\end{figure}

\Cref{fig:lecam_b} illustrates the reasoning behind \cref{ex:lecam_b}, which is to add the smallest amount of mass allowed by the $(a,b)$-standard assumption to enable a shorter path that drastically changes the \ac{fdtm}.
The presence of a small atom of mass $\alpha$ is necessary to ensure that the shorter path admissible only for $\nu$ is indeed shorter \ac{fdtm}-wise.
The $\rho$ part of the measure does not play a significant role besides simply holding the rest of the mass far away from the setup of the example so that it does not interfere with the reasoning. It is not defined as an atom to ensure the assumption that there are no atoms of mass greater than $m$. 

Notice that the measures in \cref{ex:lecam_b} do not have connected support.
However, a curve of small enough density may be added to connect all components of the support, taking sufficiently long detours away from the origin so that it does not enable shorter paths. Provided that this curve has a mass small enough that it changes the \ac{dtm} by a negligible amount, the result still holds with different constants.
The technical computations required to establish the inequality in \cref{ex:lecam_b} are omitted here, as they do not involve any particularly deep arguments.

The assumption that $\K$ is not a segment in \cref{thm:intro:minimax} ensures due to convexity that one can find a $2$-dimensional disc within $\K$ in which one can fit the example measures from \cref{ex:lecam_b}.
Moreover, the dependency in $\|x-y\|^{\exponent+1}$ comes naturally from the scaling behavior of the \ac{fdtm} as discussed in \cref{app:restriction_compact}.
\Cref{thm:intro:minimax} is then obtained using Le Cam's Lemma with \cref{ex:lecam_b} and letting $\epsilon$ be of order $n^{-1/b}$ so that $\dtv(\mu_1,\mu_2)$ is of order $n^{-1}$.
In the case where $b=1$, the lower bound of order $n^{-1/2}$ can be obtained by considering two atoms on a line with a uniform density on the segment between them, and creating a second measure by shifting a small amount of mass from one atom to another.

\subsection{Practical Computations and Heuristics}
\label{sec:algorithm}

Regarding the computational complexity of the estimator, notice that the restriction of $D_{\hat\mu_n}$ to the sample points $\XX_n$ is the metric of a complete weighted graph with vertices $\XX_n$ and weights
\begin{align*}
    w_{x,y}
    = \int_{[x,y]} d_{\hat\mu_n}
    = \|x-y\| \int_0^1 d_{\hat\mu_n}\bLp(1-t)x + ty\bRp^\exponent \d t~,
\end{align*}
where, letting $k = \lfloor mn\rfloor$ and $d^i(z,\XX_n)$ be the distance from $z$ to its $i$-th neighbor in $\XX_n$,
\begin{align*}
    d_{\hat\mu_n}(z)^p
    = \frac1m \int_0^m \delta_u(z)^p \d u
    = \frac1m \Lp\sum_{i=1}^{k} \frac1n \bLp d^i(z,\XX_n)\bRp^p + \frac{mn-k}{n} \bLp d^{k+1}(z,\XX_n)\bRp^p\Rp~.
\end{align*}
While theoretically possible, computing such metric is unreasonable in practice, hence we propose a few methods to reduce time complexity while ensuring that the limiting object remains the same.

\medskip\noindent
\textit{Selecting edges.}
When considering the complete graph over $\XX_n$, the quadratic amount of edges affects both the computation time of weights and the shortest path algorithm.
In the context of classic Fermat distance, the graph edges may be restricted to a subset of order $n\log(n)$ by pruning the graph, connecting each vertex to its $\log(n)$ closest neighbors.
It is indeed guaranteed that with high probability such procedure does not affect the Fermat distance as $n$ grows larger, see~\cite[Proposition 2.16]{groismanNonhomogeneousEuclideanFirstpassage2022}.
This trick however cannot be performed in the context of \ac{fdtm} without altering the limiting object as it would effectively restrict paths to remain constrained within the support, which may exclude geodesics. If it is known that no geodesic exits the support, \eg, when the latter is convex, such method may be applied.

A similar method that can be used in general is to use a graph spanner that restrict the edges while preserving some edges to span all directions. This is done by Yao graphs, described in~\cite{yaoConstructingMinimumSpanning1982}, which divide the space from each vertex into a set of cones and keeps the shortest edges in each cone, effectively keeping one edge in each possible direction.
Building a $k$-nearest neighbors graph using a kd-tree structure has complexity $\cO(n\log(n)k)$, whereas building a $k$-directions Yao graph has complexity $\cO(n^2)$ in general, but can be computed in $\cO(n\log(n)k)$ in dimension $d=2$ according to~\cite{funkeEfficientYaoGraph2023}.

Both the pruned graph and the Yao graph are sub-graphs with $\cO(nk)$ edges.
Yao graphs are an efficient method when $d=2$ as it is fast and do not restrict the paths to the support.
However, it does not scale well with dimension, hence a nearest neighbors method may be wiser in higher dimension, at the cost of potentially altering the limit if geodesics exit the support.

\medskip\noindent
\textit{Computing weights.}
The exact formula for $w_{x,y}$ can be first simplified by rounding $m$ to $\lfloor mn\rfloor/n$ and replacing the integral with the discrete approximation of the segment
\begin{align}
    \label{eq:weights}
    w_{x,y}
    = \frac{\|x-y\|}{r} \sum_{t=1}^r \Lp\frac{1}{\lfloor mn \rfloor} \sum_{i=1}^{\lfloor mn \rfloor} d^i(x_t,\XX_n)^p \Rp^{\frac{\exponent}{p}}~,
\end{align}
where $x_1, \dots, x_r$ is a regular subdivision of the segment $[x,y]$.

\medskip\noindent
\textit{Computing geodesics.}
Computing all geodesics from one source has a complexity of order $\cO((|E|+n)\log(n))$ using Dijkstra algorithm, where $|E|$ is the amount of edges.
Using the previous edge selection methods, this yields a complexity of $\cO(n\log(n)k)$.

\medskip\noindent
\textit{Overall complexity.}
Recall that \cite{taupinEstimationConformalMetrics2026} studies a similar estimation problem where paths are constrained to the support of the measure. Regarding the choice of the above nearest-neighbors parameter $k$ and resolution $r$ used for the discrete approximation of edges, it was shown in this setup that setting both $k$ and $r$ to scale as $\sqrt{n}$ is large enough to ensure the error associated with both approximations is negligible.
\Cref{fig:knn_comparison} displays the influence of the parameter $k$ on the estimated distance. It appears that on top of being faster, scaling $k$ as $\sqrt{n}$ yields a better approximation error than using the full graph. This suggests that while the empirical \ac{dtm} converges, it is typically slightly biased, and its bias is compensated by the pruned graph choosing suboptimal paths.

With the above choice of parameters $k$ and $r$, there are $\cO(n^{3/2})$ edges in the graph, selected in $\cO(n^2)$ time, and each weight is computed in $\cO(n^{3/2})$ time.
The total complexity of computing the \ac{fdtm} for all pairs of endpoints is therefore $\cO(n^3)$.
As a comparison, the complexity achieved for the Fermat distance in~\cite{groismanNonhomogeneousEuclideanFirstpassage2022} is $\cO(n^2\log(n)^2)$, which is faster. The same complexity may be achieved for the \ac{fdtm} by reducing parameters $k$ and $r$ to $\log(n)$, however \cref{fig:knn_comparison_error_runtime} shows that it is less efficient even accounting for the time complexity.
Indeed, for a given computation time it appears that running the algorithm on a graph with less sample points but more edges yields a better approximation.

\begin{figure}[tb]
    \centering
    \begin{minipage}{0.48\textwidth}
        \centering
        \includegraphics[width=\linewidth]{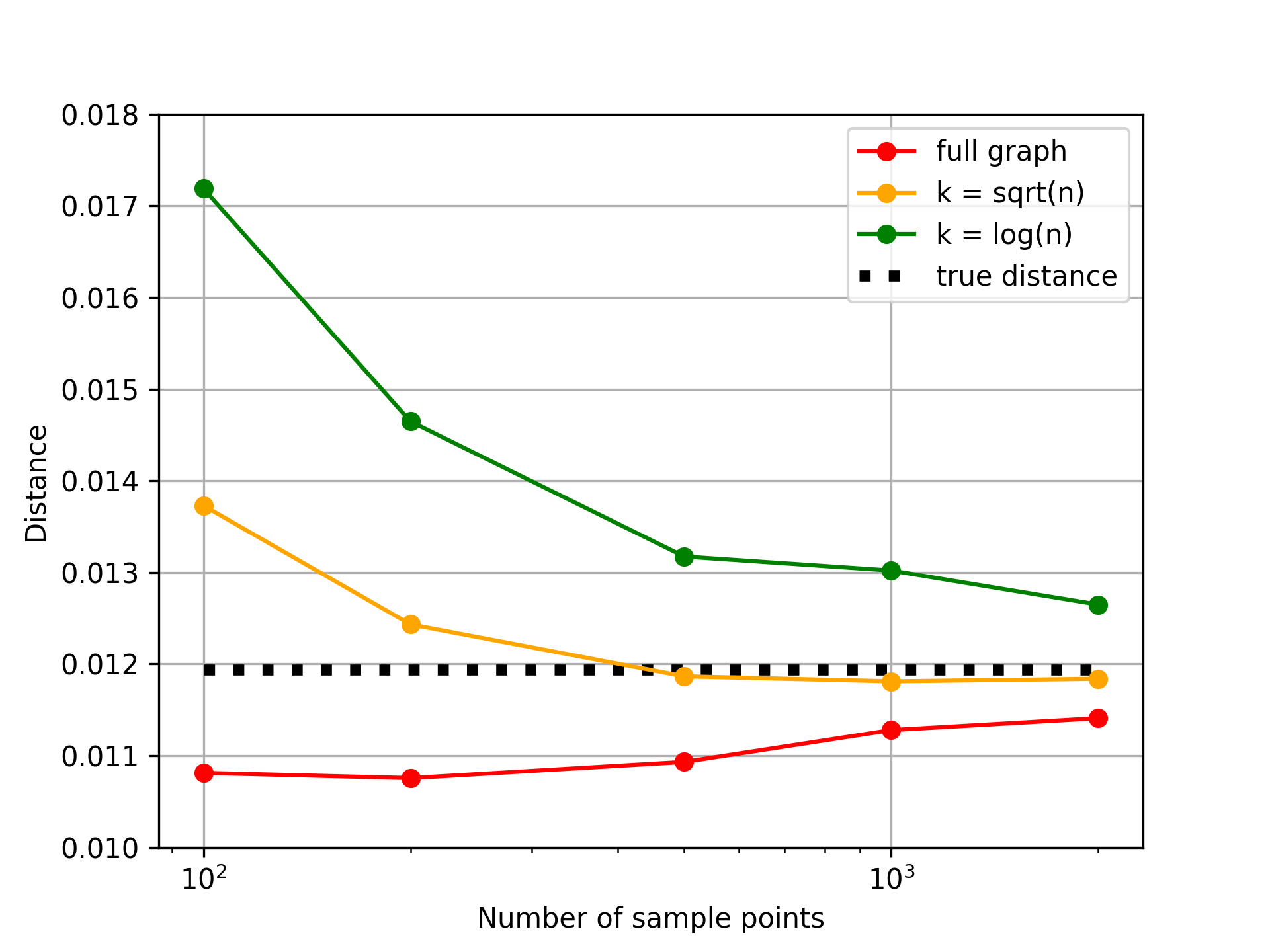}
        \subcaption{Estimated distance}
        \label{fig:knn_comparison_distance}
    \end{minipage}
    \begin{minipage}{0.48\textwidth}
        \centering
        \includegraphics[width=\linewidth]{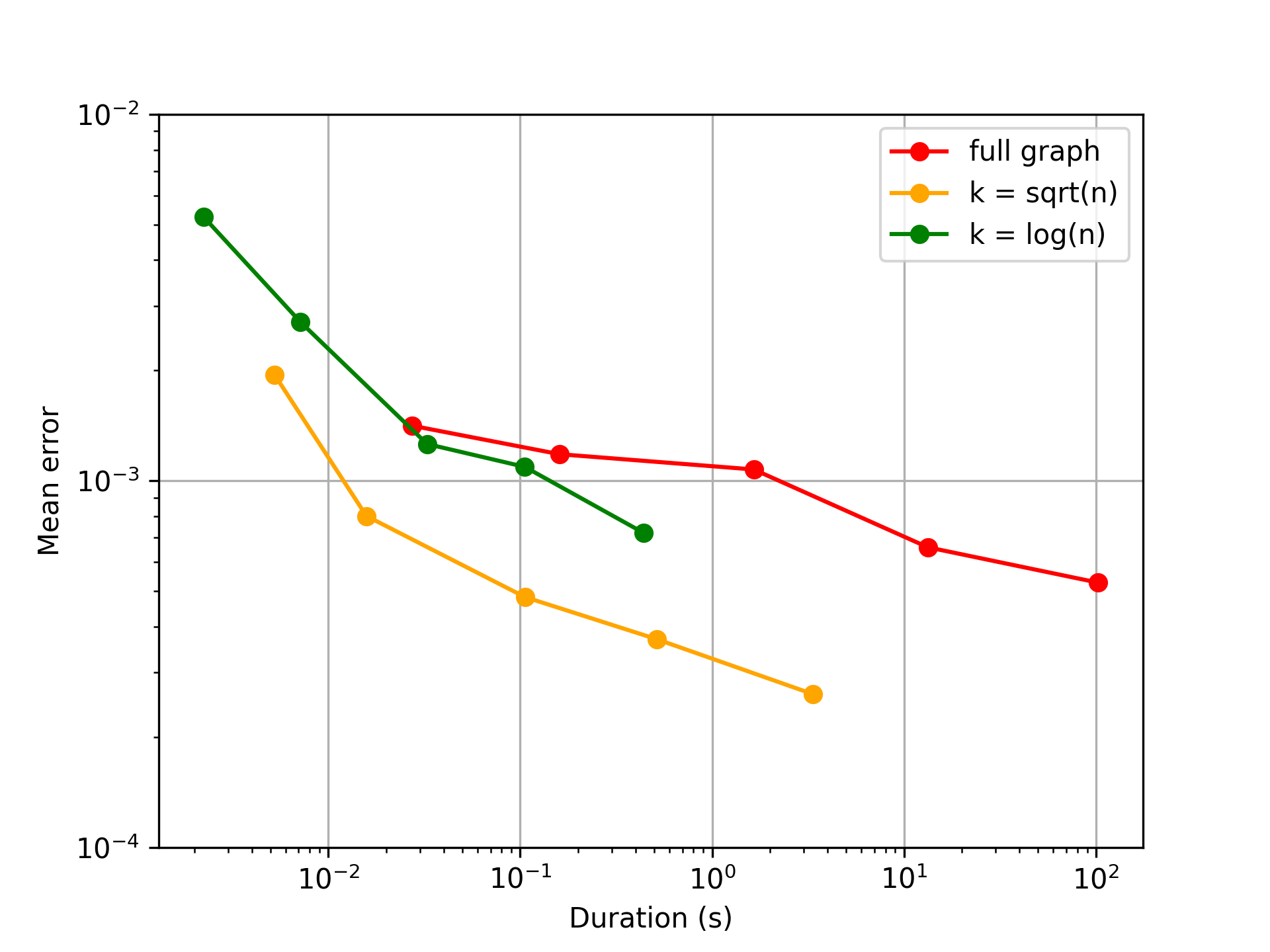}
        \subcaption{Mean error and runtime comparison}
        \label{fig:knn_comparison_error_runtime}
    \end{minipage}
    
    \caption{Estimated \ac{fdtm} and associated mean error and runtime using the full graph, the $k$-nn graph with $k = \sqrt{n}$, and with $k = \log(n)$.
    The \ac{fdtm} parameters are $m=0.05$, $p=2$, $\exponent=2$.
    The underlying measure is the uniform measure over $[-1,1]\times[-1/4,1/4] \subset \RR^2$ and the distance is computed between $(-3/4,0)$ and $(3/4,0)$.
    Plots are averaged over $200$ iterations, except for the full graph plot, averaged over $10$ iterations.}
    \label{fig:knn_comparison}
\end{figure}

\medskip\noindent
\textit{Algorithm summary.}
The proposed algorithm for the computation of the \ac{fdtm} may be summarized as follows:
\begin{itemize}
    \item Sample $n$ points \ac{iid} according to the measure $\mu$.
    \item Compute the Yao graph or nearest-neighbors graph with parameter $k = \sqrt{n}$.
    \item Compute all weights according to \cref{eq:weights} with $r = \sqrt{n}$.
    \item Run a shortest-path algorithm such as Dijkstra to compute geodesics and their length.
\end{itemize}

\section{Numerical Illustrations}
\label{sec:simulations}

Finally, we display a few simulations of the \ac{fdtm} on different setups to illustrate its general behavior.
Computations were made using the \textsc{Gudhi} library (see \url{https://gudhi.inria.fr/}), which provides an efficient way to compute the \ac{dtm}.

\subsection{Convergence of the Empirical FDTM on the Unit Circle}

\begin{figure}[tb]
    \centering
    \begin{minipage}{0.25\textwidth}
        \centering
        \includegraphics[width=\linewidth]{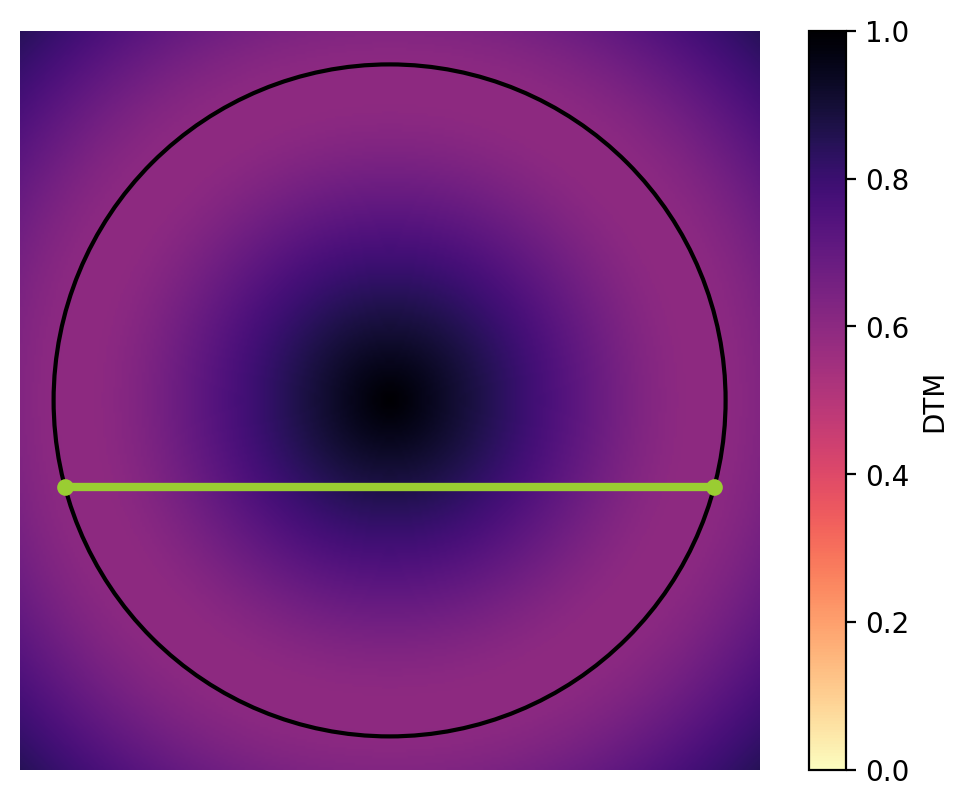}
        \subcaption{$m=0.2$, $\exponent=0.5$}
    \end{minipage}
    \begin{minipage}{0.25\textwidth}
        \centering
        \includegraphics[width=\linewidth]{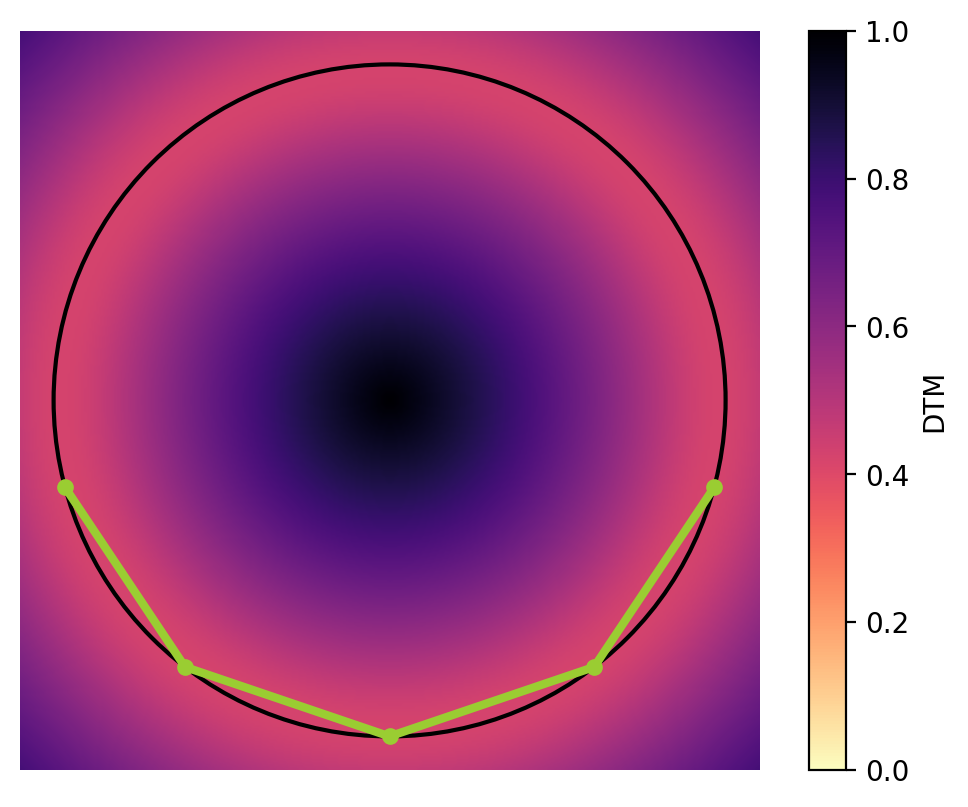}
        \subcaption{$m=0.1$, $\exponent=0.5$}
    \end{minipage}
    \begin{minipage}{0.25\textwidth}
        \centering
        \includegraphics[width=\linewidth]{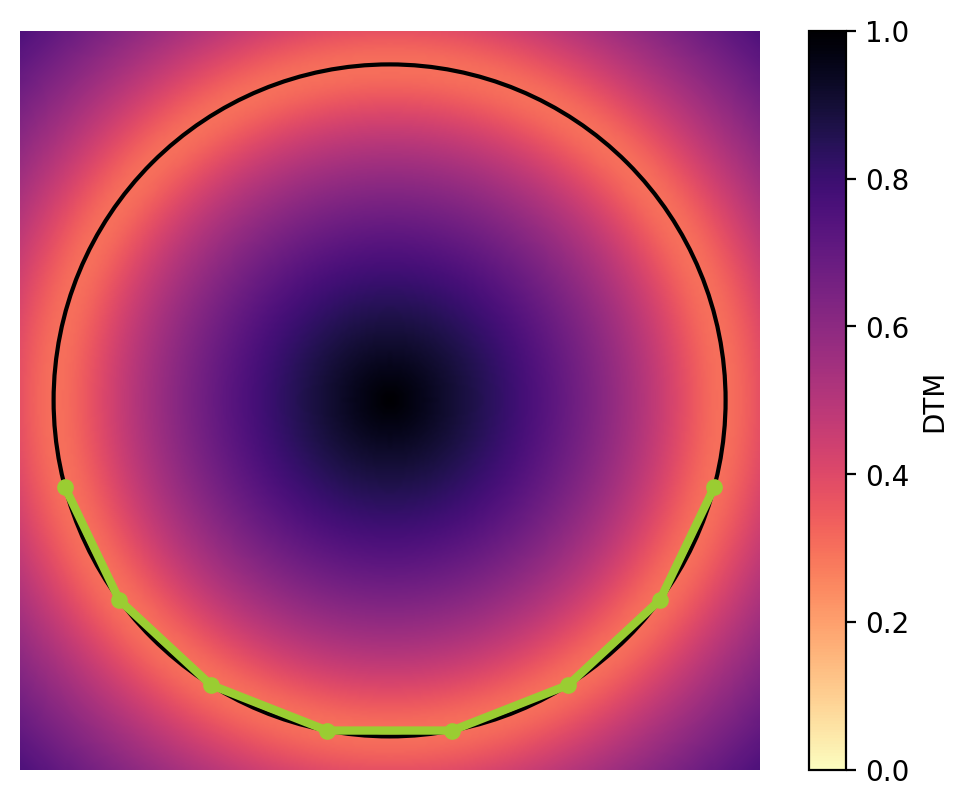}
        \subcaption{$m=0.05$, $\exponent=0.5$}
    \end{minipage}
    
    \vspace{0.3cm}

    \begin{minipage}{0.25\textwidth}
        \centering
        \includegraphics[width=\linewidth]{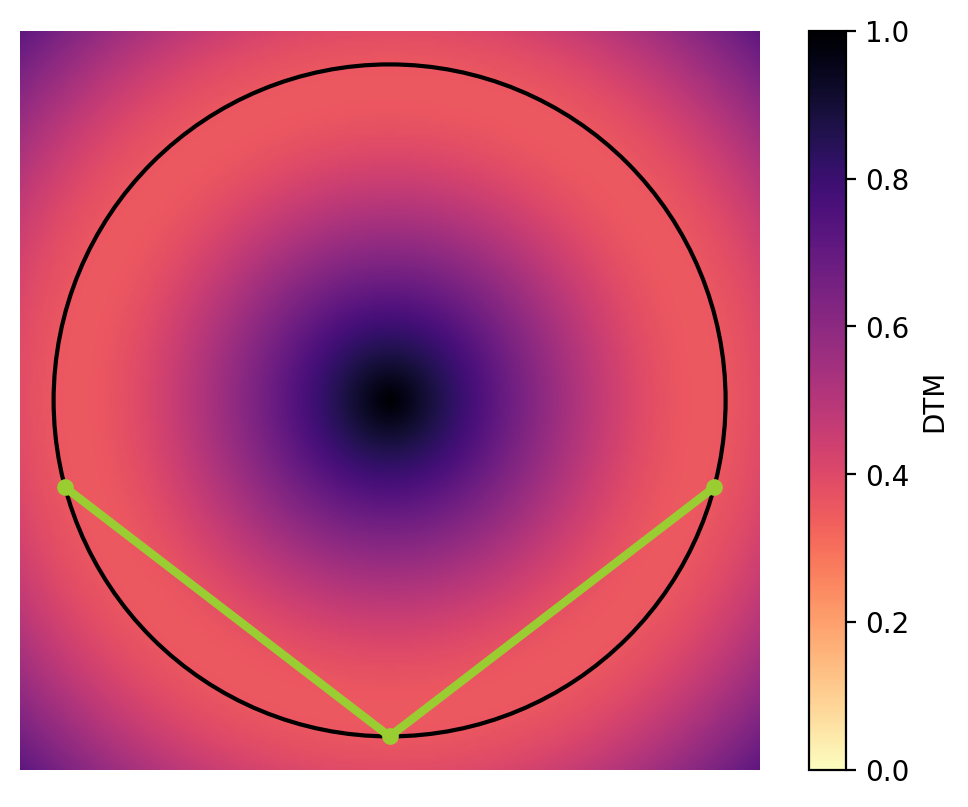}
        \subcaption{$m=0.2$, $\exponent=1$}
    \end{minipage}
    \begin{minipage}{0.25\textwidth}
        \centering
        \includegraphics[width=\linewidth]{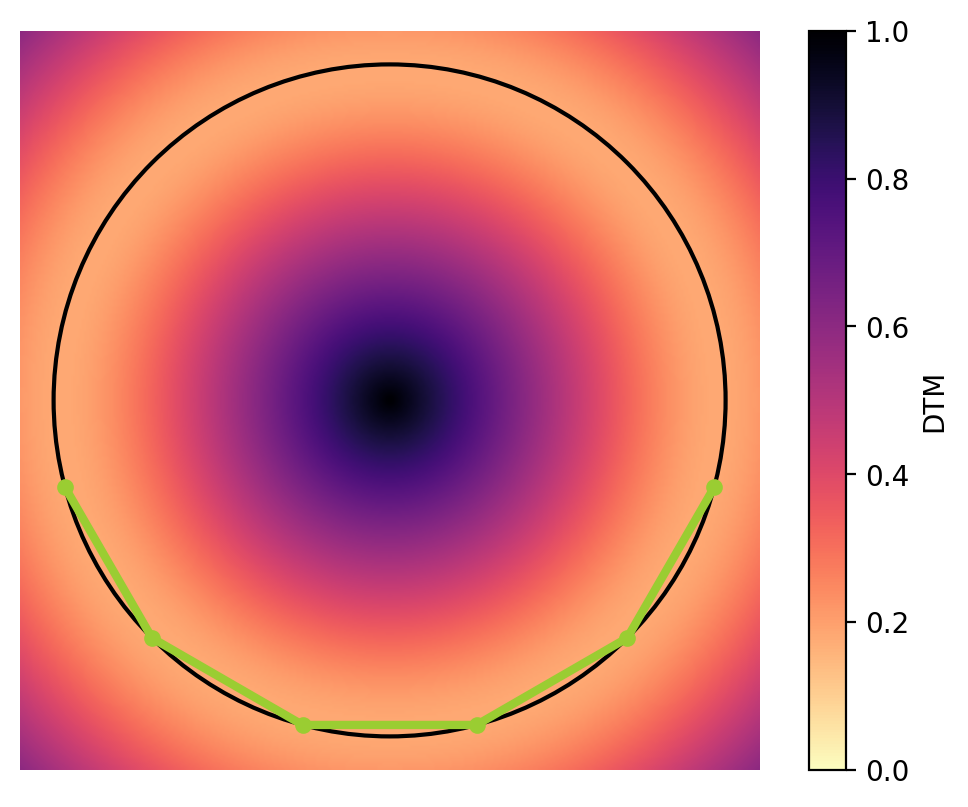}
        \subcaption{$m=0.1$, $\exponent=1$}
    \end{minipage}
    \begin{minipage}{0.25\textwidth}
        \centering
        \includegraphics[width=\linewidth]{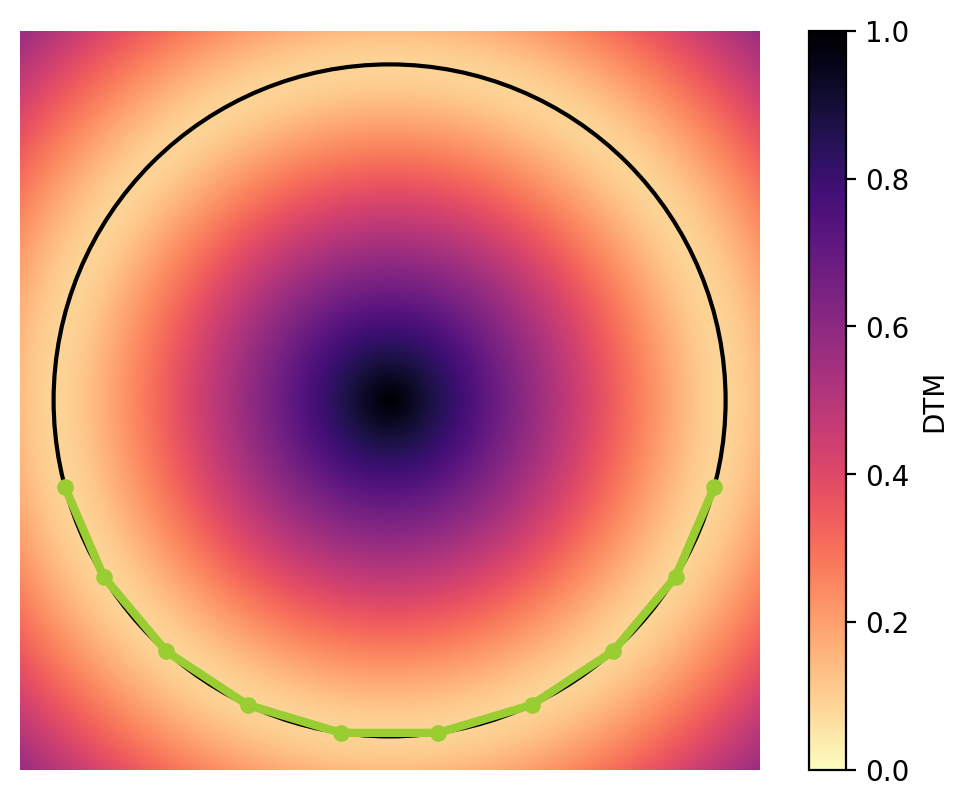}
        \subcaption{$m=0.05$, $\exponent=1$}
    \end{minipage}
    
    \vspace{0.3cm}
    
    \begin{minipage}{0.25\textwidth}
        \centering
        \includegraphics[width=\linewidth]{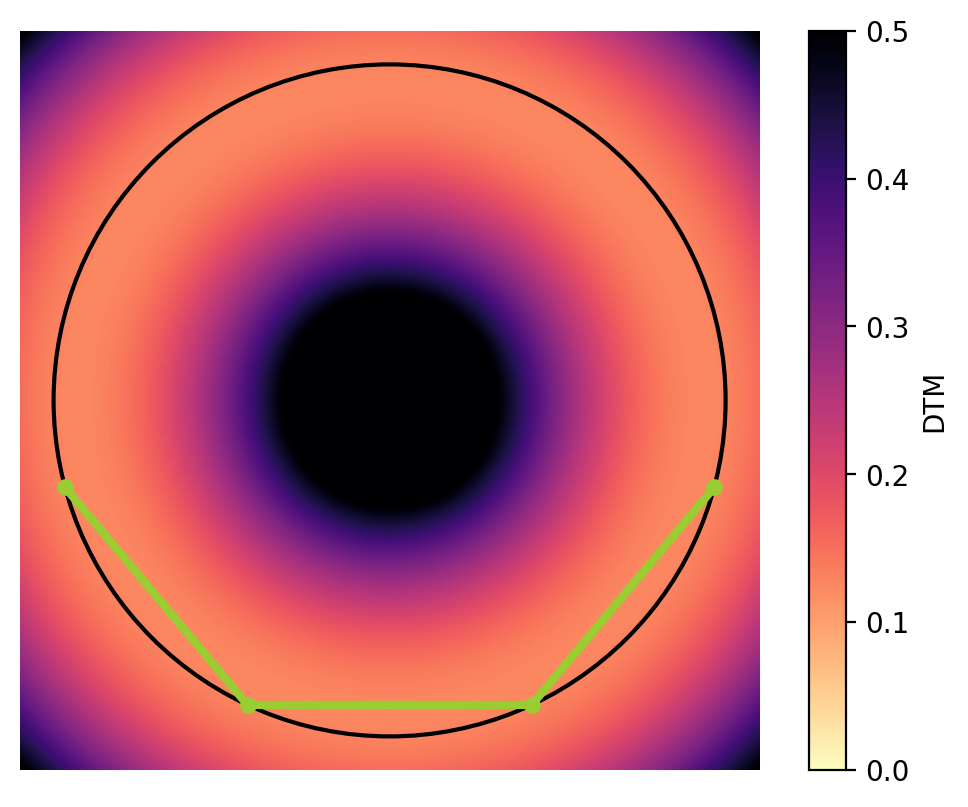}
        \subcaption{$m=0.2$, $\exponent=2$}
    \end{minipage}
    \begin{minipage}{0.25\textwidth}
        \centering
        \includegraphics[width=\linewidth]{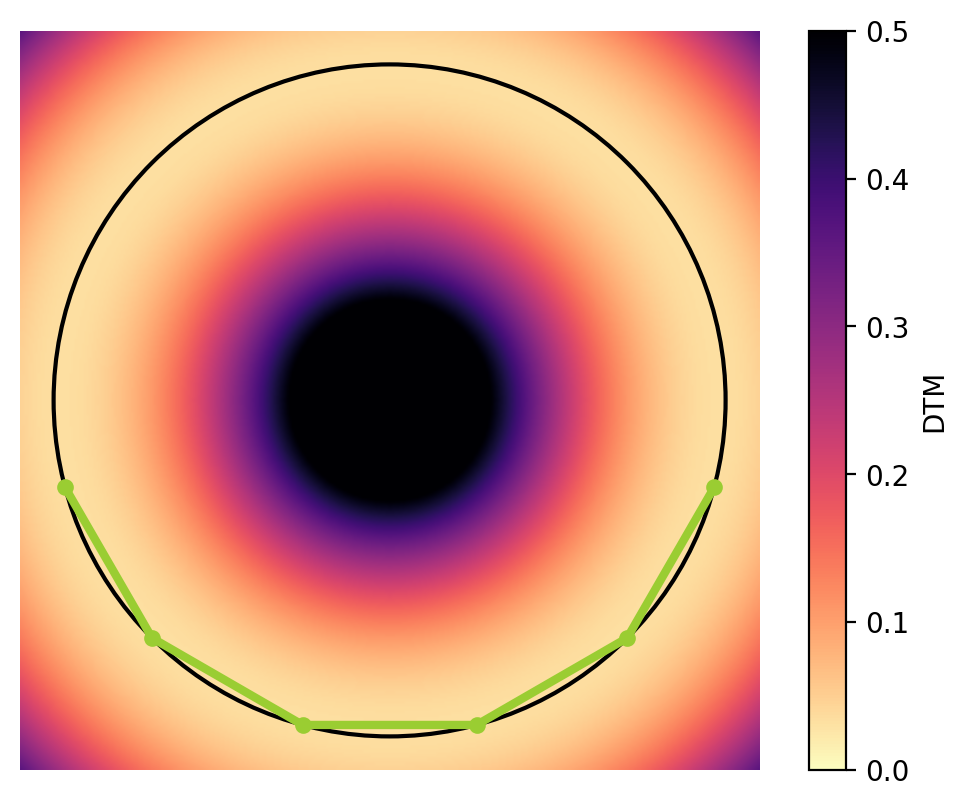}
            \subcaption{$m=0.1$, $\exponent=2$}
    \end{minipage}
    \begin{minipage}{0.25\textwidth}
        \centering
        \includegraphics[width=\linewidth]{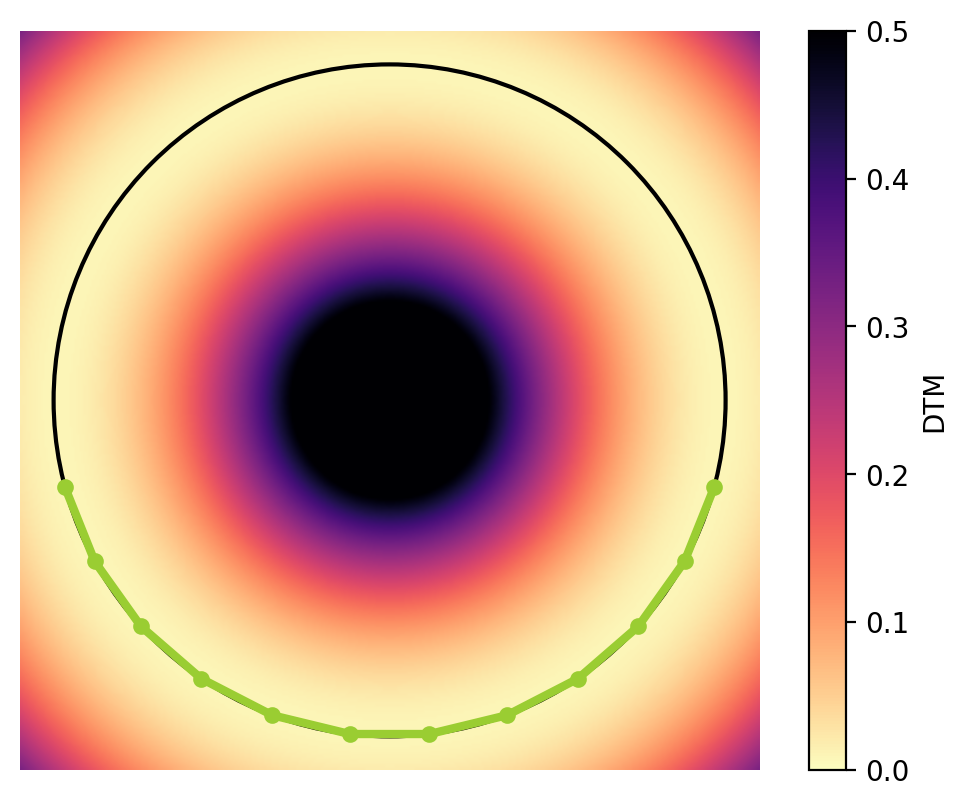}
        \subcaption{$m=0.05$, $\exponent=2$}
    \end{minipage}
    
    \caption{An example geodesic on the unit circle for $p=2$ and various values of $m$ and $\exponent$.
    Background colors represent the values of the \ac{dtm} cropped to a relevant range.
    The smaller $m$ or the higher $\exponent$, the closer geodesics stay to the support circle.}
    \label{fig:circle_geodesics}
\end{figure}

\begin{figure}
    \centering
    \includegraphics[width=0.6\linewidth]{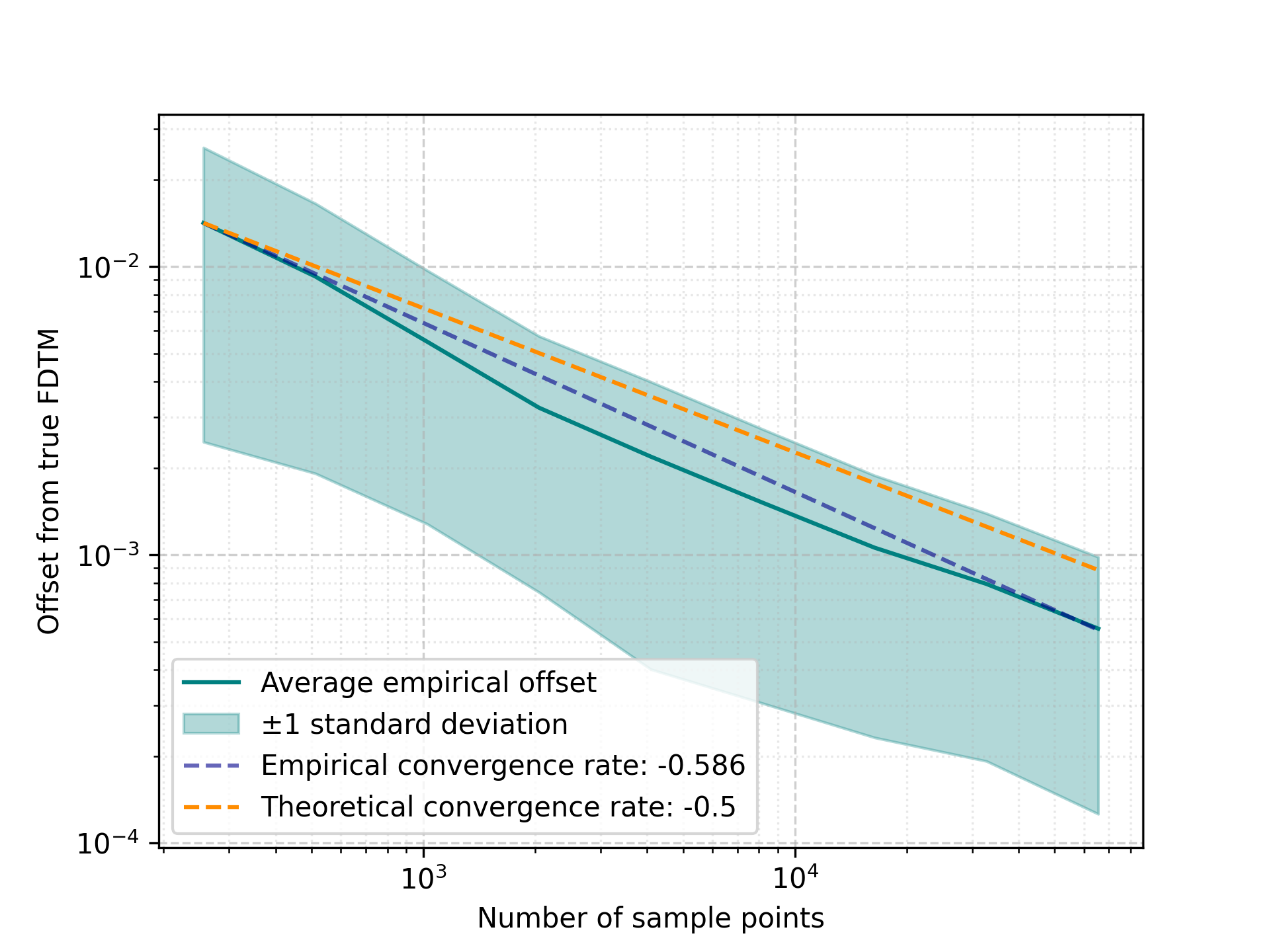}
    \caption{Convergence of the empirical \ac{fdtm} on the unit circle compared to the theoretical rate $n^{-0.5}$.
    The \ac{fdtm} parameters are $m=0.1$, $p=2$, and $\exponent=2$. The empirical \ac{fdtm} is averaged over $500$ iterations of sampling the points.
    The empirical convergence rate is computed by linear regression of the log-log plot.}
    \label{fig:circle_convergence}
\end{figure}

Consider $\mu$ the uniform distribution on the unit circle in $\RR^2$.
In this specific case, the \ac{fdtm} can be explicitly computed and geodesics between two endpoints on the circle are always made up of finite number of equally sized cords---see \cref{fig:circle_geodesics}.
We see  in \cref{fig:circle_convergence} that the theoretical upper bound from \cref{thm:intro:fdtm_estimate} (in this case with $b=1$) is coherent with the experiment.

\subsection{Comparison with the Fermat Distance}

\begin{figure}[tb]
    \centering
    \begin{minipage}{0.32\textwidth}
        \centering
        \includegraphics[width=\linewidth]{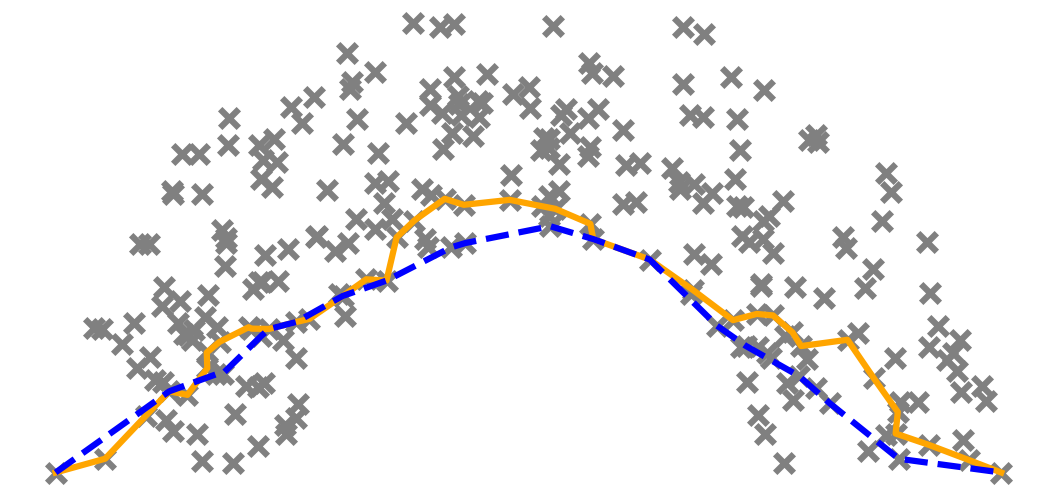}
        \subcaption*{$256$ sample points}
    \end{minipage}
    \begin{minipage}{0.32\textwidth}
        \centering
        \includegraphics[width=\linewidth]{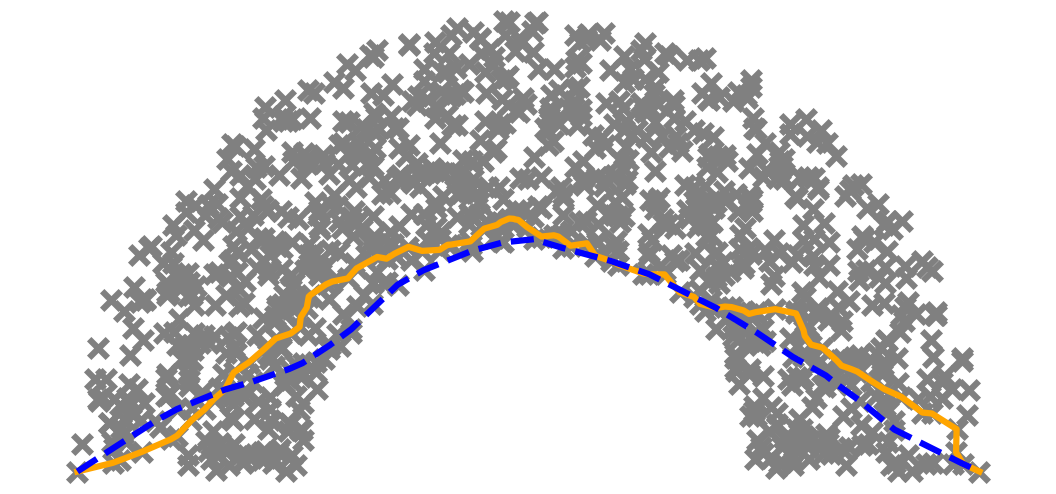}
        \subcaption*{$1024$ sample points}
    \end{minipage}
    \begin{minipage}{0.32\textwidth}
        \centering
        \includegraphics[width=\linewidth]{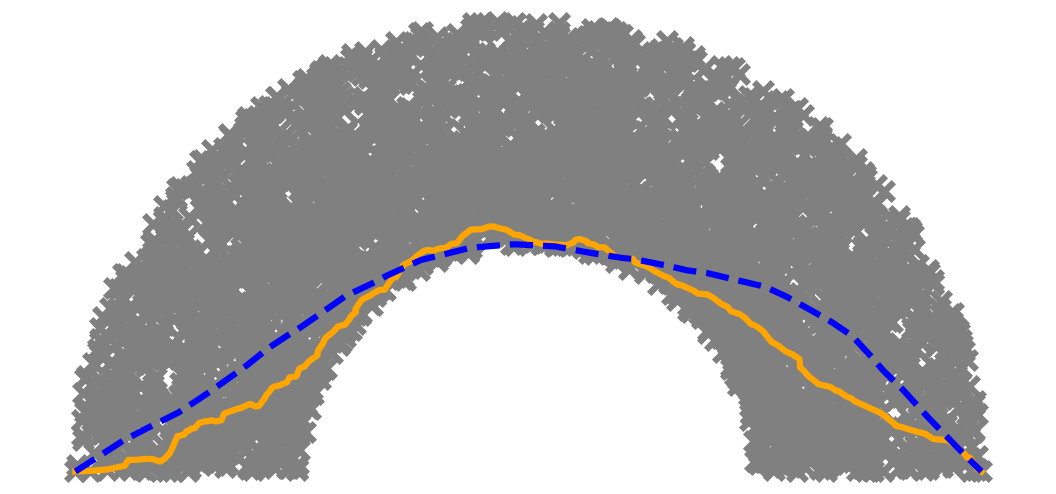}
        \subcaption*{$4096$ sample points}
    \end{minipage}
    
    \vspace{0.25cm}
    
    \begin{minipage}{0.32\textwidth}
        \centering
        \includegraphics[width=\linewidth]{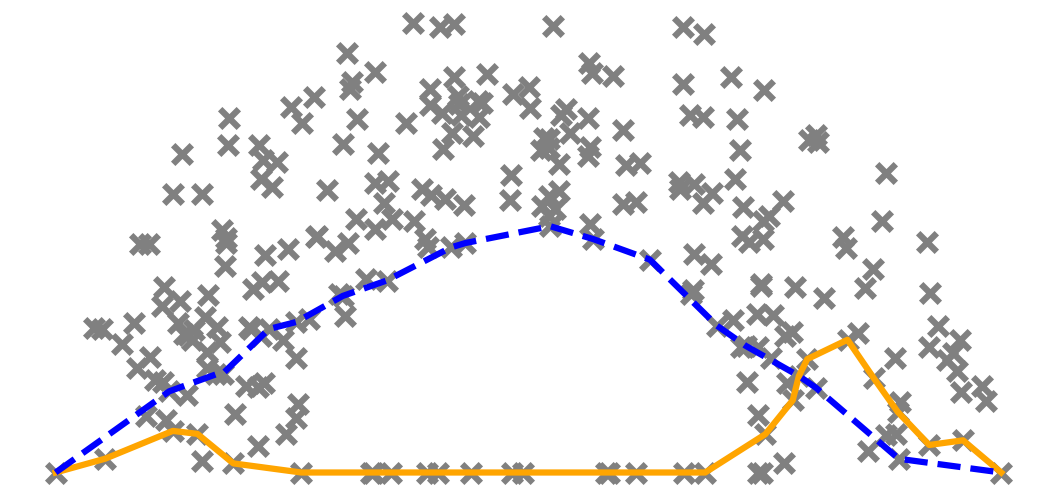}
    \end{minipage}
    \begin{minipage}{0.32\textwidth}
        \centering
        \includegraphics[width=\linewidth]{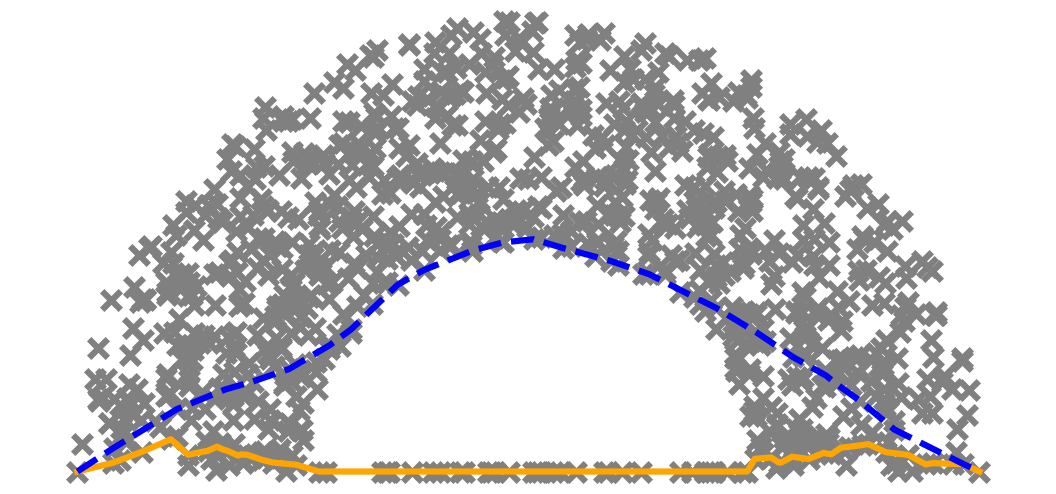}
    \end{minipage}
    \begin{minipage}{0.32\textwidth}
        \centering
        \includegraphics[width=\linewidth]{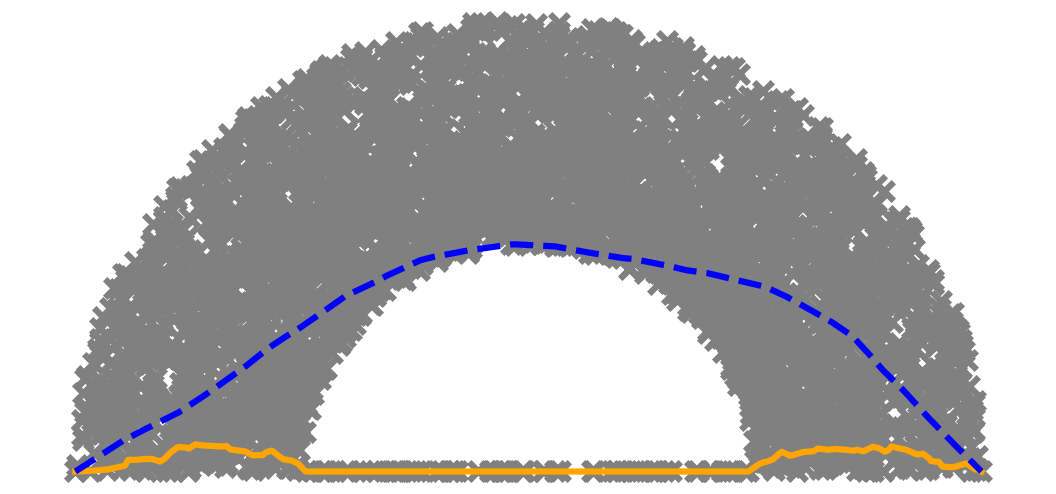}
    \end{minipage}
    
    \caption{Sample Fermat (orange solid line) and \ac{fdtm} (blue dashed line) geodesics.
    The sample Fermat parameter is $\alpha=1.25$ and the \ac{fdtm} parameters are $m=0.05$, $p=2$, $\exponent=0.25$.}
    \label{fig:ring_geodesics}
\end{figure}

\begin{figure}
    \centering
    \includegraphics[width=0.6\linewidth]{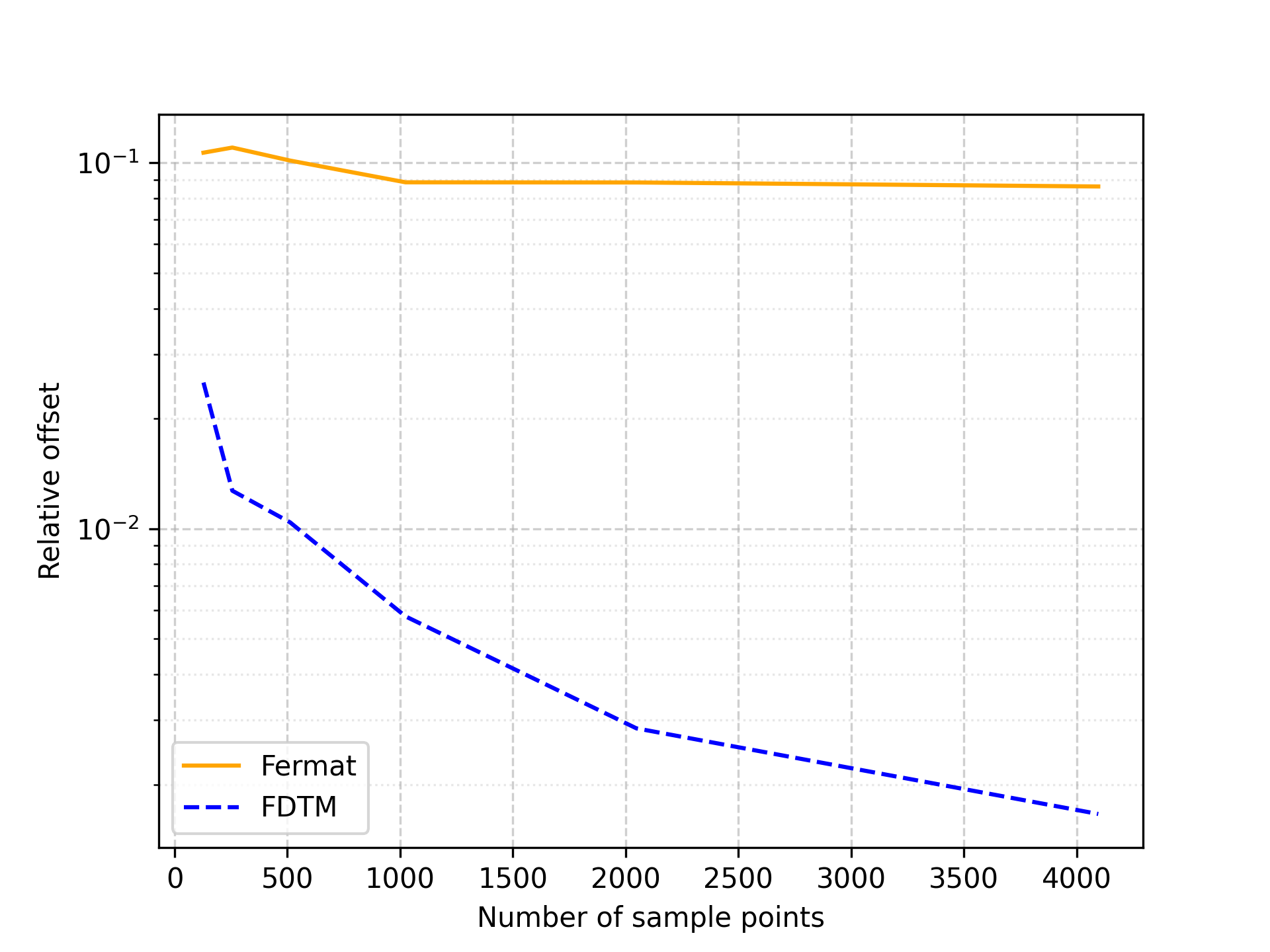}
    \caption{Relative offset in distance due to the addition of a shortcut.
    Each point is obtained by averaging the offset over $200$ random datasets.}
    \label{fig:ring_offset}
\end{figure}

One shortcoming of the Fermat distance is its requirement to have a density supported on a manifold of fixed dimension. As a consequence, while the sample Fermat distance can be defined for any point cloud, there are no guarantees on its convergence if the underlying measure is a mixture of measures of different dimensions.

\Cref{fig:ring_geodesics} provides an example of a measure that can be modified to ``trick'' the sample Fermat distance by adding a shortcut of negligible mass but lower dimensionality, allowing for a more efficient path.
Recall that the sample Fermat distance uses edge weights edges weights $w_{x,y} = \|x-y\|^\alpha$ for some $\alpha>1$.
As the \ac{fdtm} is based on the mass instead of the density, it is not affected by this shortcut, which highlights its robustness.
In each column, the bottom point cloud is obtained by shifting approximately $\sqrt{n}$ points (where $n$ is the total amount of points) among the top point cloud to create a shortcut in the middle. This means that the Total Variation distance between both point clouds vanishes as $n$ grows larger, but the typical distance between points is the same on the shortcut as on the ring so that the sample Fermat distance favors the shortcut.
This illustrates the instability of the geodesics for the Fermat distance.

The Fermat distance itself is also unstable.
Indeed, \cref{fig:ring_offset} plots the relative offset of the distance caused by the addition of the shortcut, \ie, $\frac{|l-l'|}{l}$ where $l$ is the distance for the point cloud without shortcut and $l'$ the distance with shortcut.
We see that the \ac{fdtm} offset vanishes as the number of points increases, whereas the Fermat offset remains consistent despite the Total Variation distance between both measures decreasing.

\subsection{Single-Linkage Clustering}

In this section we perform a single-linkage algorithm on a simple dataset featuring a mixture of dimensions to highlight the different behaviors of the \ac{fdtm} and Fermat distances.
In single-linkage clustering, the algorithm starts with a cluster for each point, sorts all the edges by weight, and iteratively merges clusters joined by edges of increasing weight.
As a consequence, only the weights matter, and computing the shortest paths between any two points is not needed.
In particular, using the sample Fermat distance yields the exact same result as the Euclidean distance, as the weights of the former are given by applying the increasing function $w \mapsto w^\alpha$ to the weights of the latter.

The dataset used for \cref{fig:cluster} consists in two unit balls with uniform density linked by a uniform line. The proportion of points on the line is chosen to scale as $n^{-1/2}$, so that points on the line are typically closer than points on the balls---so that they are merged together quicker by the Euclidean algorithm---but the mass of the line is small. Then, by setting the \ac{fdtm} parameter $m$ to scale similarly, the \ac{fdtm} weight of edges on the line is significantly larger than inside the balls, which leads to the algorithm separating the two balls more efficiently.

\Cref{fig:cluster} compares the result of the single-linkage clustering on this dataset using both the \ac{fdtm} and the Euclidean distances, showing the evolution of the merging for different amounts of clusters.
The aim of this example is to showcase the differences between these metrics. Single-linkage clustering is a very basic algorithm and the presented dataset is very artificial, however this example illustrates the capabilities of the \ac{fdtm} especially in the context of mixed dimensions.

\begin{figure}[tb]
    \centering
    \begin{minipage}{0.4\textwidth}
        \centering
        \includegraphics[width=\linewidth]{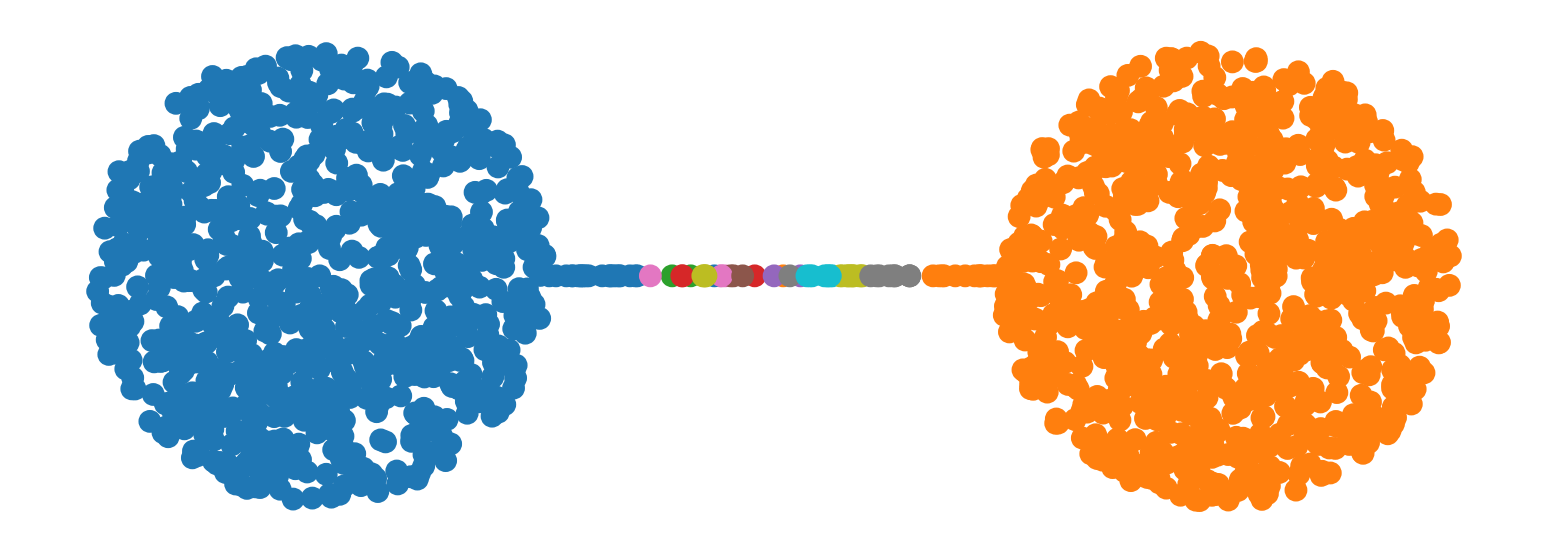}
    \end{minipage}
    \begin{minipage}{0.4\textwidth}
        \centering
        \includegraphics[width=\linewidth]{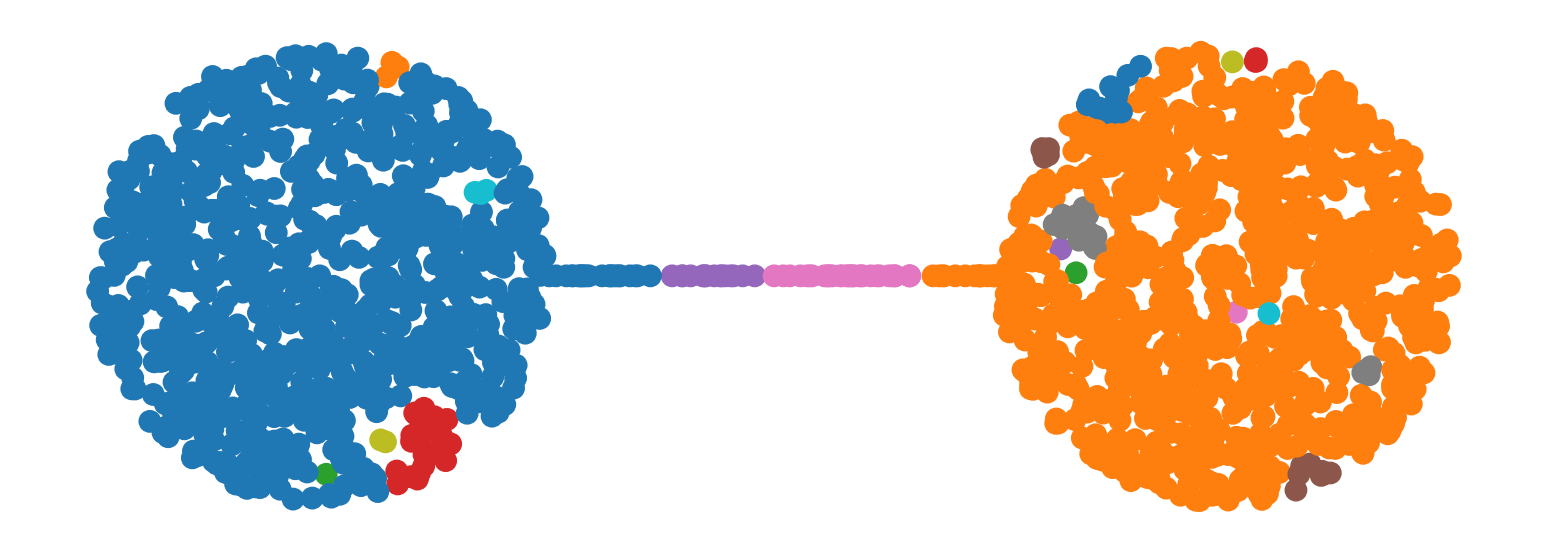}
    \end{minipage}
    
    \vspace{0.25cm}
    
    \begin{minipage}{0.4\textwidth}
        \centering
        \includegraphics[width=\linewidth]{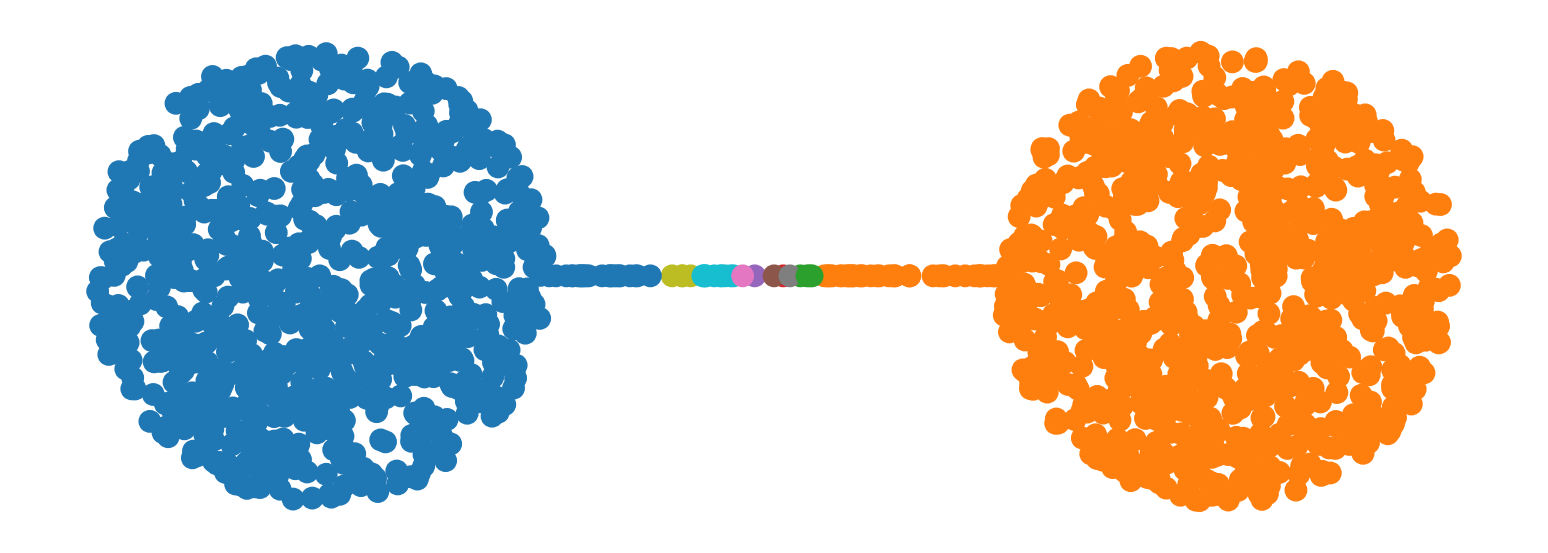}
    \end{minipage}
    \begin{minipage}{0.4\textwidth}
        \centering
        \includegraphics[width=\linewidth]{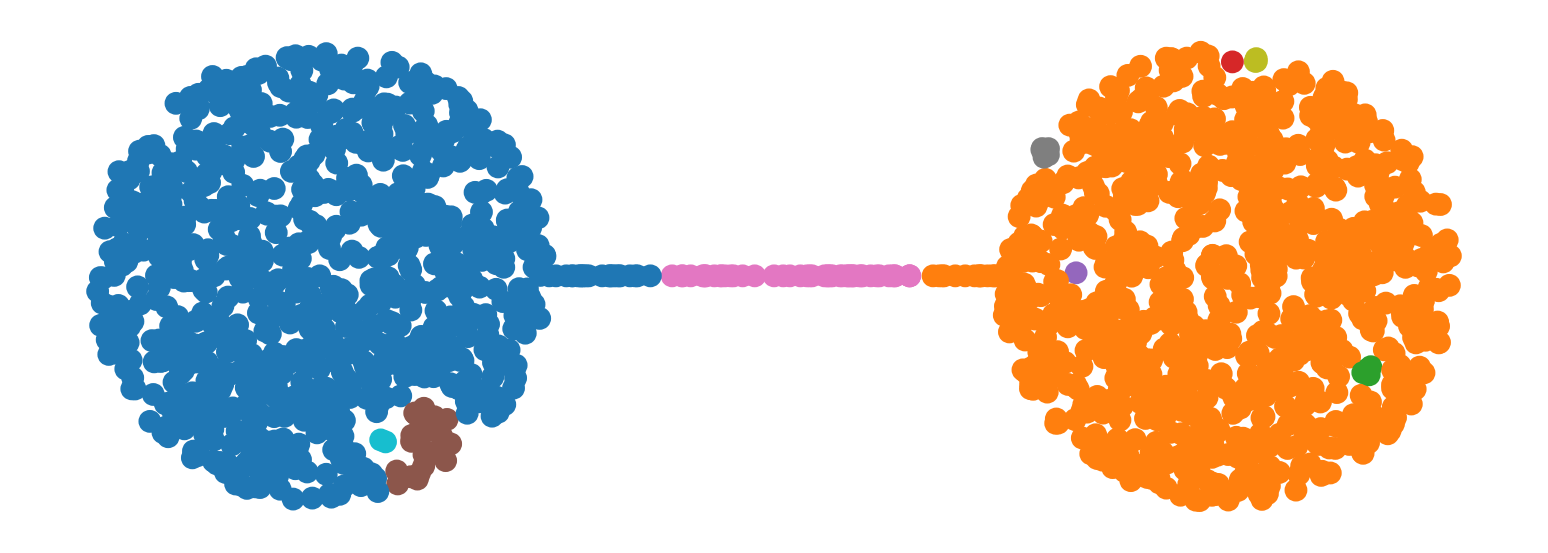}
    \end{minipage}
    
    \vspace{0.25cm}
    
    \begin{minipage}{0.4\textwidth}
        \centering
        \includegraphics[width=\linewidth]{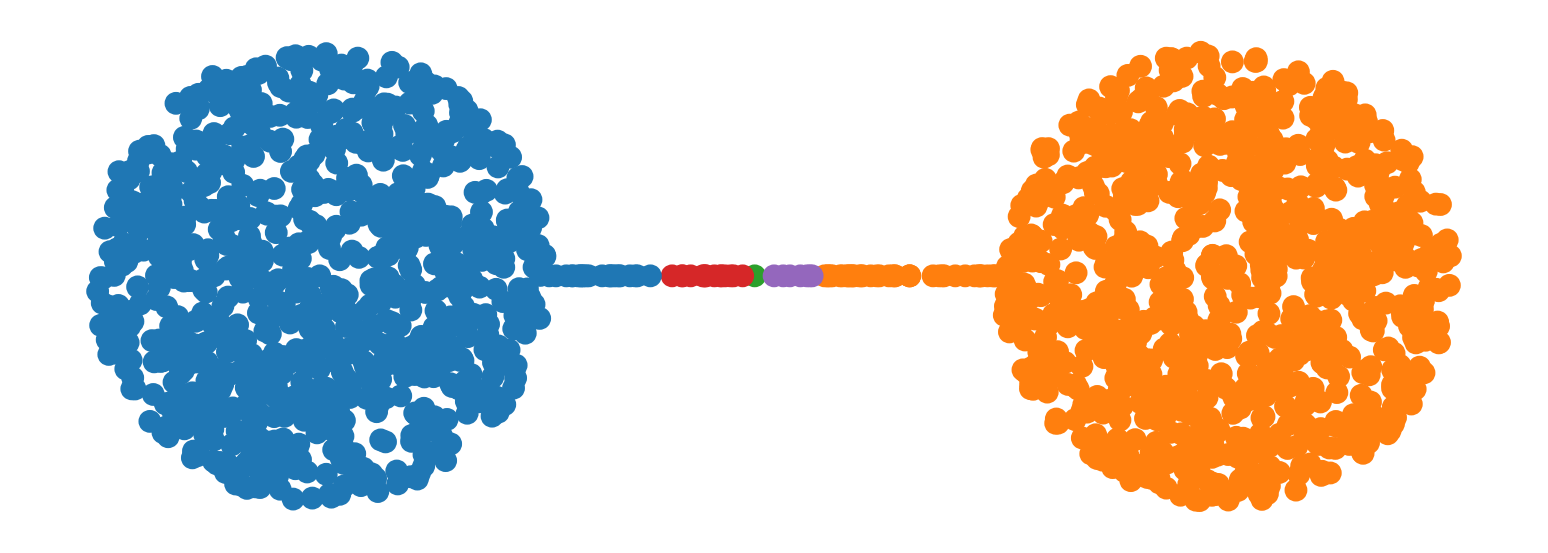}
    \end{minipage}
    \begin{minipage}{0.4\textwidth}
        \centering
        \includegraphics[width=\linewidth]{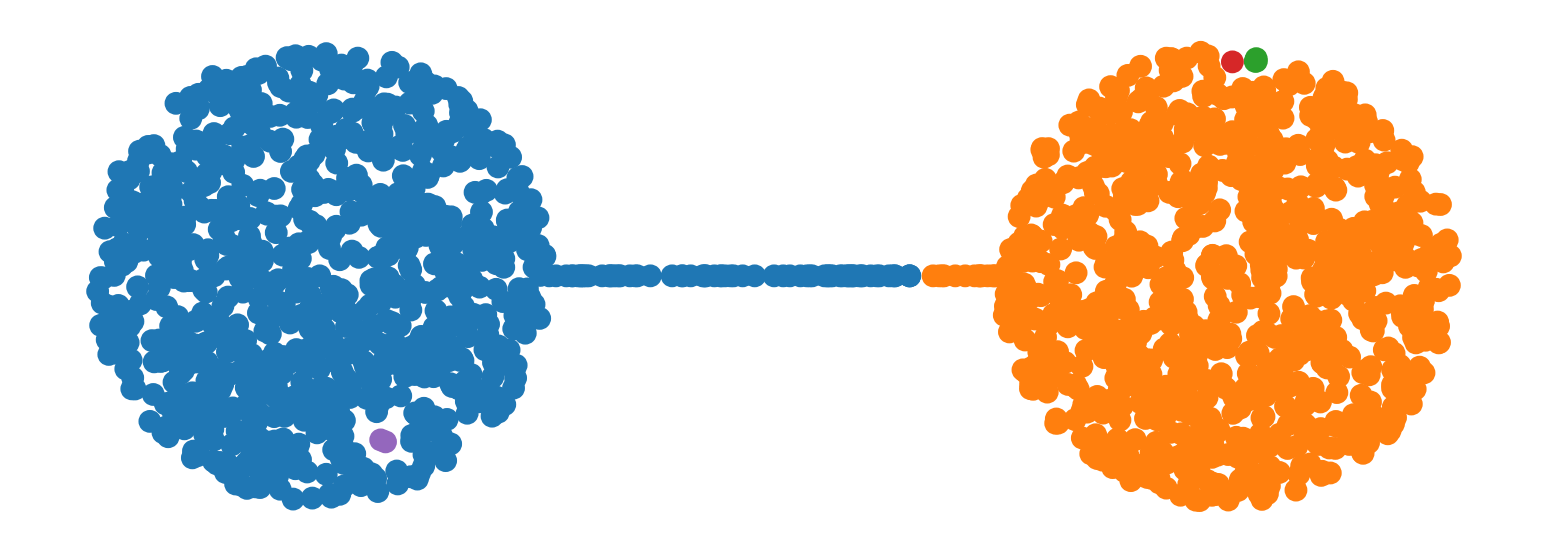}
    \end{minipage}
    
    \vspace{0.25cm}
    
    \begin{minipage}{0.4\textwidth}
        \centering
        \includegraphics[width=\linewidth]{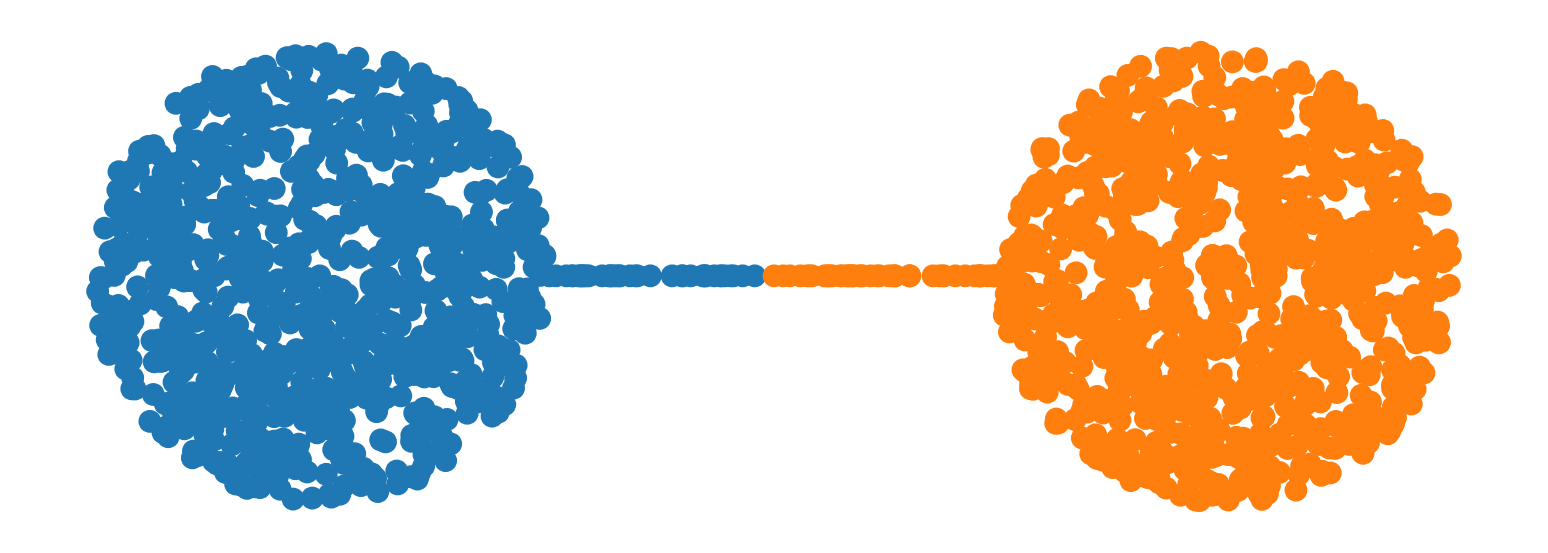}
        \subcaption*{\ac{fdtm}}
    \end{minipage}
    \begin{minipage}{0.4\textwidth}
        \centering
        \includegraphics[width=\linewidth]{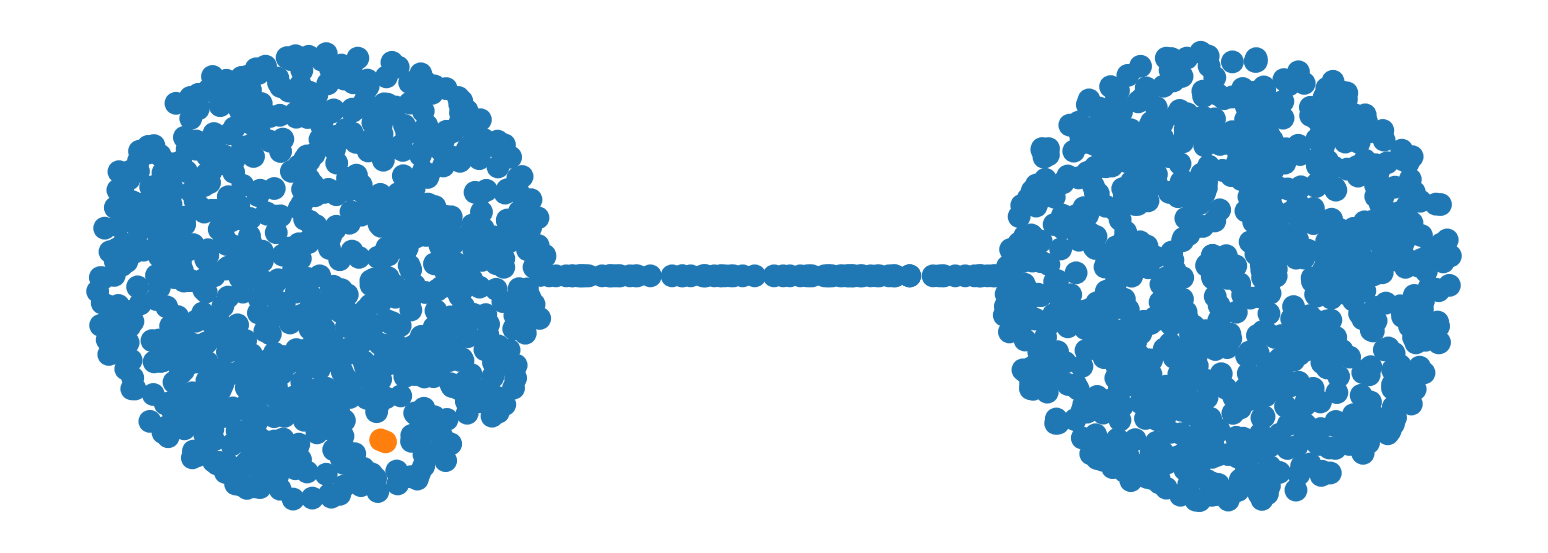}
        \subcaption*{Fermat / Euclidean}
    \end{minipage}
    
    \caption{Single-linkage clustering of a dataset ($n=2000$ sample points) using \ac{fdtm} (left) and Fermat / Euclidean (right) metrics, for $k \in \{20,10,5,2\}$ clusters.
    The \ac{fdtm} parameters are $m\approx0.05$, $p=2$, $\exponent=4$.}
    \label{fig:cluster}
\end{figure}

\section{Conclusion}
\label{sec:conclusion}

In this work we introduced a new type of density-based metric inspired by the Fermat distance and derived quantitative properties of stability and estimation.
The \ac{fdtm} metric is defined without any restriction on the subject measure and our results hold for minimal assumptions on this measure.
Moreover, our proposed estimator of the \ac{fdtm} is naturally defined as the \ac{fdtm} of the empirical measure, and its convergence is not influenced by the ambient dimension of the data.

Future work may investigate theoretical questions such as the optimal minimax bounds for the convergence of the empirical \ac{fdtm} or the generalization to non-Euclidean spaces.
Moreover, in the optic of further adapting results coming from the Fermat literature, objects such as the Fermat graph Laplacians studied in~\cite{garciatrillosFermatDistancesMetric2024} may be adapted to the \ac{fdtm} model, hopefully with analogous results.

\paragraph{Acknowledgements}
This work has been supported by the ANR TopAI chair in Artificial Intelligence (ANR-19-CHIA-0001) and by École Normale Supérieure de Paris. The authors also thank Laure Ferraris for contributing to a preliminary weaker stability result for the \ac{fdtm}.

\printbibliography

\newpage
\appendix

\section{Complementary Discussion}

The first part of this section details the discussion about the assumption of considering all objects within a large convex compact subset $\K\subset\RR^d$.
The second part details the fact that the \ac{fdtm} and Euclidean metrics share the same topology over the domain $\F$.

\subsection{Restriction to a Compact Set}
\label{app:restriction_compact}

As long as we restrict the study of the \ac{fdtm} to endpoints that belong to a compact set~$A$, the fact that $d_\mu$ is a Lipschitz and proper function ensures that geodesic paths remain in a greater compact set $B$.
Indeed, for all $x,y\in A$,
\[D_{\mu,\F}(x,y)
\le D_\mu([x,y])
\le \diam(A) \cdot \max_A d_\mu^\exponent~,\]
and any path leaving a compact set $B$ chosen large enough would achieve a greater \ac{fdtm} length, making it impossible to be a geodesic.
Therefore, replacing $\F$ with the cropped domain $\F\cap B$ does not affect the \ac{fdtm} within $A$.
Moreover, the compact subset $\K$ defined as the convex hull of
\begin{align*}
    \bigcup_{x\in B} \ball\bLp x,\delta_{\mu,m}(x) \bRp
\end{align*}
contains all points that are relevant \ac{wrt} the \ac{dtm} in $B$, so that replacing $\mu$ with the push-forward measure $(\pi_\K)_*\mu$ where $\pi_\K$ is the projection over $\K$ does not affect the \ac{dtm} within $B$, nor does it affect the \ac{fdtm} within $A$ by extension.
Therefore, the assumption that all objects are included in a great compact convex set $\K$ is without loss of generality.

Finally, let us detail \cref{eq:dtm_uniform_bound}:
Since $\supp(\mu)\subset\K$, all the mass lies entirely within a radius of at most $\diam(\K)$ from any $x\in\K$, hence $\delta_{\mu,u}(x) \le \diam(\K)$ for any $u \in (0,1)$, and thus $d_\mu(x) \le \diam(\K)$.
As a consequence, since admissible paths for $x,y\in\K$ are included in the convex hull of $\F\cup\{x,y\}$ and therefore in $\K$, it follows that $d_\mu \le \diam(\K)$ over any path $\gamma\in\Gamma_\F(x,y)$.

\subsection{Equivalent Topologies}
\label{app:equivalent_topology}

Here we detail the following statement from \cref{sec:geodesic_existence}.

\begin{lemma}
\label{lemma:equivalent_topology}
$D_{\mu,\F}$ defines a metric over $\F$ and its topology is equivalent to the Euclidean topology.
\end{lemma}

\begin{proof}
$D_\mu$ is immediately non-negative and symmetric.
Provided that $x,y,z\in \F$, if $\gamma\in\Gamma_\F(x,y)$ and $\gamma'\in\Gamma_\F(y,z)$ then the concatenation of both paths belongs to $\Gamma_\F(x,z)$. The triangular inequality follows from considering the infimum over $\gamma$ and $\gamma'$.

Now, let us show that in $\RR^d$, any ball for the Euclidean metric contains a ball for the \ac{fdtm} metric.
First, the \ac{dtm} may be equal to $0$ only over a finite set, as $d_\mu(x)=0$ if and only if $\mu$ has an atom of mass at least $m$ at $x$.
Then, for all $x\in\RR^d$ and small enough $r>0$, $d_\mu$ is lower bounded over the sphere $\cS(x,r)$ by some $\delta>0$.
Given any point $y$ at Euclidean distance larger than $r$ from $x$, any path $\gamma\in\Gamma_\F(x,y)$ must intersect $\cS(x,r)$ and its \ac{fdtm} length can be lower bounded by some $\epsilon = f(r,\delta,\exponent) > 0$ by Lipchitz property of the \ac{dtm}.
Therefore, for all small enough $r>0$ there exists a \ac{fdtm} ball of radius $\epsilon$ included in the Euclidean ball of radius $r$.
This implies that the \ac{fdtm} between two distinct points is positive---hence the \ac{fdtm} is a metric---and coincidentally shows that the \ac{fdtm} topology is finer than the Euclidean topology.

Conversely, the Euclidean topology is finer than the \ac{fdtm} as a direct consequence of the continuity of $D_\mu$ \ac{wrt} the Euclidean topology.
This concludes that both topologies are equivalent.
\end{proof}

\section{Upper Bound on the Euclidean Length of Geodesics}
\label{app:geodesics}

This section provides the complete proof to \cref{thm:intro:euclidean_bound} discussed in \cref{sec:geodesic_length_general}.
After stating the upper bound on the covering numbers of the sublevel sets of the \ac{dtm}, we consider the simpler case where the domain $F$ is the entire space $\RR^d$ in \cref{app:geodesic:details_no_domain} to showcase the main reasoning of the proof without too many technicalities.
\Cref{app:geodesic:details_domain} then generalizes this reasoning to work with any $F$.
\Cref{app:euclidean_bound_local} finally details the computations required to obtain \cref{corol:euclidean_bound_local}.

Recall the definitions of the packing and covering numbers of a set.
\begin{definition}
Let $A \subset \RR^d$, $\cF$ a finite family of points in $A$ and $r>0$.
\begin{itemize}
    \item $\cF$ is said to be a $r$-\emph{packing} of $A$ if points in $\cF$ are at least $r$ apart from each other, \ie, if $\cB(x,r/2) \cap \cB(x',r/2) = \emptyset$ for any distinct points $x,x' \in \cF$.
    \item $\cF$ is said to be a $r$-\emph{covering} of $A$ if all points in $A$ are at most $r$ apart from a point in $\cF$, \ie, if $A \subset \bigcup_{x\in\cF} \cB(x,r)$.
\end{itemize}
The packing number is then defined as
\[
    \pack(A,r)
    = \max\bLa n\ge0~:~\textrm{$A$ admits a $r$-packing of size $n$}\bRa~,
\]
and the covering number as
\[
    \cov(A,r)
    = \min\bLa n\ge0~:~\textrm{$A$ admits a $r$-covering of size $n$}\bRa~.
\]
\end{definition}

An $r$-packing $\cF$ of $A$ is \emph{maximal} if no point from $A$ can be appended to $\cF$ while preserving the packing property.
A maximal $r$-packing is thus also an $r$-covering. In particular, $\cov(A,r) \le \pack(A,r)$.
A key property of the \ac{dtm} is the following upper bound on the packing numbers of its sublevel sets.

\begin{lemma}[Covering property of the \ac{dtm} sublevels]
\label{lemma:dtm_covering}
Let $\mu$ be a probability measure.
Then for all $\delta>0$,
\begin{align*}
    \cov\bLp L_{\mu,\delta},4\delta \bRp
    \le \pack\bLp L_{\mu,\delta},4\delta \bRp
    \le \frac2m~.
\end{align*}
\end{lemma}

\begin{proof}
Let $\delta>0$.
For all $x\in L_{\mu,\delta}$, since $u\mapsto \delta_{\mu,u}(x)$ is non-decreasing,
\begin{align*}
    \delta^p
    > d_\mu(x)^p
    = \frac1m \int_0^m \delta_{\mu,u}(x)^p \d u
    \ge \frac1m \int_{m/2}^m \delta_{\mu,u}(x)^p \d u
    \ge \frac{1}{2} \delta_{\mu,m/2}(x)^p~,
\end{align*}
hence $\delta_{\mu,m/2}(x) < 2^{1/p}\delta \le 2\delta$, which by definition implies that $\mu\bLp\cB(x,2\delta)\bRp \ge \frac m2$.
Consider a maximal $4\delta$-packing of $L_{\mu,\delta}$ by points $(x_i)_{1\le i\le N}$ where $N = \pack\bLp L_{\mu,\delta},4\delta \bRp$. By definition, the balls $\bLp\cB(x_i,2\delta)\bRp_i$ are pairwise disjoint.
Since $\mu$ is a probability measure it follows that
\begin{align*}
    1
    \ge \mu\Lp\bigsqcup_{i=1}^N \cB(x_i,2\delta)\Rp
    = \sum_{i=1}^N \mu\bLp\cB(x_i,2\delta)\bRp
    \ge N\frac m2~,
\end{align*}
therefore $\cov\bLp L_{\mu,\delta},4\delta \bRp \le \pack\bLp L_{\mu,\delta},4\delta \bRp \le \frac2m$.
\end{proof}

Let us now introduce some notations.
To formalize the decomposition and concatenation of paths, we work in the Abelian free group generated by all paths where the~$+$ operator represents the concatenation in a general sense, that is without assuming that endpoints are shared.
Precisely, we use the quotient group that identifies two general paths when they represent the same overall trajectory, \eg, $[x,y] + [y,z] = [y,z] + [x,y] = [x,y,z]$.
The Euclidean and \ac{fdtm} lengths both naturally extend as morphisms to general paths:
\begin{align*}
    |\gamma + \gamma'|
    = |\gamma| + |\gamma'|
    \qquad\textrm{and}\qquad
    D_\mu(\gamma + \gamma')
    = D_\mu(\gamma) + D_\mu(\gamma')~.
\end{align*}
This is also true for countable concatenations of paths.
Given $\gamma\in\Gamma_\F(x,y)$ and $\delta>0$, we call \emph{sub-$\delta$ sections} of $\gamma$, denoted by $[\gamma]^\delta$, the connected sections of the path located in the sublevel area $L_{\mu,\delta}$ of the \ac{dtm}.
Likewise, the remaining part of the path is called \emph{super-$\delta$ sections} of $\gamma$ and denoted by $[\gamma]_\delta$.
The part of the path located in $L_{\mu,\delta} \setminus L_{\mu,\delta'}$ is denoted by
\begin{align*}
    [\gamma]_{\delta'}^\delta 
    \eqdef [\gamma]^\delta - [\gamma]_{\delta'}. 
\end{align*}
Finally, the connected sections of a path that do not belong to the domain $\F$, which are required to be straight lines by definition, are called \emph{outer sections}. On the other hand, connected sections remaining within the domain are called \emph{inner sections}.

\subsection{Details without Domain Constraints}
\label{app:geodesic:details_no_domain}

We provide in this section the proof of \cref{thm:intro:euclidean_bound} in the case where $\F = \RR^d$, \ie, disregarding any constraint on the domain.
The resulting bound on the length of geodesics is smaller than in the general case.
The first step is the following path modification lemma, which is illustrated by \cref{fig:path_decomposition_no_domain}.

\begin{lemma}
\label{lemma:geodesic:gamma_modif_no_domain}
Let $\gamma\in\Gamma(x,y)$ and $\delta>0$.
Then there exists a modified path $\tilde\gamma = \tilde\eta + \tilde\omega \in \Gamma(x,y)$ such that
\begin{enumerate}[(i)]
    \item $\tilde\eta$ is a modification of the sub-$\delta$ sections $[\gamma]^\delta$ and satisfies
    \begin{align}
        \label{eq:modification1_eta}
        D_\mu(\tilde\eta)
        \le c \delta^{\exponent+1}~,
    \end{align}
    where $c = \frac{16\cdot5^\exponent}{m}$.
    \item $\tilde\omega \subset [\gamma]_\delta$ is a subset of the super-$\delta$ sections of $\gamma$.
\end{enumerate}
\end{lemma}

\begin{figure}
    \centering
    \begin{minipage}{0.45\textwidth}
        \centering
        \includegraphics[width=\linewidth]{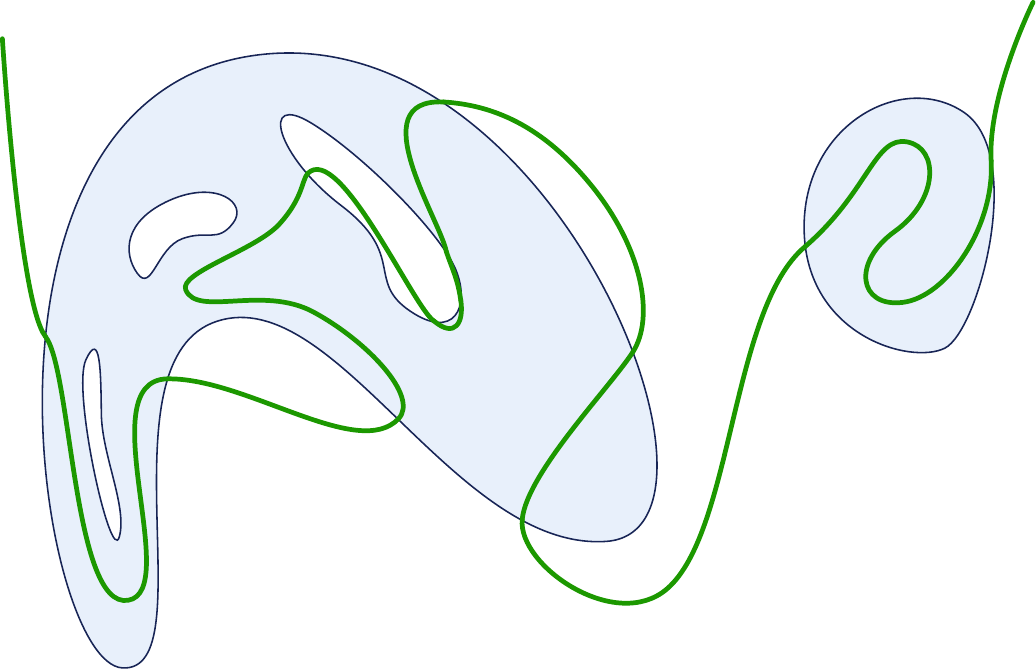}
        \subcaption*{Original path}
    \end{minipage}
    \begin{minipage}{0.45\textwidth}
        \centering
        \includegraphics[width=\linewidth]{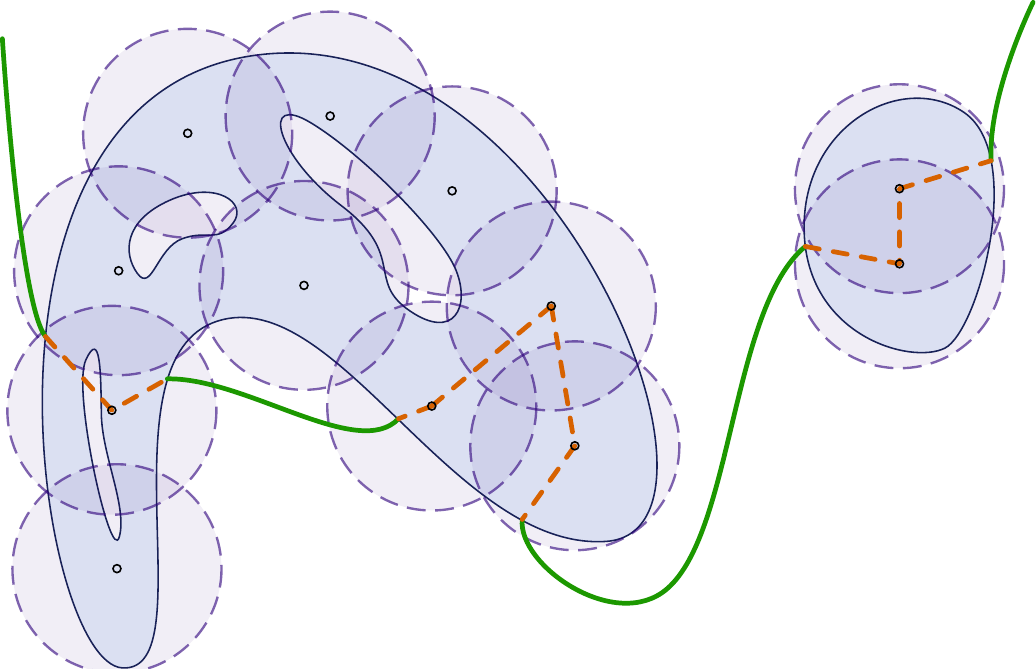}
        \subcaption*{Modified path}
    \end{minipage}
    \caption{Example of the modification process of \cref{lemma:geodesic:gamma_modif_no_domain}.
    The blue area represents a sublevel and the dashed circles on the right figure represent a covering of this sublevel.
    The sublevel sections of the original path are replaced with new sections making up $\tilde\eta$ (in orange) on the right figure.
    Some of the superlevel sections of the original path are removed due to creating loops and the remaining parts make up $\tilde\omega$ (in green).}
    \label{fig:path_decomposition_no_domain}
\end{figure}

\begin{proof}
Denote $\eta_1\in\Gamma(x_1,y_1)$, \dots, $\eta_i\in\Gamma(x_i,y_i)$, \dots the connected paths making up $[\gamma]^\delta$.
Consider a minimal $4\delta$-covering $\cF$ of the sublevel $L_{\mu,\delta}$. Then according to \cref{lemma:dtm_covering}, $|\cF| \le \frac2m$.
For a given $i\ge0$, $x_i$ and $y_i$ belong to the same connected component of $L_{\mu,\delta}$ since they are connected by~$\eta_i$, and therefore can be joined by ``hopping'' between points in $\cF$: There exist pairwise distinct points $z_{i,1},\dots,z_{i,k_i}\in\cF$ such that $\|x_i - z_{i,1}\| \le 4\delta$, $\|y_i - z_{i,k_i}\| \le 4\delta$, and for all $1\le k\le k_i-1$, $\|z_{i,k} - z_{i,k+1}\| \le 8\delta$.
We then define for all $i$, $\overline\eta_i = [x_i, z_{i,1}, z_{i,2}, \dots, z_{i,k_i-1}, z_{i,k_i}, y_i] \in \Gamma_\F(x_i,y_i)$ the corresponding polygonal path from $x_i$ to $y_i$, along with
\begin{align*}
    \overline\gamma
    = [\gamma]_\delta + \sum_i\overline\eta_i
    = \gamma + \sum_i\Lp\overline\eta_i-\eta_i\Rp
    \in \Gamma(x,y)~,
\end{align*}
a new path replacing subsections of $\gamma$ with these new paths.
Finally, we define $\tilde\gamma$ as the result of deleting all loops from $\overline\gamma$ occurring in $\cF$, so that each point $z\in\cF$ is visited at most once by $\tilde\gamma$.
Moreover, we write $\tilde\gamma = \tilde\eta + \tilde\omega$ where $\tilde\eta \subset \sum_i\overline\eta_i$ and $\tilde\omega \subset [\gamma]_\delta$ are the remaining parts after the removal of the loops. This concludes the construction of the modified path $\tilde\gamma \in \Gamma(x,y)$.

It remains to upper bound $D_\mu(\tilde\eta)$. Breaking down the straight lines $[z_{i,k},z_{i,k+1}]$ into two halves in the definition of the paths $\overline\eta_i$ shows that $\tilde\eta$ is the sum of at most $2|\cF|$ paths of the form $[w,z]$ or $[z,w]$ where $z\in\cF \subset L_{\mu,\delta}$, and $w \in \ball(z,4\delta)$ is either an endpoint $x_i$ or $y_i$, or a point halfway to another $z'\in\cF$.
It follows immediately that $|\tilde\eta| \le 8|\cF|\delta$. Moreover, $\tilde\eta$ remains at any point at distance at most $4\delta$ from $\cF \subset L_{\mu,\delta}$, hence $\tilde\eta \subset L_{\mu,5\delta}$ by Lipschitz property of $d_\mu$. Thus,
\begin{align*}
    D_\mu(\tilde\eta)
    \le |\tilde\eta| \cdot (5\delta)^\exponent
    \le \frac{16\cdot5^\exponent}{m} \delta^{\exponent+1}~,
\end{align*}
which concludes the proof.
\end{proof}

We are now able to bound the length of any geodesic.

\begin{proof}[Proof of \cref{thm:intro:euclidean_bound}]
Let $\gamma \in \Gamma_\mu^\star(x,y)$ be a geodesic and $\delta_0>0$.
For all $k\ge0$ let $\delta_k = \rho^{-k} \delta_0$ with $\rho = 2^{1/\exponent}$.
Consider $\tilde\gamma_k = \tilde\eta_k + \tilde\omega_k$ the modified path associated with the sub-$\delta_k$ level given by \cref{lemma:geodesic:gamma_modif_no_domain}.
Then
\begin{align*}
    \delta_{k+1}^\exponent \bv[\gamma]^{\delta_k}_{\delta_{k+1}}\bv
    &\le D_\mu\bLp[\gamma]^{\delta_k}_{\delta_{k+1}}\bRp
        && \textrm{($d_\mu^\exponent \ge \delta_{k+1}^\exponent$ over $[\gamma]^{\delta_k}_{\delta_{k+1}}$)}\\
    &\le D_\mu(\gamma) - D_\mu\bLp[\gamma]_{\delta_k}\bRp\\
    &\le D_\mu(\tilde\gamma_k) - D_\mu(\tilde\omega_k)
        && \textrm{($\gamma$ is optimal and $\tilde\omega_k\subset[\gamma]_{\delta_k}$)}\\
    &= D_\mu(\tilde\eta_k)\\
    &\le c \delta_k^{\exponent+1}~.
        && \textrm{(\crefnopar{eq:modification1_eta})}
\end{align*}
Therefore,
\begin{align*}
    \bv[\gamma]^{\delta_0}\bv
    = \sum_{k=0}^\pinfty \bv[\gamma]^{\delta_k}_{\delta_{k+1}}\bv
    \le c \sum_{k=0}^\pinfty \frac{\delta_k^{\exponent+1}}{\delta_{k+1}^\exponent}
    = c\rho^\exponent \sum_{k=0}^\pinfty \rho^{-k} \delta_0
    = c \frac{2^{1+1/\exponent}}{2^{1/\exponent}-1} \delta_0
    \le 2c \Lp1+\frac{\exponent}{\log(2)}\Rp \delta_0~.
\end{align*}
The inequality $\frac{t}{t-1} \le 1 + \frac{1}{\log(t)}$ with $t = 2^{1/\exponent}$ was used at the last step.
In particular, choosing $\delta_0 = \|d_\mu\|_{\infty,\gamma}$ provides an upper bound on the length of the entire geodesic and concludes the proof.
\end{proof}

\subsection{Details with Domain Constraints}
\label{app:geodesic:details_domain}

We provide in this section the proof of \cref{thm:intro:euclidean_bound} in the general case.
First, we state an intermediate result to upper bound the length of any segment within a sublevel of the \ac{dtm}.

\begin{lemma}
\label{lemma:segment_in_covering}
Let $A$ be a measurable subset of $\RR^d$ and $x,y \in \RR^d$.
Then for all $r>0$,
\begin{align*}
    \bv[x,y] \cap A\bv
    \le \cov(A,r) \cdot 2r~.
\end{align*}
\end{lemma}

\begin{proof}
Denote $l = \|x-y\|$ and $\gamma:[0,l] \to \RR^d$ the arc length parameterization of $[x,y]$.
Let $\cF$ be a minimal $r$-covering of $A$. Then
\begin{align*}
    \bv[x,y] \cap A\bv
    = \int_0^l \1_A(\gamma(t)) \d t
    \le \int_0^l \1_{\bLp\bigcup_{a\in\cF}\cB(a,r)\bRp}(\gamma(t)) \d t
    \le \sum_{a\in\cF} \int_0^l \1_{\cB(a,r)}(\gamma(t)) \d t~.
\end{align*}
If $a\in\cF$ and $\gamma(t),\gamma(s) \in \cB(a,r)$, then $t-s = \|\gamma(t)-\gamma(s)\| \le 2r$. Therefore,
\begin{align*}
    \bv[x,y] \cap A\bv
    \le \sum_{a\in\cF} \bLp\sup\{t:\gamma(t)\in\cB(a,r)\} - \inf\{t:\gamma(t)\in\cB(a,r)\}\bRp
    \le \cov(A,r) \cdot 2r~.
\end{align*}
\end{proof}

Together, \cref{lemma:segment_in_covering,lemma:dtm_covering} yield the following result.

\begin{corollary}
\label{corol:segment_in_sublevel}
Let $\delta>0$ and $x,y \in \RR^d$.
Then
\begin{align*}
    \bv[x,y] \cap L_{\mu,\delta}\bv
    \le \frac{16}{m} \delta~.
\end{align*}
\end{corollary}

Let us now give the general version of \cref{lemma:geodesic:gamma_modif_no_domain}.

\begin{lemma}
\label{lemma:geodesic:gamma_modif_domain}
Let $\gamma\in\Gamma_\F(x,y)$ and $\delta>0$.
Then there exists a decomposition of the path $\gamma = \eta + \chi + \omega$ along with a modified path $\tilde\gamma = \tilde\eta + \tilde\chi + \tilde\omega \in \Gamma_\F(x,y)$ such that
\begin{enumerate}[(i)]
    \item $\eta \subset [\gamma]^{\rho\delta}$ contains the inner sections of $[\gamma]^\delta$ and the outer sections of $\gamma$ that intersect $L_{\mu,\delta}$ while not being too long outside this sublevel, with $\rho = 1 + \frac{32}{m}$.
    $\tilde\eta$ is a modification of $\eta$ satisfying
    \begin{align}
        \label{eq:modification2_eta}
        D_\mu(\tilde\eta)
        \le c \delta^{\exponent+1}~,
    \end{align}
    where $c = \frac{16\cdot5^\exponent+1152\cdot165^\exponent}{m^{2+\exponent}}$.
    \item $\chi$ contains the remaining outer sections of $\gamma$ intersecting $L_{\mu,\delta}$ and satisfies
    \begin{align}
        \label{eq:modification2_chi}
        \bv[\chi]^{\delta}\bv
        \le \frac12 \bv[\gamma]_\delta^{\rho\delta}\bv~.
    \end{align}
    $\tilde\chi \subset \chi$ is a subset of these outer sections.
    \item $\omega$ contains the rest of $\gamma$ and $\tilde\omega \subset \omega \subset [\gamma]_\delta$ is a subset of this remaining part.
\end{enumerate}
\end{lemma}

\begin{figure}
    \centering
    \begin{minipage}{0.45\textwidth}
        \centering
        \includegraphics[width=\linewidth]{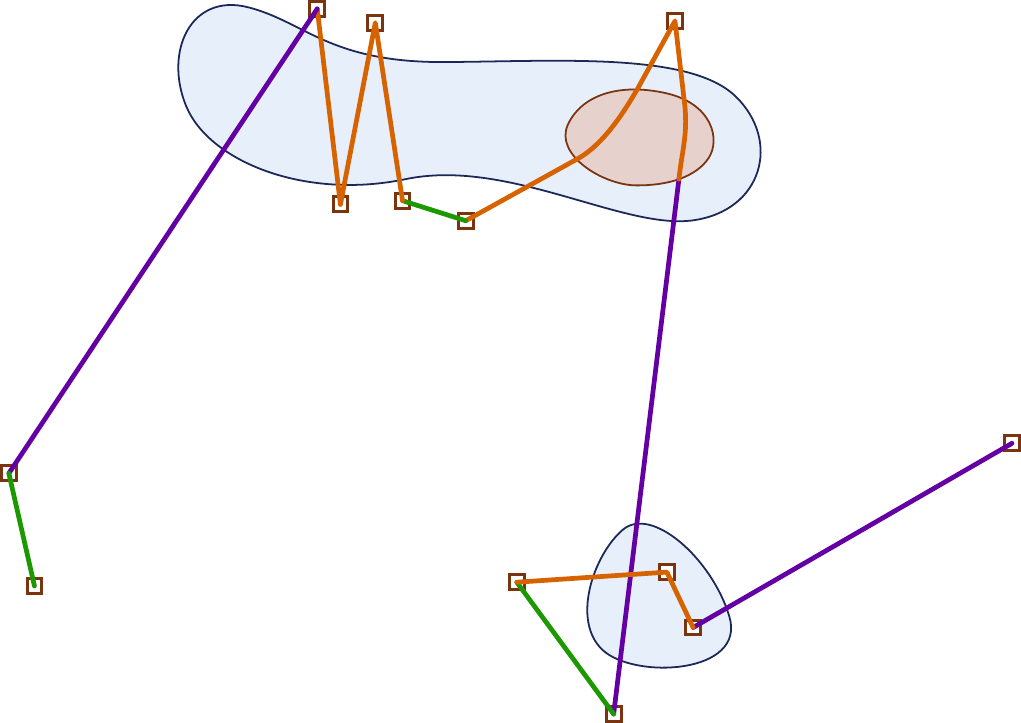}
        \subcaption*{Original path}
    \end{minipage}
    \begin{minipage}{0.45\textwidth}
        \centering
        \includegraphics[width=\linewidth]{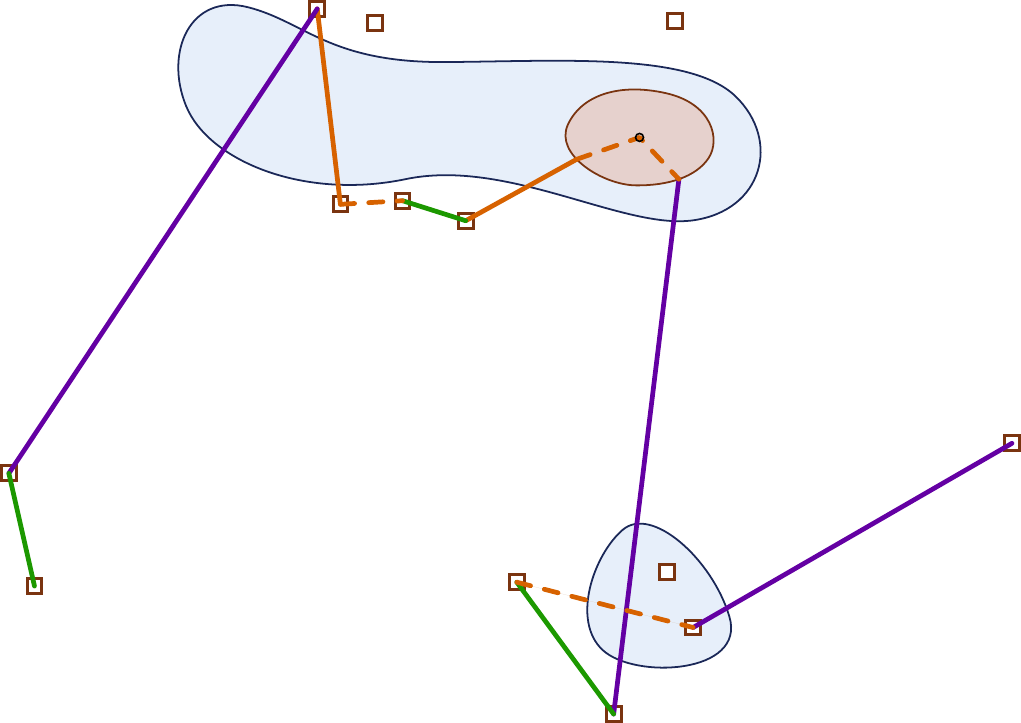}
        \subcaption*{Modified path}
    \end{minipage}
    \caption{Example of the modification process of \cref{lemma:geodesic:gamma_modif_domain}.
    The blue area represents a sublevel, the brown area and squares represent the domain.
    The $\eta$ sections of the original path (left figure, orange) are partially replaced (right figures, orange dashed lines) to make up $\tilde\eta$ (right figure, orange).
    The $\chi$ and $\omega$ sections (purple and green respectively) are left untouched in this example as no loop is removed.}
    \label{fig:path_decomposition_with_domain}
\end{figure}

When dealing with the domain constraints to modify the sublevel of a path, the main issue is that if a long outer section crosses the sublevel, then it is not possible to modify only the sublevel part of this outer section. Indeed, doing so would break the segment at a point that is outside the domain $\F$ and the new path would no longer belong to $\Gamma_\F(x,y)$.
To account for this issue, we choose to modify the ``problematic'' outer sections when their length outside the sublevel is small enough similarly as we do with inner sections of the sublevel. All these sections make up $\eta$.
The remaining outer sections make up $\chi$ and have the property that their length inside the sublevel is smaller than a fraction of their length outside, which allows to bound the overall length by other means.

\Cref{fig:path_decomposition_with_domain} illustrates the process described in \cref{lemma:geodesic:gamma_modif_domain}.
Notice that a loop subsists in the modified path, which cannot be removed as it occurs at a point outside the domain. This example illustrates how it is not always possible to modify the path with limited length within the sublevel proportional to its covering number, as $\chi$ sections may not be reliably modified due to their arbitrarily large length, even when they create loops.

\begin{proof}
Denote $\eta_1\in\Gamma_\F(x_1,y_1)$, \dots, $\eta_i\in\Gamma_\F(x_i,y_i)$, \dots the connected paths making up $[\gamma]^\delta \cap \F$ and $[a_1,b_1]$, \dots, $[a_j,b_j]$, \dots the outer sections of $\gamma$ that intersect with the sublevel $L_{\mu,\delta}$, and such that any part of the segment outside $L_{\mu,\delta}$ is of length at most $\theta = \frac{32}{m} \delta$, which in particular implies that all $a_j$ and $b_j$ belong to $L_{\mu,\delta+\theta}$ by Lipschitz property of $d_\mu$.
Denote $\eta = \sum_i \eta_i + \sum_j [a_j,b_j]$, $\chi$ the remaining outer sections of $\gamma$ intersecting with $L_{\mu,\delta}$, and finally $\omega = \gamma - \eta - \chi \subset [\gamma]_\delta$ the rest of the path.

Consider a minimal $4\delta$-covering $\cF$ of $L_{\mu,\delta}\cap \F$.
According to \cref{lemma:dtm_covering}, $|\cF| \le \frac2m$.
For a given $i\ge0$, $x_i$ and $y_i$ belong to the same connected component of $L_{\mu,\delta}\cap \F$ since they are connected by $\eta_i$.
Therefore, there exist pairwise distinct points $z_{i,1},\dots,z_{i,k_i}\in\cF$ such that $\|x_i - z_{i,1}\| \le 4\delta$, $\|y_i - z_{i,k_i}\| \le 4\delta$, and for all $1\le k\le k_i-1$, $\|z_{i,k} - z_{i,k+1}\| \le 8\delta$.
We then define for all $i$, $\overline\eta^1_i = [x_i, z_{i,1}, z_{i,2}, \dots, z_{i,k_i-1}, z_{i,k_i}, y_i] \in \Gamma_\F(x_i,y_i)$ the corresponding polygonal path from $x_i$ to $y_i$.

Consider a minimal $4(\delta+\theta)$-covering $\cG$ of $\{a_1,b_1,\dots,a_j,b_j,\dots\} \subset L_{\mu,\delta+\theta}$.
According to \cref{lemma:dtm_covering}, $|\cG| \le \frac2m$.
For all $j\ge0$, define $\overline\eta^2_j = [a_j,z'_{j,1},z'_{j,2},b_j]$ where $z'_{j,1}$ and $z'_{j,2}$ are respectively the closest neighbors of $a_j$ and $b_j$ in $\cG$, and finally let
\begin{align*}
    \overline\gamma
    = \sum_i\overline\eta^1_i + \sum_j\overline\eta^2_j + \chi + \omega
    = \gamma + \sum_i\Lp\overline\eta^1_i-\eta_i\Rp + \sum_j\Lp\overline\eta^2_j-[a_j,b_j]\Rp~.
\end{align*}
Since $\overline\gamma$ is the result of replacing some sections of $\gamma$ with straight lines between points in~$\F$, $\overline\gamma \in \Gamma_\F(x,y)$.
Finally, we define $\tilde\gamma$ as the result of deleting all loops from $\overline\gamma$ occurring at points in $\cF\cup\cG \subset \F$, so that $\tilde\gamma \in \Gamma_\F(x,y)$ and each point $z\in\cF\cup\cG$ is visited at most once by $\tilde\gamma$.
Moreover, we write $\tilde\gamma = \tilde\eta + \tilde\chi + \tilde\omega$ where $\tilde\eta = \tilde\eta^1 + \tilde\eta^2$, and $\tilde\eta^1 \subset \sum_i\overline\eta^1_i$, $\tilde\eta^2 \subset \sum_j\overline\eta^2_j$, $\tilde\chi \subset \chi$, and $\tilde\omega \subset \omega \subset [\gamma]_\delta$ are the remaining parts after the removal of the loops.
This concludes the construction of the modified path $\tilde\gamma$.

It remains to establish \cref{eq:modification2_eta,eq:modification2_chi}.
With the exact same reasoning as for \cref{eq:modification1_eta}, $|\tilde\eta^1| \le 8|\cF|\delta$, $\tilde\eta^1 \subset L_{\mu,5\delta}$, and thus
\begin{align}
    \label{eq:proof:modification_general1}
    D_\mu(\tilde\eta^1)
    \le |\tilde\eta^1| \cdot (5\delta)^\exponent
    \le \frac{16\cdot5^\exponent}{m} \delta^{\exponent+1}~.
\end{align}
On the other hand, $\tilde\eta^2$ is the sum of at most $|\cG|$ paths $\overline\eta^2_j$ (either whole or just part of them).
For all $j$, $[a_j,b_j]$ has length at most $\theta$ outside $L_{\mu,\delta}$ by assumption and at most $\frac{16}{m} \delta$ inside $L_{\mu,\delta}$ due to \cref{corol:segment_in_sublevel}. Moreover, by definition of $\cG$, $\|a_j-z'_{j,1}\| \le 4(\delta+\theta)$, $\|b_j-z'_{j,2}\| \le 4(\delta+\theta)$, and
\begin{align*}
    \|z'_{j,1}-z'_{j,2}\|
    \le \|z'_{j,1}-a_j\| + \|a_j-b_j\| + \|b_j-z'_{j,2}\|
    \le 8(\delta+\theta) + \theta + \frac{16}{m} \delta~,
\end{align*}
so that
\begin{align*}
    |\overline\eta^2_j|
    = \|a_j-z'_{j,1}\| + \|z'_{j,1}-z'_{j,2}\| + \|z'_{j,2}-b_j\|
    \le 16(\delta+\theta) + \theta + \frac{16}{m} \delta
    \le \frac{32}{m}\delta + 17\theta
    = 18\theta~,
\end{align*}
and finally
\begin{align*}
    |\tilde\eta^2|
    \le |\cG| \cdot 18\theta
    \le \frac{1152}{m^2} \delta~.
\end{align*}
Then, for all $j$, any point of $\overline\eta^2_j$ is at distance at most $4(\delta+\theta)$ from $[a_j,b_j] \subset L_{\mu,\delta+\theta}$, hence $\tilde\eta^2 \subset L_{\mu,5(\delta+\theta)}$ and
\begin{align}
    \label{eq:proof:modification_general2}
    D_\mu(\tilde\eta^2)
    \le |\tilde\eta^2| \cdot (5(\delta+\theta))^\exponent
    \le \frac{1152}{m^2} \Lp\frac{160}{m}+5\Rp^\exponent \delta^{\exponent+1}
    \le \frac{1152\cdot165^\exponent}{m^{2+\exponent}}\delta^{\exponent+1}~.
\end{align}
Put together, \cref{eq:proof:modification_general1,eq:proof:modification_general2} yield
\begin{align*}
    D_\mu(\tilde\eta)
    = D_\mu(\tilde\eta^1) + D_\mu(\tilde\eta^2)
    \le \frac{16\cdot5^\exponent+1152\cdot165^\exponent}{m^{2+\exponent}} \delta^{\exponent+1}~.
\end{align*}
Now, denote $(\overline\chi_k)_k$ the different segments making up $\chi$. For each segment $\overline\chi_k$, according to \cref{corol:segment_in_sublevel},
\begin{align}
    \label{eq:proof:modification_general5}
    \bv[\overline\chi_k]^\delta\bv
    \le \frac{16}{m} \delta
    = \frac12 \theta
    \le \frac12 \bv[\overline\chi_k]^{\delta+\theta}_\delta\bv~.
\end{align}
The last inequality is due to the following:
If $[\overline\chi_k]_\delta \subset L_{\mu,\delta+\theta}$ then it follows immediately from the assumption that $\bv[\overline\chi_k]_\delta\bv > \theta$.
Else, $\overline\chi_k$ intersects both $L_{\mu,\delta}$ and $\RR^d\setminus L_{\mu,\delta+\theta}$, which are separated by a distance of at least $\theta$ by Lipschitz property of $d_\mu$, so that $\bv[\overline\chi_k]^{\delta+\theta}_\delta\bv \ge \theta$.
Then, summing \cref{eq:proof:modification_general5} over all $\overline\chi_k$ yields
\begin{align*}
    \bv[\chi]^\delta\bv
    \le \frac12 \bv[\chi]^{\rho\delta}_\delta\bv
    \le \frac12 \bv[\gamma]^{\rho\delta}_\delta\bv~,
\end{align*}
where $\rho = \frac{\delta+\theta}{\delta} = 1+\frac{32}{m}$, which concludes the proof.
\end{proof}

We are now able to bound the length of any geodesic.

\begin{proof}[Proof of \cref{thm:intro:euclidean_bound}]
Let $\gamma \in \Gamma_{\mu,\F}^\star(x,y)$ be a geodesic and $\delta_0>0$.
For all $k\ge0$ let $\delta_k = \rho^{-k} \delta_0$ with $\rho = 1 + \frac{32}{m}$.
Consider $\gamma = \eta_k + \chi_k + \omega_k$ and $\tilde\gamma_k = \tilde\eta_k + \tilde\chi_k + \tilde\omega_k$ the decomposition and modified path associated with the sub-$\delta_k$ level given by \cref{lemma:geodesic:gamma_modif_domain}.
Recall that by construction $[\gamma]^{\delta_k} = \eta_k + [\chi_k]^{\delta_k}$, thus
\begin{align}
    \label{eq:proof:optimal_length_general2}
    \bv[\gamma]^{\delta_0}\bv
    = \sum_{k=0}^\pinfty \bv[\gamma]^{\delta_k}_{\delta_{k+1}}\bv
    = \sum_{k=0}^\pinfty \bv[\eta_k]^{\delta_k}_{\delta_{k+1}} + [\chi_k]^{\delta_k}_{\delta_{k+1}}\bv
    \le \sum_{k=0}^\pinfty \bv[\eta_k]_{\delta_{k+1}}\bv + \sum_{k=0}^\pinfty  \bv[\chi_k]^{\delta_k}\bv~.
\end{align}
On one hand, $\bv[\chi_k]^{\delta_k}\bv \le \frac12 \bv[\gamma]^{\delta_{k-1}}_{\delta_k}\bv$ for all $k\ge0$ by \cref{eq:modification2_chi}, hence
\begin{align}
    \label{eq:proof:optimal_length_general3}
    \sum_{k=0}^\pinfty  \bv[\chi_k]^{\delta_k}\bv
    \le \frac12 \sum_{k=0}^\pinfty  \bv[\gamma]^{\delta_{k-1}}_{\delta_k}\bv
    = \frac12 \bv[\gamma]^{\rho\delta_0}\bv~.
\end{align}
On the other hand, for all $k\ge0$,
\begin{align*}
    \delta_{k+1}^\exponent \bv[\eta_k]_{\delta_{k+1}}\bv
    &\le D_\mu\bLp[\eta_k]_{\delta_{k+1}}\bRp
        && \textrm{($d_\mu^\exponent \ge \delta_{k+1}^\exponent$ over $[\eta_k]_{\delta_{k+1}}$)}\\
    &\le D_\mu(\eta_k)\\
    &= D_\mu(\gamma) - D_\mu(\chi_k+\omega_k)\\
    &\le D_\mu(\tilde\gamma_k) - D_\mu(\tilde\chi_k+\tilde\omega_k)
        && \textrm{($\gamma$ is optimal and $\tilde\chi_k+\tilde\omega_k \subset \chi_k+\omega_k$)}\\
    &= D_\mu(\tilde\eta_k)\\
    &\le c \delta_k^{\exponent+1}~.
        && \textrm{(\crefnopar{eq:modification2_eta})}
\end{align*}
Therefore
\begin{align}
    \label{eq:proof:optimal_length_general4}
    \sum_{k=0}^\pinfty \bv[\eta_k]_{\delta_{k+1}}\bv
    \le c \sum_{k=0}^\pinfty \frac{\delta_k^{\exponent+1}}{\delta_{k+1}^\exponent}
    = c\rho^\exponent \sum_{k=0}^\pinfty \rho^{-k} \delta_0
    = c \frac{\rho^{\exponent+1}}{\rho-1} \delta_0~.
\end{align}
It follows from \crefrange{eq:proof:optimal_length_general2}{eq:proof:optimal_length_general4} that
\begin{align*}
    \bv[\gamma]^{\delta_0}\bv
    \le c \frac{\rho^{\exponent+1}}{\rho-1} \delta_0
        + \frac12 \bv[\gamma]^{\rho\delta_0}\bv~,
\end{align*}
and then
\begin{align*}
    \bv[\gamma]^{\delta_0}\bv
    \le 2c \frac{\rho^{\exponent+1}}{\rho-1} \delta_0 + \bv[\gamma]^{\rho\delta_0}_{\delta_0}\bv~.
\end{align*}
In particular, choosing $\delta_0 = \|d_\mu\|_{\infty,\gamma}$ provides an upper bound on the length of the entire geodesic and concludes the proof with
\begin{align*}
    |\gamma|
    \le 2c \frac{\rho^{\exponent+1}}{\rho-1} \|d_\mu\|_{\infty,\gamma}~.
\end{align*}
\end{proof}

\subsection{Local Bound}
\label{app:euclidean_bound_local}

Finally, we provide the proof of \cref{corol:euclidean_bound_local}.
The first step is to upper bound the \ac{dtm} over a geodesic using its values at the endpoints and the Euclidean distance between them.

\begin{lemma}
\label{lemma:geodesic_domain_delta}
Let $x,y\in\RR^d$ and $\gamma\in\Gamma_{\mu,\F}^\star(x,y)$ be a geodesic path.
Then
\begin{align*}
    \|d_\mu\|_{\infty,\gamma}
    \le \max\bLp d_\mu(x),d_\mu(y)\bRp + \|x-y\|~.
\end{align*}
\end{lemma}

\begin{proof}
Denote $r = \|x-y\|$ and $\delta_0 = \max\bLp d_\mu(x),d_\mu(y)\bRp$.
For all $t\in[0,1]$, by Lipschitz property of $d_\mu$, $d_\mu\bLp (1-t)x + ty\bRp
\le d_\mu(x) + t\|x-y\| \le \delta_0 + tr$, so that
\begin{align}
    \label{eq:lemma:geodesic_domain_delta:1}
    D_\mu([x,y])
    \le r\int_0^1 \bLp \delta_0 + tr\bRp^\exponent \d t
    = \frac{1}{\exponent+1} \Lp \bLp \delta_0+r\bRp^{\exponent+1}-\delta_0^{\exponent+1}\Rp~.
\end{align}
Let $\gamma\in\Gamma_{\mu,\F}^\star(x,y)$ be an arc length parameterization, that is, $\gamma : [0,|\gamma|] \to \RR^d$ with $\|\dot\gamma(t)\|=1$ almost everywhere.
Denote $\|d_\mu\|_{\infty,\gamma} = \delta$. If $\delta<\delta_0$ then the result is immediate.
Else, assume that $\delta \ge \delta_0$ and denote $\tau\in[0,|\gamma|]$ such that $d_\mu(\gamma(\tau)) = \delta$.
The Lipschitz property of $d_\mu$ ensures that $\delta - d_\mu(x) \le \tau$ and the \ac{fdtm} length of the section of $\gamma$ between times $\tau-\delta+d_\mu(x)$ and $\tau$ can be lower bounded as
\begin{align*}
    \int_{\tau-\delta+d_\mu(x)}^\tau d_\mu(\gamma(t))^\exponent \d t
    \ge \int_0^{\delta-d_\mu(x)} (\delta-t)^\exponent \d t
    = \frac{1}{\exponent+1} \bLp\delta^{\exponent+1} - d_\mu(x)^{\exponent+1}\bRp~.
\end{align*}
Doing the same reasoning for the integral between $\tau$ and $\tau + \delta-d_\mu(y)$ then implies that
\begin{align}
    \label{eq:lemma:geodesic_domain_delta:2}
    D_{\mu,\F}(\gamma)
    \ge \frac{1}{\exponent+1} \bLp\delta^{\exponent+1} - d_\mu(x)^{\exponent+1}\bRp
        + \frac{1}{\exponent+1} \bLp\delta^{\exponent+1} - d_\mu(y)^{\exponent+1}\bRp
    \ge \frac{2}{\exponent+1} \bLp\delta^{\exponent+1} - \delta_0^{\exponent+1}\bRp~.
\end{align}
Now, the fact that $\gamma$ is a geodesic together with \cref{eq:lemma:geodesic_domain_delta:1,eq:lemma:geodesic_domain_delta:2} imply that
\begin{align*}
    \frac{2}{\exponent+1} \bLp\delta^{\exponent+1} - \delta_0^{\exponent+1}\bRp
    \le D_\mu(\gamma)
    \le D_\mu([x,y])
    \le \frac{1}{\exponent+1} \Lp \bLp \delta_0+r\bRp^{\exponent+1}-\delta_0^{\exponent+1}\Rp~,
\end{align*}
hence
\begin{align*}
    \delta^{\exponent+1}
    \le \frac12 \Lp\bLp \delta_0+r\bRp^{\exponent+1} + \delta_0^{\exponent+1}\Rp
    \le \bLp \delta_0+r\bRp^{\exponent+1}~,
\end{align*}
and finally
\begin{align*}
    \|d_\mu\|_{\infty,\gamma}
    = \delta
    \le \delta_0 + r
    = \max\bLp d_\mu(x),d_\mu(y)\bRp + \|x-y\|~.
\end{align*}
\end{proof}

\begin{proof}[Proof of \cref{corol:euclidean_bound_local}]
Let $\gamma$ be an arc length parameterization of a geodesic. From \cref{thm:intro:euclidean_bound,lemma:geodesic_domain_delta} we know that
\begin{align*}
    |\gamma|
    \le \lambda \Lp\max\bLp d_\mu(x),d_\mu(y)\bRp + \|x-y\|\Rp~,
\end{align*}
where $\lambda$ is the hidden constant in \cref{eq:thm:euclidean_bound}.
If $\max\bLp d_\mu(x),d_\mu(y)\bRp \le 2^{\exponent+2} \|x-y\|$ then the result follows immediately as
\begin{align*}
    |\gamma|
    \le (2^{\exponent+2}+1)\lambda \|x-y\|~.
\end{align*}
Else, assume without loss of generality that $d_\mu(x) > 2^{\exponent+2}\|x-y\|$ and notice that by Lipschitz property of $d_\mu$,
\begin{align}
    \label{eq:proof:euclidean_bound_local1}
    D_\mu(\gamma)
    \le D_\mu([x,y])
    \le \|x-y\| \Lp d_\mu(x) + \frac{d_\mu(x)}{2^{\exponent+2}}\Rp^\exponent
    \le 2 \|x-y\| d_\mu(x)^\exponent~.
\end{align}
On the other hand,
\begin{align}
    \label{eq:proof:euclidean_bound_local2}
    D_\mu(\gamma)
    = \int_0^{|\gamma|} d_\mu\bLp\gamma(t)\bRp^\exponent \d t
    \ge \int_0^{|\gamma| \wedge d_\mu(x)} \bLp d_\mu(x) - t \bRp^\exponent \d t~.
\end{align}
Let us show that $|\gamma| < \frac12 d_\mu(x)$. Indeed, assuming that $|\gamma| \ge \frac12 d_\mu(x)$ implies that
\begin{align*}
    D_\mu(\gamma)
    \ge \int_0^{\frac12d_\mu(x)} \bLp d_\mu(x) - t \bRp^\exponent \d t
    \ge \frac{d_\mu(x)}{2} \Lp d_\mu(x) - \frac{d_\mu(x)}{2} \Rp^\exponent
    = \Lp\frac{d_\mu(x)}{2}\Rp^{\exponent+1}~,
\end{align*}
hence with \cref{eq:proof:euclidean_bound_local1},
\begin{align*}
    d_\mu(x)^{\exponent+1}
    \le 2^{\exponent+1} D_\mu(\gamma)
    \le 2^{\exponent+2} \|x-y\| d_\mu(x)^\exponent~,
\end{align*}
and
\begin{align*}
    d_\mu(x)
    \le 2^{\exponent+2} \|x-y\|~,
\end{align*}
leading to a contradiction.
Now, the fact that $|\gamma| \ge \frac12 d_\mu(x)$ allows to further lower bound \cref{eq:proof:euclidean_bound_local2} as
\begin{align*}
    D_\mu(\gamma)
    \ge \int_0^{|\gamma|} \bLp d_\mu(x) - t \bRp^\exponent \d t
    \ge |\gamma| \Lp\frac{d_\mu(x)}{2}\Rp^\exponent~.
\end{align*}
Together with \cref{eq:proof:euclidean_bound_local1} this finally implies that
\begin{align*}
    |\gamma|
    \le \Lp\frac{2}{d_\mu(x)}\Rp^\exponent \cdot 2 \|x-y\| d_\mu(x)^\exponent
    \le 2^{\exponent+1} \|x-y\|~.
\end{align*}
In either case it always holds that
\begin{align*}
    |\gamma|
    \le (2^{\exponent+2}+1)\lambda \|x-y\|~,
\end{align*}
which concludes the proof.
\end{proof}

\section{Stability with regard to the Domain}
\label{app:stability_domain}

This section is devoted to the proof of \cref{thm:fdtm_stab_domain_intermediate}.
We fix $\K\subset\RR^d$, $\F\subset\K$, and $\mu\in\cM_\K$. Recall that over $\K$, $d_\mu^\exponent$ is upper bounded by $\diam(\K)^\exponent$ (see \crefnopar{eq:dtm_uniform_bound}) and is $\kappa$-Lipschitz where
\begin{align*}
    \kappa
    = \exponent\diam(\K)^{q-1} \Lp\frac{\diam(\K)}{\min d_\mu}\Rp^{(1-q)\vee0}~.
\end{align*}
The above requires $\min d_\mu > 0$ only if $q<1$.

\subsection{Preliminary Bounds}

First, we provide various upper bounds that will be used to compare the \ac{fdtm} length of slightly offset paths.

\begin{lemma}
\label{lemma:fdtm_upper_bound_segment}
Let $x,y\in\K$.
Then
\begin{align*}
    D_\mu([x,y]) \le D_{\mu,\F}(x,y) + \frac{\kappa}{2} \|x-y\|^2~.
\end{align*}
\end{lemma}

\begin{proof}
Let $\gamma\in\Gamma_{\mu,\F}^\star(x,y)$ be a geodesic and $\overline\gamma = [x,y]$ the straight path from $x$ to $y$, both with an arc length parameterization.
First, cut the integral in two halves as
\begin{align*}
    D_{\mu,\F}([x,y])
    = \int_0^{|\overline\gamma|/2} d_\mu(\overline\gamma(t))^\exponent \d t
        +\int_{|\overline\gamma|/2}^{|\overline\gamma|} d_\mu(\overline\gamma(t))^\exponent \d t~.
\end{align*}
Let us compare each half to its equivalent half in $D_{\mu,\F}(x,y) = D_\mu(\gamma)$.
Since $d_\mu^\exponent$ is $\kappa$-Lipschitz over $\K$,
\begin{align*}
    \int_0^{|\overline\gamma|/2} d_\mu(\overline\gamma(t))^\exponent \d t
    \le \int_0^{|\overline\gamma|/2} d_\mu(\gamma(t))^\exponent \d t
        + \kappa \int_0^{|\overline\gamma|/2} \|\overline\gamma(t) - \gamma(t)\| \d t~.
\end{align*}
The first term is less than the first half of the integral $D_\mu(\gamma)$ as $|\overline\gamma| = \|x-y\| \le |\gamma|$.
As for the second term, both $\overline\gamma$ and $\gamma$ are $1$-Lipschitz and coincide at $0$ so that $\|\overline\gamma(t) - \gamma(t)\| \le 2t$. Then
\begin{align*}
    \int_0^{|\overline\gamma|/2} d_\mu(\overline\gamma(t))^\exponent \d t
    \le \int_0^{|\gamma|/2} d_\mu(\gamma(t))^\exponent \d t + \frac{\kappa}{4}\|x-y\|^2~.
\end{align*}
With the exact same reasoning, a similar bound holds for the end of the paths as
\begin{align*}
    \int_{|\overline\gamma|/2}^{|\overline\gamma|} d_\mu(\overline\gamma(t))^\exponent \d t
    \le \int_{|\gamma|/2}^{|\gamma|} d_\mu(\gamma(t))^\exponent \d t + \frac{\kappa}{4}\|x-y\|^2~.
\end{align*}
Finally, summing both inequalities concludes that
\begin{align*}
    D_{\mu,\F}([x,y])
    \le D_{\mu,\F}(x,y) + \frac{\kappa}{2}\|x-y\|^2~.
\end{align*}
\end{proof}

\begin{lemma}
\label{lemma:fdtm_upper_bounds}
Let $x,y,x',y'\in\K$ such that $\|x-x'\| \le \epsilon$ and $\|y-y'\| \le \epsilon$.
Then
\begin{align}
    \label{eq:lemma:fdtm_upper_bounds1}
    D_\mu([x',y'])
    &\le D_\mu([x,y]) + \kappa\|x-y\|\epsilon + 2\diam(\K)^\exponent \epsilon~.
\end{align}
Moreover, if $x,y\in \F$,
\begin{align}
    \label{eq:lemma:fdtm_upper_bounds2}
    D_\mu([x',y'])
    &\le D_{\mu,\F}(x,y) + \frac{\kappa}{2} \bLp\|x-y\|+2\epsilon\bRp^2 + 2\diam(\K)^\exponent \epsilon~.
\end{align}
\end{lemma}

\begin{proof}
First, by triangular inequality,
\begin{align}
    \label{eq:proof:fdtm_upper_bounds1}
    \|x'-y'\|
    \le \|x'-x\| + \|x-y\| + \|y-y'\|
    \le \|x-y\| + 2\epsilon
\end{align}
and
\begin{align}
    \label{eq:proof:fdtm_upper_bounds2}
    \Lvv \bLp(1-t)x+ty\bRp - \bLp(1-t)x'+ty'\bRp \Rvv
    \le (1-t)\|x-x'\| + t\|y-y'\|
    \le \epsilon~.
\end{align}
Then,
\begin{align*}
    D_\mu([x',y'])
    &= \|x'-y'\| \int_0^1 d_\mu\bLp(1-t)x' + ty'\bRp^\exponent \d t\\
    &\le \bLp\|x-y\|+2\epsilon\bRp \int_0^1 d_\mu\bLp(1-t)x' + ty'\bRp^\exponent \d t
        && \textrm{(\crefnopar{eq:proof:fdtm_upper_bounds1})}\\
    &\le \|x-y\| \int_0^1 d_\mu\bLp(1-t)x' + ty'\bRp^\exponent \d t + 2\diam(\K)^\exponent \epsilon
        && \textrm{($d_\mu \le \diam(\K)$ over $\K$)}\\
    &\le \|x-y\| \int_0^1 \Lp d_\mu\bLp(1-t)x + ty\bRp^\exponent + \kappa\epsilon \Rp \d t + 2\diam(\K)^\exponent \epsilon
        && \textrm{(\cref{eq:proof:fdtm_upper_bounds2} and Lipschitz property)}\\
    &= D_\mu([x,y]) + \kappa\|x-y\|\epsilon + 2\diam(\K)^\exponent \epsilon~,
\end{align*}
which is \cref{eq:lemma:fdtm_upper_bounds1}.
Alternatively, if $x,y\in \F$, then $[x',x]+[x,y]+[y,x'] \in \Gamma_\F(x',y')$ and
\begin{align*}
    D_\mu([x',y'])
    &\le D_{\mu,\F}(x',y') + \frac{\kappa}{2}\|x'-y'\|^2
        && \textrm{(\cref{lemma:fdtm_upper_bound_segment})}\\
    &\le D_{\mu,\F}(x,y) + \frac{\kappa}{2}\|x'-y'\|^2
    + D_{\mu,\F}(x',x) + D_{\mu,\F}(y,y')
        && \textrm{(triangular inequality)}\\
    &\le D_{\mu,\F}(x,y) + \frac{\kappa}{2} \bLp\|x-y\|+2\epsilon\bRp^2 + 2\diam(\K)^\exponent \epsilon~,
        && \textrm{(\cref{eq:proof:fdtm_upper_bounds1} and $d_\mu \le \diam(\K)$)}
\end{align*}
which is \cref{eq:lemma:fdtm_upper_bounds2}.
\end{proof}

\subsection{Main Proof}

In order to establish the stability result we show that for all $r>0$ any path can be decomposed into sections of length close to $r$, excluding long outer sections that cannot be broken down, and with the exception of a possible small section between two long outer sections.
This decomposition will be applied to a geodesic $\gamma\in\Gamma_{\mu,\F}^\star(x,y)$ to construct a path in $\Gamma_\G(x,y)$ that approximates $\gamma$.
From now on, the convention used for path parameterization is to use an arc length parameterization: $\gamma : [0,|\gamma|] \to \RR^d$ with $\|\dot\gamma(t)\|=1$ almost everywhere.

\begin{lemma}[Path Decomposition]
\label{lemma:geodesic_decomposition}
Let $\gamma\in\Gamma_\F(x,y)$.
For all $r>0$ one can construct a sequence of parameters $0=t_0,\dots,t_N=|\gamma|$ and associated points $x_i=\gamma(t_i)$ such that
\begin{itemize}
    \item For all $i\in\{0\dots,N\}$, $x_i \in \F\cup\{x,y\}$.
    \item For all $i\in\{0\dots,N-1\}$, the section $[t_i,t_{i+1}]$ satisfies one of the following:
    \begin{enumerate}[(a)]
        \item $t_{i+1}-t_i \ge r$ and $[t_i,t_{i+1}]$ is an outer section of $\gamma$ \ac{wrt} the domain $\F$.
        \item $\frac{1}{2}r \le t_{i+1}-t_i \le 2r$.
        \item $t_{i+1}-t_i \le \frac{1}{2}r$ and nearby sections $[t_{i-1},t_i]$ and $[t_{i+1},t_{i+2}]$ (if defined) both satisfy~(a).
    \end{enumerate}
\end{itemize}
Furthermore, denote $\{0\dots,N-1\} = A\sqcup B\sqcup C$ with $A$ (resp. $B$ and $C$) a set of indices $i$ such that $[t_i,t_{i+1}]$ satisfies (a) (resp. (b) and (c)).
Assuming that $\gamma$ is not composed of a single section of type (c), it holds that
\begin{align}
    \label{eq:lemma:geodesic_decomposition:length}
    |\gamma|
    \ge \frac14\Lp 2\sum_{i\in A} (t_{i+1}-t_i) + 2|B|r + |C|r \Rp~.
\end{align}
\end{lemma}

Note that the case of a single section of type (c) is excluded from \cref{eq:lemma:geodesic_decomposition:length} for technical reasons and correspond to the case where the length of the path is negligible compared to the parameter $r$, which will later be chosen proportional to $\sqrt{\max_{x\in \F} d(x,\G)}$.
\cref{thm:fdtm_stab_domain_intermediate} is easy to derive in this case regardless, which will be done at the end of the section.

\begin{proof}
First, identify all the outer sections $[s,t]$ of $\gamma$ such that $(t-s) \ge r$. These sections all satisfy (a), and to conclude it suffices to show that the remaining connected parts of $\gamma$ can each be decomposed into sections of type (b) or (c).

Consider such section $[t_0,t_{-1}]$.
By assumption, no outer section $[s,t] \subset [t_0,t_{-1}]$ has length $t-s \ge r$. In other words, if $t-s \ge r$, then $\gamma([s,t])$ intersects $\F$.
If $|\gamma| \le \frac12r$, then the whole section satisfies (c) right away and no decomposition is needed.
Else, we construct the decomposition iteratively. Assuming $t_0,\dots,t_i$ are constructed and $t_{-1}-t_i \ge \frac{1}{2}r$---which is true for $i=0$---we do the following:
\begin{itemize}
    \item If $t_{-1}-t_i \le 2r$, set $t_{i+1} = t_{-1}$. Then $[t_i,t_{i+1}]$ satisfies (b), which concludes the construction.
    \item Else, set $t_{i+1} = \inf\{t\ge t_i+\frac12r : 
    \gamma(t)\in \F\}$. Then:
    \begin{itemize}
        \item $t_{i+1}-t_i \le \frac32r$.
        Indeed, else $[t_i+\frac12r, t_i+\frac32r]$ would be an outer section of length at least $r$, leading to a contradiction.
        Therefore, $[t_i,t_{i+1}]$ satisfies (b).
        \item We have
        \[t_{-1}-t_{i+1} = t_{-1}-t_i - (t_i-t_{i+1}) \ge 2r - \frac32r = \frac12r~,\]
        which allows to continue the construction.
    \end{itemize}
\end{itemize}
Then each section of this decomposition of $[t_0,t_{-1}]$ satisfy (b), which concludes the construction.

We now proceed to prove \cref{eq:lemma:geodesic_decomposition:length}.
By definition of the decomposition, after a type (c) section comes a type (a) section or the end of the path, hence $|C|\le |A|+1$, which in turns implies $|C| \le 2|A|$ provided that $|C|=0$ or $|A|\ge1$, which is true as long as the path is not composed of a unique type (c) section.
In this case,
\begin{align*}
    |\gamma|
    &= \sum_{i\in A\sqcup B\sqcup C} (t_{i+1}-t_i)\\
    &\ge \sum_{i\in A} (t_{i+1}-t_i) + \frac12|B|r
        &&\textrm{($t_{i+1}-t_i \ge r/2$ if $i\in B$)}\\
    &\ge \frac12\sum_{i\in A} (t_{i+1}-t_i) + \frac12|A|r + \frac12|B|r
        &&\textrm{($t_{i+1}-t_i \ge r$ if $i\in A$)}\\
    &\ge \frac14\Lp 2\sum_{i\in A} (t_{i+1}-t_i) + 2|B|r + |C|r \Rp~,
        &&\textrm{($|A| \ge |C|/2$)}
\end{align*}
which concludes the proof.
\end{proof}

Given a geodesic $\gamma$, \cref{thm:fdtm_stab_domain_intermediate} is proven by constructing a path $\overline\gamma$ admissible for the domain $\G$ as close as possible to $\gamma$.
This is done by decomposing $\gamma$ as in \cref{lemma:geodesic_decomposition} and forming the polygonal path going through the closest neighbors in $\G$ of each intermediate step.
Upper bounds from \cref{lemma:fdtm_upper_bounds} are then used to compare the \ac{fdtm} distance between each section of $\gamma$ and the corresponding segment in $\overline\gamma$.

\begin{proof}[Proof of \cref{thm:fdtm_stab_domain_intermediate}]
Let $\gamma\in\Gamma_{\mu,\F}^\star(x,y)$ be a geodesic and denote
\begin{align*}
    \epsilon = \max_{x\in \F} d(x,\G) \le \frac{\diam(\K)^\exponent}{25\kappa}
    \qquad \textrm{and} \qquad
    r = \sqrt{\frac{\diam(\K)^\exponent}{\kappa}\epsilon}~.
\end{align*}
In particular, $\epsilon \le \frac{r}{5}$.
Consider the decomposition of $\gamma$ given by \cref{lemma:geodesic_decomposition} with such $r$ and assume that we are not in the case of a single type (c) section.
Let $x'_0=x$, $x'_N=y$, and for each $1\le i\le N-1$, $x'_i\in \G$ such that $\|x_i-x'_i\| \le \epsilon$.
The polygonal path $\overline\gamma = [x'_0, x'_1, \dots, x'_N]$ belongs to $\Gamma_\G(x,y)$ as it is made of segments between points of $\G\cup\{x,y\}$.
Then, upper bounding the difference between $D_\mu([x'_i,x'_{i+1}])$ and $D_{\mu,\F}(x_i,x_{i+1})$ and summing over $i$ allows to obtain the desired result.

The following provides bounds for all three cases. They are obtained by using \cref{lemma:fdtm_upper_bounds}, the bounds on $\epsilon$, and by upper bounding complex numerical fractions by ``cleaner'' constants.
If $i\in A$, the geodesic $\gamma$ draws a straight line between $x_i$ and $x_{i+1}$, which implies that this line itself is a geodesic and that $\|x_i-x_{i+1}\| = t_{i+1}-t_i \ge r$.
Then, the first inequality in \cref{lemma:fdtm_upper_bounds} yields
\begin{align}
    D_\mu([x'_i,x'_{i+1}]) - D_{\mu,\F}(x_i,x_{i+1})
    &\le \kappa(t_{i+1}-t_i)\epsilon + 2\diam(\K)^\exponent\epsilon\\
    &\le \kappa(t_{i+1}-t_i)\sqrt{\epsilon}\sqrt{\frac{\diam(\K)^q}{25\kappa}} + 2\sqrt{\diam(\K)^\exponent\epsilon} \sqrt{\kappa}r
\end{align}
and eventually
\begin{align}
    \label{eq:proof:stab_domainA}
    D_\mu([x'_i,x'_{i+1}]) - D_{\mu,\F}(x_i,x_{i+1})
    \le 5\sqrt{\kappa\diam(\K)^\exponent\epsilon} (t_{i+1}-t_i)~.
\end{align}
If $i\in B$, $\|x_i-x_{i+1}\|\le 2r$ and the second inequality in \cref{lemma:fdtm_upper_bounds} eventually yields
\begin{align}
    \label{eq:proof:stab_domainB}
    D_\mu([x'_i,x'_{i+1}]) - D_{\mu,\F}(x_i,x_{i+1})
    \le \frac{\kappa}{2} (2r+2\epsilon)^2 + 2\diam(\K)^\exponent\epsilon
    \le 5\sqrt{\kappa\diam(\K)^\exponent\epsilon}r~.
\end{align}
If $i\in C$, $\|x_i-x_{i+1}\|\le \frac{r}{2}$ and the second inequality in \cref{lemma:fdtm_upper_bounds} eventually yields
\begin{align}
    \label{eq:proof:stab_domainC}
    D_\mu([x'_i,x'_{i+1}]) - D_{\mu,\F}(x_i,x_{i+1})
    \le \frac{\kappa}{2} \Lp\frac{r}{2}+2\epsilon\Rp^2 + 2\diam(\K)^\exponent\epsilon
    \le \frac52\sqrt{\kappa\diam(\K)^\exponent\epsilon}r~.
\end{align}
Upper bounding the \ac{fdtm} on the path $\overline{\gamma}$ according to \crefrange{eq:proof:stab_domainA}{eq:proof:stab_domainC} then yields
\begin{align*}
    D_{\mu,\G}(x,y)
    &\le D_\mu(\overline\gamma)\\
    &= \sum_{i\in A\sqcup B\sqcup C} D_\mu([x'_i,x'_{i+1}])\\
    &\le \sum_{i\in A\sqcup B\sqcup C} D_\mu(x_i,x_{i+1})
        + \frac52\sqrt{\kappa\diam(\K)^\exponent\epsilon}\Lp 2\sum_{i\in A} (t_{i+1}-t_i) + 2|B|r + |C|r \Rp\\
    &\le D_\mu(\gamma) + \frac52\sqrt{\kappa\diam(\K)^\exponent\epsilon}\cdot 4|\gamma|
        &&\textrm{(\crefnopar{eq:lemma:geodesic_decomposition:length})}\\
    &= D_{\mu,\F}(x,y) + 10\sqrt{\kappa\diam(\K)^\exponent\epsilon}|\gamma|~,
\end{align*}
which concludes the proof by upper bounding $|\gamma|$ according to \cref{corol:euclidean_bound_uniform} and replacing $\kappa$ with the appropriate quantity depending on $\exponent$.
It now remains to treat the case of a single type (c) section in the decomposition of $\gamma$.
In this case,
\begin{align*}
    \|x-y\|
    \le \frac{r}{2}
    = \frac12\sqrt{\frac{\diam(\K)^\exponent}{\kappa}\epsilon}
\end{align*}
along with $\|x-y\| \le |\gamma|$, hence according to \cref{lemma:fdtm_upper_bound_segment},
\begin{align*}
    D_{\mu,\G}(x,y) - D_{\mu,\F}(x,y)
    \le \frac{\kappa}{2} \|x-y\|^2
    \le \frac{\kappa}{4} \sqrt{\frac{\diam(\K)^\exponent}{\kappa}\epsilon} |\gamma|
    \le 10\sqrt{\kappa\diam(\K)^\exponent\epsilon} |\gamma|~.
\end{align*}
\end{proof}

\section{Convergence of the Empirical FDTM}
\label{app:estimation}

This section is devoted to proving \cref{thm:fdtm_estimate} which we recall below.
\fdtmEstimate*

This result derives from convergence results regarding $\bv d_\mu^p-d_{\hat\mu_n}^p \bv$ and $\hausdorff(\supp(\mu),\XX_n)$.
The following briefly detail these results and how they imply \cref{thm:fdtm_estimate} using \cref{thm:fdtm_stab_dtm,thm:intro:fdtm_stab_domain}. 
First, we state the convergence of the empirical \ac{dtm}, based on~\cite[Theorem 3]{chazalRatesConvergenceRobust2016}.

\begin{proposition}
\label{prop:approx_dtm}
Assume that $\mu$ is $(a,b)$-standard, $\supp(\mu)$ is connected, and that $\frac1n \le m \le \frac12$.
Then there exists a numerical constant $c>0$ such that
\begin{align*}
    \EE\Lb\bvv d_\mu^p-d_{\hat\mu_n}^p \bvv_{\infty,\K}\Rb
    \le c~d~p\diam(\K)^{p-1} \Lp\frac{m}{a}\Rp^{\frac{1}{b}} \frac{\log(n)}{\sqrt{nm}}~.
\end{align*}
\end{proposition}

\begin{proof}
Assume that $m = \frac{k}{n}$ for some integer $k \ge 1$.
Under our regularity assumptions, the modulus of continuity of $u \mapsto \delta_{\mu,u}(x)^p$ is upper bounded according to~\cite[Lemma 3]{chazalRatesConvergenceRobust2016} by the function
\begin{align*}
    \omega : v\in[0,1] \mapsto p\diam(\K)^{p-1}\Lp\frac{v}{a}\Rp^{\frac{1}{b}}~.
\end{align*}
Now, denoting $\log^+ : x \mapsto \log(x) \vee 1$, ~\cite[Theorem 3]{chazalRatesConvergenceRobust2016} implies that
\begin{align*}
    \EE\Lb\bvv d_\mu^p-d_{\hat\mu_n}^p \bvv_{\infty,\K}\Rb
    &\lesssim \frac{1}{\sqrt{nm}} \Lp
        (d+5) \omega(m) \log^+\Lp \frac{nm}{\omega(m)^2} \Rp
        + (d-1) \omega\Lp\sqrt{\frac{m}{n}}\Rp \log^+\Lp \frac{\sqrt{nm}}{\omega\Lp\sqrt{\frac{m}{n}}\Rp} \Rp
    \Rp\\
    &\lesssim d~\omega(m) \frac{\log(n)}{\sqrt{nm}} \Lp 1 + n^{-\frac{1}{2b}} \Rp\\
    &\lesssim d~p\diam(\K)^{p-1} \Lp\frac{m}{a}\Rp^{\frac{1}{b}} \frac{\log(n)}{\sqrt{nm}}~,
\end{align*}
where $\lesssim$ hides a multiplicative numerical constant.
The general statement for any $m\in(0,1)$ is interpolated by applying the above with the integers $k$ and $k+1$ such that $\frac{k}{n} \le m \le \frac{k+1}{n}$.
\end{proof}

The standard assumption implies that for any point $x$ in the support, $\mu\bLp\cB(x,(an)^{-1/b})\bRp \ge \frac1n$. Hence, this ball is expected to contain at least $1$ point in~$\XX_n$. This is formalized by the following statement.
\begin{proposition}[{\cite[Appendix A.2]{taupinEstimationConformalMetrics2026}}]
\label{prop:approx_support_unif}
Assume that $\mu$ is $(a,b)$-standard.
Then for all $n\ge1$,
\begin{align*}
    \EE\bLb \hausdorff(\supp(\mu),\XX_n) \bRb
    \le 2\sqrt2 \Lp\frac{\log(n)}{an}\Rp^{\frac{1}{b}}~.
\end{align*}
\end{proposition}
We now move on to the proof of \cref{thm:fdtm_estimate}.

\begin{proof}[Proof of \cref{thm:fdtm_estimate}]
In order to use \cref{prop:approx_dtm}, assume that $m\le\frac12$ and $n\ge\frac1m$.
Decompose the offset into the offset \ac{wrt} the \ac{dtm} and the offset \ac{wrt} the domain as
\begin{align}
    \label{eq:proof:fdtm_estimate1}
    \bvv D_\mu - D_{\hat\mu_n} \bvv_{\infty,\K}
    \le \bvv D_{\mu,\supp(\mu)} - D_{\mu,\XX_n} \bvv_{\infty,\K}
        + \bvv D_{\mu,\XX_n} - D_{\hat\mu_n,\XX_n} \bvv_{\infty,\K}~.
\end{align}
Regarding the first term,
\begin{align}
    \label{eq:proof:fdtm_estimate2}
    \EE\Lb \bvv D_{\mu,\supp(\mu)} - D_{\mu,\XX_n} \bvv_{\infty,\K} \Rb
    &\lesssim \diam(\K)^{\exponent+\frac12} \Lp\frac{\diam(\K)}{\sigma}\Rp^{\frac{1-\exponent}{2}\vee0} \EE\Lb\sqrt{\hausdorff(\supp(\mu),\XX_n)}\Rb
        &&\textrm{(\cref{thm:fdtm_stab_domain_intermediate})}\nonumber\\
    &\le \diam(\K)^{\exponent+\frac12} \Lp\frac{\diam(\K)}{\sigma}\Rp^{\frac{1-\exponent}{2}\vee0} \sqrt{\EE\bLb\hausdorff(\supp(\mu),\XX_n)\bRb}
        &&\textrm{(Jensen's inequality)}\nonumber\\
    &\lesssim \diam(\K)^{\exponent+\frac12} \Lp\frac{\diam(\K)}{\sigma}\Rp^{\frac{1-\exponent}{2}\vee0} \Lp\frac{\log(n)}{an}\Rp^{\frac{1}{2b}}~,
        &&\textrm{(\cref{prop:approx_support_unif})}
\end{align}
where $\lesssim$ hides a multiplicative constant depending on $m$ and $\exponent$.
As for the second term in \cref{eq:proof:fdtm_estimate1}, recall that $\sigma \le d_\mu \le \diam(\K)$ and $d_{\hat\mu_n} \le \diam(\K)$ \ac{as} over $\K$ by assumption.
If $\exponent\ge p$ then the Lipschitz property of $t\in[0,\diam(\K)] \mapsto t^{\exponent/p}$ yields
\begin{align*}
    \bv d_\mu^\exponent - d_{\hat\mu_n}^\exponent \bv
    \le \frac{\exponent}{p} \diam(\K)^{\exponent-p} \bv d_\mu^p - d_{\hat\mu_n}^p \bv
    \qquad\textrm{\ac{as}}
\end{align*}
If $\exponent\le p$, a similar inequality can be achieved even though only $d_\mu$ is lower bounded. For all $t>0$ and $s\ge0$, $t^{\exponent-p} s^p \le s^\exponent \le t^\exponent = t^{\exponent-p} t^p$ if $s\le t$, and $t^{\exponent-p} s^p \ge s^\exponent \ge t^\exponent = t^{\exponent-p} t^p$ if $s \ge t$, so that either way
\begin{align*}
    \bv t^\exponent - s^\exponent \bv
    \le t^{\exponent-p} \bv t^p - s^p \bv~.
\end{align*}
In particular, with $t=d_\mu$ and $s=d_{\hat\mu_n}$,
\begin{align*}
    \bv d_\mu^\exponent - d_{\hat\mu_n}^\exponent \bv
    \le d_\mu^{\exponent-p} \bv d_\mu^p - d_{\hat\mu_n}^p \bv
    \le \sigma^{\exponent-p} \bv d_\mu^p - d_{\hat\mu_n}^p \bv
    \qquad\textrm{\ac{as}}
\end{align*}
Therefore, whether $\exponent$ is larger or smaller than $p$, it always holds that
\begin{align}
    \label{eq:proof:fdtm_estimateLip}
    \bv d_\mu^\exponent - d_{\hat\mu_n}^\exponent \bv
    \le \Lp\frac{\exponent}{p}\vee1\Rp \diam(\K)^{\exponent-p} \Lp\frac{\diam(\K)}{\sigma}\Rp^{(p-\exponent)\vee0} \bv d_\mu^p - d_{\hat\mu_n}^p \bv
    \qquad\textrm{\ac{as}}
\end{align}
The second term in \cref{eq:proof:fdtm_estimate1} is thus bounded as follows:
\begin{align}
    \label{eq:proof:fdtm_estimate3}
    \EE\Lb \bvv D_{\mu,\XX_n} - D_{\hat\mu_n,\XX_n} \bvv_{\infty,\K} \Rb
    &\lesssim \diam(\K)~\EE\Lb \bvv d_\mu^\exponent - d_{\hat\mu_n}^\exponent \bvv_{\infty,\K} \Rb
        &&\textrm{(\cref{thm:fdtm_stab_dtm})}\nonumber\\
    &\lesssim \diam(\K)^{q-p+1} \Lp\frac{\diam(\K)}{\sigma}\Rp^{(p-\exponent)\vee0} \EE\Lb \bvv d_\mu^p - d_{\hat\mu_n}^p \bvv_{\infty,\K} \Rb
        &&\textrm{(\crefnopar{eq:proof:fdtm_estimateLip}})\nonumber\\
    &\lesssim d~\diam(\K)^q \Lp\frac{\diam(\K)}{\sigma}\Rp^{(p-\exponent)\vee0} \Lp\frac{m}{a}\Rp^{\frac{1}{b}} \frac{\log(n)}{\sqrt{nm}}~,
        &&\textrm{(\cref{prop:approx_dtm})}
\end{align}
where $\lesssim$ hides a multiplicative constant depending on $m$, $p$, and $\exponent$.
Putting together \cref{eq:proof:fdtm_estimate1,eq:proof:fdtm_estimate2,eq:proof:fdtm_estimate3} concludes the proof of \cref{thm:fdtm_estimate}.
\end{proof}

\section{Link between DTM and density}
\label{sec:dtm_density}

This section is devoted to proving \cref{thm:dtm_density} which we recall below.
\dtmDensity*

\begin{proof}
Fix some $x \in \RR^d$ and denote
\begin{align*}
    \epsilon
    = \Lp\frac{d+1}{\omega_d} L^d m\Rp^{\frac{1}{d+1}}~.
\end{align*}

Let us show that $\bv\hat f_m(x) - f(x)\bv \le 2\epsilon$.
By using the co-area formula to integrate the density over spheres $\partial\cB(x,t)$, for all $r>0$ it holds that
\begin{align*}
    \mu(\cB(x,r))
    = \int_0^r \int_{\partial\cB(x,t)} f(y)~\d y~\d t~.
\end{align*}
Since $f$ takes values in $[f(x) - Lt, f(x) + Lt]$ over $\partial\cB(x,t)$ and $\partial\cB(x,t)$ has measure $d\omega_d t^{d-1}$, one has
\begin{align}
    \label{eq:proof:dtm_density1}
    \mu(\cB(x,r))
    \ge d\omega_d \int_0^r t^{d-1} (f(x) - Lt) \d t
    \ge \omega_d r^d f(x) - \frac{d}{d+1}\omega_d r^{d+1} L~,
\end{align}
and a similar upper bound holds. Further upper bounding $\frac{d}{d+1} \le 1$ yields
\begin{align}
    \label{eq:proof:dtm_density2}
    \omega_d r^d (f(x) - Lr)
    \le \mu(\cB(x,r))
    \le \omega_d r^d (f(x) + Lr)~.
\end{align}

\paragraph{First case :} Assume that $f(x) \ge \epsilon$ and for all $u \le m$ denote
\begin{align}
    \label{eq:proof:dtm_density3}
    r_u
    = \Lp\frac{d+1}{\omega_d \epsilon} u\Rp^{\frac1d}
    \le \Lp\frac{d+1}{\omega_d \epsilon} m\Rp^{\frac1d}
    = \Lp\frac{d+1}{\omega_d L} m\Rp^{\frac{1}{d+1}}
    = \frac{\epsilon}{L}~.
\end{align}
Then, from \cref{eq:proof:dtm_density1,eq:proof:dtm_density3} it holds that
\begin{align*}
    \mu(\cB(x,r_u))
    \ge \omega_d r_u^d \Lp\epsilon - \frac{d}{d+1}Lr_u\Rp
    \ge \frac{\omega_d}{d+1} r_u^d \epsilon
    = u~.
\end{align*}
By definition, this implies that $\delta_{\mu,u}(x) \le r_u \le \frac{\epsilon}{L}$. Now, applying \cref{eq:proof:dtm_density2} to $r = \delta_{\mu,u}(x)$, so that ${\mu(\cB(x,r)) = u}$, yields
\begin{align*}
    u
    \le \omega_d \delta_{\mu,u}(x)^d (f(x)+L\delta_{\mu,u}(x))
    \le \omega_d \delta_{\mu,u}(x)^d (f(x)+\epsilon)
\end{align*}
and
\begin{align*}
    u
    \ge \omega_d \delta_{\mu,u}(x)^d (f(x)-L\delta_{\mu,u}(x))
    \ge \omega_d \delta_{\mu,u}(x)^d (f(x)-\epsilon)~,
\end{align*}
hence
\begin{align*}
    \frac{u}{\omega_d (f(x)+\epsilon)}
    \le \delta_{\mu,u}(x)^d
    \le \frac{u}{\omega_d (f(x)-\epsilon)}~.
\end{align*}
It follows that
\begin{align*}
    d_{\mu,m}(x)^p
    \le \frac1m \int_0^m \Lp \frac{u}{\omega_d (f(x)-\epsilon)} \Rp^{\frac{p}{d}} \d u
    = \frac{d}{d+p} \Lp \frac{m}{\omega_d (f(x)-\epsilon)} \Rp^{\frac{p}{d}}~,
\end{align*}
and
\begin{align*}
    d_{\mu,m}(x)^p
    \ge \frac1m \int_0^m \Lp \frac{u}{\omega_d (f(x)+\epsilon)} \Rp^{\frac{p}{d}} \d u
    = \frac{d}{d+p} \Lp \frac{m}{\omega_d (f(x)+\epsilon)} \Rp^{\frac{p}{d}}~,
\end{align*}
so that
\begin{align*}
    |\hat f_m(x) - f(x)|
    = \Lv \frac{m}{\omega_d} \Lp\frac{d}{d+p}\Rp^{\frac{d}{p}} \frac{1}{d_{\mu,m}(x)^d} - f(x)\Rv
    \le \epsilon~,
\end{align*}
which concludes this case.

\paragraph{Second case :} Assume that $f(x) \le \epsilon$ and let $u \le m$.
\Cref{eq:proof:dtm_density1} with $r = \delta_{\mu,u}(x)$ yields
\begin{align*}
    u
    \le \omega_d \delta_{\mu,u}(x)^d \bLp\epsilon + L\delta_{\mu,u}(x)\bRp~.
\end{align*}
If $\delta_{\mu,u}(x) \le \frac{\epsilon}{L}$ then $u \le 2\omega_d \delta_{\mu,u}(x)^d \epsilon$, and
\begin{align*}
    \delta_{\mu,u}(x)^d
    \ge \frac{u}{2\omega_d \epsilon}
    = \Lp\frac{1}{(d+1) \omega_d^d L^d m}\Rp^{\frac{1}{d+1}} \frac{u}{2}~.
\end{align*}
Else, if $\delta_{\mu,u}(x) \ge \frac{\epsilon}{L}$ then $u \le 2L\omega_d \delta_{\mu,u}(x)^{d+1}$, and
\begin{align*}
    \delta_{\mu,u}(x)^{d+1}
    \ge \frac{u}{2\omega_d L}
    \ge \frac{u}{2\omega_d L} \Lp\frac{u}{2(d+1)m}\Rp^{\frac1d}
    = \frac{1}{\omega_d L} \Lp\frac{u^{d+1}}{2^{d+1}(d+1)m}\Rp^{\frac1d}~,
\end{align*}
so that in either case, for all $u\le m$,
\begin{align*}
    \delta_{\mu,u}(x)
    \ge \Lp\frac{u^{d+1}}{2^{d+1}(d+1) \omega_d^d L^d m}\Rp^{\frac{1}{d(d+1)}}~.
\end{align*}
Therefore,
\begin{align*}
    d_{\mu,m}(x)^p
    &= \frac1m \int_0^m \delta_{\mu,u}(x)^p \d u\\
    &\ge \Lp\frac{1}{2^{d+1}(d+1) \omega_d^d L^d m}\Rp^{\frac{p}{d(d+1)}} \frac1m \int_0^m u^{\frac{p}{d}} \d u\\
    &= \Lp\frac{1}{2^{d+1}(d+1) \omega_d^d L^d m}\Rp^{\frac{p}{d(d+1)}} \frac{d}{d+p} m^{\frac{p}{d}}\\
    &= \Lp\frac{m^d}{2^{d+1}(d+1) \omega_d^d L^d}\Rp^{\frac{p}{d(d+1)}} \frac{d}{d+p}~.
\end{align*}
Finally,
\begin{align*}
    \hat f_m(x)
    = \frac{1}{\omega_d} \Lp\frac{d}{d+p}\Rp^{\frac{d}{p}} \frac{m}{d_{\mu,m}(x)^d}
    \le 2(d+1)^\frac{1}{d+1} \frac{m}{\omega_d} \Lp\frac{\omega_d L}{m}\Rp^{\frac{d}{d+1}}
    = 2\epsilon~,
\end{align*}
and
\begin{align*}
    |\hat f_m(x) - f(x)|
    \le \max(\hat f_m(x), f(x))
    \le 2\epsilon~,
\end{align*}
which concludes this case and the proof.
\end{proof}

\end{document}